\documentclass[11pt, reqno]{amsart}

\usepackage{amsmath}

\usepackage{tabu}
\usepackage{tikz}
\usepackage[margin=1.25in]{geometry}
\usepackage[matrix,arrow,curve,frame]{xy}
 \usepackage{hyperref}
\usepackage{amsthm}
\usepackage{amsfonts}
\usepackage{amssymb}
\usepackage{wasysym}
\usepackage{mathrsfs}
\usepackage{mathtools}

\definecolor{my-linkcolor}{rgb}{0.75,0,0}
\definecolor{my-citecolor}{rgb}{0.1,0.57,0}
\definecolor{my-urlcolor}{rgb}{0,0,0.75}
\hypersetup{
	colorlinks,
	linkcolor={my-linkcolor},
	citecolor={my-citecolor},
	urlcolor={my-urlcolor}
}

\title[Singlet vertex algebras]{
Ribbon tensor structure on the full representation categories of the singlet vertex algebras
}
 \author{Thomas Creutzig, Robert McRae and Jinwei Yang}
\date{}

\address{(T. C.) Department of Mathematical and Statistical Sciences, University of Alberta, Edmonton, Alberta T6G 2G1, Canada}

 \email{creutzig@ualberta.ca}

 \address{(R. M.) Yau Mathematical Sciences Center, Tsinghua University, Beijing 100084, China}
  \email{rhmcrae@tsinghua.edu.cn}
  
  \address{(J. Y.) School of Mathematical Sciences, Shanghai Jiaotong University, Shanghai 200240, China}
  \email{jinwei2@sjtu.edu.cn}


\newtheorem{thm}{Theorem}[section]
\newtheorem{cor}[thm]{Corollary}
\newtheorem{lem}[thm]{Lemma}
\newtheorem{prop}[thm]{Proposition}
\theoremstyle{definition}\newtheorem{defi}[thm]{Definition}
\theoremstyle{definition}\newtheorem{rem}[thm]{Remark}
\theoremstyle{definition}
\theoremstyle{definition}

\newcommand{\cE}{\mathcal{E}}
\newcommand{\cH}{\mathcal{H}}
\newcommand{\cY}{\mathcal{Y}}

\newcommand{\cV}{\mathcal{V}}
\newcommand{\cA}{\mathcal{A}}
\newcommand{\cR}{\mathcal{R}}
\newcommand{\cM}{\mathcal{M}}
\newcommand{\cS}{\mathcal{S}}

\newcommand{\cF}{\mathcal{F}}

\newcommand{\cL}{\mathcal{L}}
\newcommand{\cO}{\mathcal{O}}
\newcommand{\cP}{\mathcal{P}}
\newcommand{\cC}{\mathcal{C}}
\newcommand{\cG}{\mathcal{G}}
\newcommand{\cW}{\mathcal{W}}

\newcommand{\til}{\widetilde}

\newcommand{\CC}{\mathbb{C}}
\newcommand{\ZZ}{\mathbb{Z}}
\newcommand{\NN}{\mathbb{N}}

\newcommand{\QQ}{\mathbb{Q}}

\newcommand{\Id}{\mathrm{Id}}

\newcommand{\tens}{\boxtimes}
\newcommand{\vac}{\mathbf{1}}
\newcommand{\ind}{\mathrm{Ind}}
 \DeclareMathOperator{\im}{Im}

 \DeclareMathOperator{\rep}{Rep}
 \let\ker\relax
 \let\hom\relax
 \DeclareMathOperator{\ker}{Ker}
 \DeclareMathOperator{\hom}{Hom}
 \DeclareMathOperator{\Endo}{End}

\begin{document}
\bibliographystyle{alpha}

\numberwithin{equation}{section}

 \begin{abstract}
 We show that the category of finite-length generalized modules for the singlet vertex algebra $\mathcal{M}(p)$, $p\in\mathbb{Z}_{>1}$, is equal to the category $\mathcal{O}_{\mathcal{M}(p)}$ of $C_1$-cofinite $\mathcal{M}(p)$-modules, and that this category admits the vertex algebraic braided tensor category structure of Huang-Lepowsky-Zhang. Since $\mathcal{O}_{\mathcal{M}(p)}$ includes the uncountably many typical $\mathcal{M}(p)$-modules, which are simple $\mathcal{M}(p)$-module structures on Heisenberg Fock modules, our results substantially extend our previous work on tensor categories of atypical $\mathcal{M}(p)$-modules. We also introduce a tensor subcategory $\mathcal{O}_{\mathcal{M}(p)}^T$, graded by an algebraic torus $T$, which has enough projectives and is conjecturally tensor equivalent to the category of finite-dimensional weight modules for the unrolled restricted quantum group of $\mathfrak{sl}_2$ at a $2p$th root of unity. We compute all tensor products involving simple and projective $\mathcal{M}(p)$-modules, and we prove that both tensor categories $\mathcal{O}_{\mathcal{M}(p)}$ and $\mathcal{O}_{\mathcal{M}(p)}^T$ are rigid and thus also ribbon. As an application, we use vertex operator algebra extension theory to show that the representation categories of all finite cyclic orbifolds of the triplet vertex algebras $\mathcal{W}(p)$ are non-semisimple modular tensor categories, and we confirm a conjecture of Adamovi\'{c}-Lin-Milas on the classification of simple modules for these finite cyclic orbifolds.
\end{abstract}

\maketitle

\tableofcontents

\section{Introduction}

Besides being a rich subject in its own right, the representation theory of vertex (operator) algebras has applications to a variety of branches of mathematics and physics. Here we study the singlet vertex algebras $\cM(p)$, $p \in \ZZ_{>1}$, using and substantially extending our previous partial results in \cite{CMY-singlet}. These algebras were among the first examples of chiral algebras of logarithmic conformal field theory, and their representation theory is currently used to obtain new invariants of three-manifolds and three-dimensional topological and even quantum field theories. The applications to low-dimensional topology require ribbon tensor categories, so we start by sketching the state of the art of such categories associated to vertex operator algebras. We will describe the applications of our results in more detail at the end of the introduction. 

\subsection{Rigid vertex tensor categories}

Representation categories of general classes of vertex operator algebras are expected to admit natural braided ribbon (and in particular rigid) tensor category structure. This is a celebrated theorem of Huang for the class of rational $C_2$-cofinite vertex operator algebras, in which case the representation categories are semisimple modular tensor categories \cite{Hu-mod}. However, vertex operator algebras are rarely rational or $C_2$-cofinite, and their representation categories are rarely semisimple or finite. 

The first non-rational examples appeared three decades ago in physics in the context of low-dimensional topology and logarithmic conformal field theory, namely the WZW theory of the Lie superalgebra $\mathfrak{gl}_{1|1}$ \cite{Rozansky:1992rx}  and the singlet algebras $\cM(p)$ \cite{Kausch:1990vg}. By now, it is understood that affine vertex algebras (and their $W$-algebras) at almost all levels admit uncountably many inequivalent simple modules \cite{KR} and should also admit logarithmic modules (which are indecomposable but reducible modules on which the Virasoro zero-mode $L(0)$ acts non-semisimply).
 Even the construction of logarithmic modules is a difficult task: for affine vertex operator (super)algebras, this is achieved only for those associated to $\mathfrak{sl}_2, \mathfrak{osp}_{1|2}, \mathfrak{sl}_3$, and $\mathfrak{gl}_{1|1}$ \cite{Ad-sl, ACG, CMY3}. Among these, the complete ribbon (super)category of modules is currently understood only in the case of $\mathfrak{gl}_{1|1}$ \cite{CMY3}, and the only additional completely-understood example of a non-finite, non-semisimple tensor category of representations for a vertex operator algebra is the $\beta\gamma$-system \cite{AW}.

By `understanding' a tensor category, we mean finding both its abelian and monoidal structures. Understanding the abelian structure especially includes classifying simple and projective modules and determining the structure of all projective modules, for example their Loewy diagrams. Understanding the monoidal structure means first establishing its existence (for representation categories of vertex operator algebras, this is the vertex tensor category structure of \cite{HLZ1}-\cite{HLZ8}). Then we want to compute fusion rules, or more precisely, prove formulas for the tensor products of at least the simple and projective objects. The final goal is to prove rigidity; once this is done for a vertex algebraic tensor category, we immediately get ribbon category structure since we always have a natural ribbon twist. The main result of the present work is an understanding in this sense of the category of $C_1$-cofinite grading-restricted generalized $\cM(p)$-modules.

\subsection{Ribbon categories of atypical singlet modules}

The singlet algebra $\cM(p)$, $p\in\ZZ_{>1}$, is a subalgebra of the rank-one Heisenberg vertex algebra $\cH$, with a modified conformal vector giving it central charge $1-6\frac{(p-1)^2}{p}$. Irreducible ($\NN$-gradable) $\cM(p)$-modules were classified by Adamovi\'{c} \cite{Ad}: every Fock $\cH$-module $\cF_\lambda$, where $\lambda\in\CC$ is the eigenvalue of the Heisenberg zero-mode $h(0)$, restricts to an $\cM(p)$-module, and simple $\cM(p)$-modules are in one-to-one correspondence with Fock modules. Generically, $\cF_\lambda$ remains irreducible as an $\cM(p)$-module, with countably many exceptions. More precisely, for $r\in\ZZ$ and $s\in\lbrace 1,2,\ldots p\rbrace$, introduce
\begin{equation*}
 \alpha_{r,s}=\frac{1-r}{2}\alpha_++\frac{1-s}{2}\alpha_-, \qquad \alpha_+ = \sqrt{2p}, \qquad \alpha_- = - \sqrt{2/p},
\end{equation*}
as well as the lattice $L = \mathbb Z \alpha_+$ whose dual is 
$L^\circ=\ZZ\frac{\alpha_-}{2}$. So $L^\circ$ consists of all $\alpha_{r,s}$ for $r\in\ZZ$ and $1\leq s\leq p$. Then the Fock module $\cF_\lambda$ is simple as an $\cM(p)$-module for $\lambda \in \CC \setminus L^\circ$, while the simple $\cM(p)$-module corresponding to $\cF_{\alpha_{r, s}}$ is its socle, which we denote $\cM_{r, s}$. For $s=p$, $\cM_{r,p}$ is still equal to $\cF_{\alpha_{r, p}}$, but for $1 \leq s \leq p-1$, $\cM_{r,s}$ is characterized by the non-split exact sequence 
\begin{equation*}
  0\longrightarrow\cM_{r,s}\longrightarrow\cF_{\alpha_{r,s}}\longrightarrow\cM_{r+1,p-s}\longrightarrow 0.
 \end{equation*}
A simple $\cM(p)$-module is called \textit{typical} if it is a Fock module, and \textit{atypical} otherwise. 

In \cite{CMY-singlet}, we used the existence of tensor structure on the category of $C_1$-cofinite modules for the Virasoro algebra at central charge $1-6\frac{(p-1)^2}{p}$ (proved in \cite{CJORY}) to construct a vertex algebraic tensor category $\cC_{\cM(p)}$ of $\cM(p)$-modules containing all atypical modules. More precisely, $\cC_{\cM(p)}$ consists of all finite-length $\cM(p)$-modules whose composition factors come from the modules $\cM_{r,s}$ for $r\in\ZZ$, $1\leq s\leq p$; this category does not include the typical Fock modules $\cF_\lambda$ for $\lambda\in\CC\setminus L^\circ$, since these are not sums of $C_1$-cofinite Virasoro modules. 

No non-zero $\cM(p)$-module is projective in $\cC_{\cM(p)}$, but there is a tensor subcategory $\cC_{\cM(p)}^0$ that does have enough projectives. This subcategory is most practically defined to consist of objects $M$ having trivial monodromy with $\cM_{3,1}$, that is, the double braiding
\begin{equation*}
\cR_{\cM_{3,1},M}^2: \cM_{3,1}\tens M\longrightarrow M\tens\cM_{3,1}\longrightarrow \cM_{3,1}\tens M
\end{equation*}
is the identity. In \cite{CMY-singlet}, we showed that the typical modules $\cM_{r,p}=\cF_{\alpha_{r,p}}$ are projective in $\cC_{\cM(p)}^0$, while for $1\leq s\leq p-1$, $\cM_{r,s}$ has a length-four projective cover $\cP_{r,s}$ in $\cC_{\cM(p)}^0$ with Loewy diagram 
  \begin{equation*}
 \begin{matrix}
  \begin{tikzpicture}[->,>=latex,scale=1.5]
\node (b1) at (1,0) {$\cM_{r, s}$};
\node (c1) at (-1, 1){$\cP_{r, s}$:};
   \node (a1) at (0,1) {$\cM_{r-1, p-s}$};
   \node (b2) at (2,1) {$\cM_{r+1, p-s}$};
    \node (a2) at (1,2) {$\cM_{r,s}$};
\draw[] (b1) -- node[left] {} (a1);
   \draw[] (b1) -- node[left] {} (b2);
    \draw[] (a1) -- node[left] {} (a2);
    \draw[] (b2) -- node[left] {} (a2);
\end{tikzpicture}
\end{matrix} .
 \end{equation*}
We also showed that both tensor categories $\cC_{\cM(p)}$ and $\cC_{\cM(p)}^0$ are rigid, and we computed all tensor products involving the modules $\cM_{r,s}$ and $\cP_{r,s}$ for $r\in\ZZ$ and $1\leq s\leq p$.
 
 Recently in \cite{GN}, Gannon and Negron used our results in \cite{CMY-singlet} to show that $\cC_{\cM(p)}^0$ is ribbon tensor equivalent to a certain category of weight modules (with a suitable braiding and ribbon structure) for the unrolled restricted quantum group of $\mathfrak{sl}_2$ at $q=e^{\pi i/p}$. However, it is conjectured \cite{CGP2, CMR} that the \textit{entire} category of finite-dimensional weight modules for the quantum group is ribbon tensor equivalent to a suitable category of $\cM(p)$-modules. Our results in the present paper, described next, achieve for the first time the braided ribbon tensor structure on the correct category of $\cM(p)$-modules for this conjectural equivalence; $\cC_{\cM(p)}^0$ is then a tensor subcategory of this larger category of $\cM(p)$-modules.

\subsection{Results}

Our first main result, proved in Sections \ref{sec:genVerma} and \ref{subsec:tens_cats} is the existence of vertex tensor category structure on the category of $C_1$-cofinite $\cM(p)$-modules:
\begin{thm}\textup{(Theorem \ref{thm:C1_equals_fl}) }
  The category $\cO_{\cM(p)}$ of $C_1$-cofinite grading-restricted generalized $\cM(p)$-modules equals the category of finite-length grading-restricted generalized $\cM(p)$-modules and admits the vertex algebraic braided tensor category structure of \cite{HLZ1}-\cite{HLZ8}.
\end{thm}

To prove this, we verify the sufficient conditions for tensor category structure from \cite{CY}. First,
 every irreducible $\cM(p)$-module is $C_1$-cofinite by \cite[Theorem 13]{CMR}, so $\cO_{\cM(p)}$ contains the category of all finite-length generalized $\cM(p)$-modules. If these two categories coincide, then \cite[Theorem 3.3.4]{CY} shows that $\cO_{\cM(p)}$ satisfies the assumptions for applying the logarithmic tensor category theory of \cite{HLZ1}-\cite{HLZ8}, and thus $\cO_{\cM(p)}$ is indeed a braided tensor category. Then  by \cite[Theorem 3.3.5]{CY}, 
this equality of categories holds if the generalized Verma $\cM(p)$-module (using terminology from \cite{Li-intw-ops}) induced from any finite-dimensional irreducible module for the Zhu algebra of $\cM(p)$ has finite length. This we prove in Section \ref{sec:genVerma} by determining all generalized Verma $\cM(p)$-modules explicitly. 

In Theorem \ref{thm:typical_GVM}, we show that the typical irreducible $\cM(p)$-module $\cF_\lambda$, $\lambda\in\CC\setminus L^\circ$, is its own generalized Verma $\cM(p)$-module cover, by a Virasoro intertwining operator argument similar to the proof of \cite[Theorem 4.4]{AM-trip}. Finding the generalized Verma $\cM(p)$-module covers of the atypical $\cM(p)$-modules is much more difficult: besides properties of Virasoro intertwining operators, we heavily use our results on the category $\cC_{\cM(p)}^0$ from \cite{CMY-singlet}, especially the existence and projectivity of the modules $\cP_{r,s}$. In Theorem \ref{thm:atypical_GVM}, we show that the generalized Verma $\cM(p)$-module cover of $\cM_{r,s}$ is $\cP_{r,s}/M$ where $M$ is the smallest submodule such that $\cP_{r,s}/M$ has the same lowest conformal weight space as $\cM_{r,s}$. This quotient has length at most three, so all generalized Verma $\cM(p)$-modules have finite length.

Our second main result, in Sections \ref{subsec:tens_cats} and \ref{sec:projective}, is the classification of projective $\cM(p)$-modules. Since Heisenberg Fock modules admit indecomposable self-extensions of arbitary length, which remain indecomposable as $\cM(p)$-modules, $\cO_{\cM(p)}$ does not have any non-zero projective objects. As in \cite{CMY-singlet}, we resolve this problem by introducing a tensor subcategory that does have enough projectives. Specifically, we define $\cO_{\cM(p)}^T$ to be the subcategory of $\cO_{\cM(p)}$ whose objects have semisimple monodromy with $\cM_{2,1}$ (see Definitions \ref{defi:Ot} and \ref{defi:OT}). Here $T$ is the algebraic torus $T = \CC/2L^\circ$: the category $\cO_{\cM(p)}^T$ is $T$-graded with homogeneous subcategories $\cO_{\cM(p)}^t\subseteq\cO_{\cM(p)}^T$ for $t=\beta+2L^\circ\in T$ consisting of all $\cM(p)$-modules $M$ in $\cO_{\cM(p)}$ such that 
\begin{equation*}
\cR^2_{\cM_{2,1},M}=e^{-2\pi i\alpha_{2,1}\beta}\Id_{\cM_{2,1}\tens M}.
\end{equation*}
In Theorem \ref{thm:OT_properties} we show that  $\cO^T_{\cM(p)}$ is a full tensor subcategory of $\cO_{\cM(p)}$ that is closed under submodules and quotients, and in Proposition \ref{prop:irred_mod_grading} we show that $\cO_{\cM(p)}^T$ contains all simple $\cM(p)$-modules. We then describe its abelian structure:
\begin{thm} \label{thm:intro_projective}
A complete list of indecomposable projective objects in $\cO_{\cM(p)}^T$  is:
\begin{enumerate} 
\item   \textup{(Proposition \ref{prop:Prs_proj})}
For $r\in\ZZ$ and $1\leq s\leq p-1$, the indecomposable $\cM(p)$-module $\cP_{r,s}$ is a projective cover of $\cM_{r,s}$. 
\item  \textup{(Theorem \ref{thm:Flambda_proj})}
 For $\lambda\in(\CC\setminus L^\circ)\cup\lbrace\alpha_{r,p}\,\vert\,r\in\ZZ\rbrace$, the irreducible $\cM(p)$-module $\cF_\lambda$ is its own projective cover. 
 \end{enumerate}
\end{thm}

Our third main result is the computation of tensor products involving irreducible and projective $\cM(p)$-modules; all fusion rules (dimensions of spaces of intertwining operators) follow as corollaries. As the atypical category $\cC_{\cM(p)}$ from \cite{CMY-singlet} is a tensor subcategory of $\cO_{\cM(p)}$, all tensor products involving the modules $\cM_{r,s}$ and $\cP_{r,s}$ are already computed in \cite[Theorem 5.2.1]{CMY-singlet}. In Section \ref{sec:fusion}, we use these fusion rules from \cite{CMY-singlet} as well as results on Virasoro intertwining operators and projectivity in $\cO_{\cM(p)}^T$ of the modules $\cP_{r,s}$ and $\cF_\lambda$, $\lambda\in\CC\setminus L^\circ$, to find the remaining tensor products involving typical modules:
\begin{thm}\label{thm:intro_fus_rules}
 The following tensor product formulas hold in $\cO_{\cM(p)}$:
 \begin{enumerate}
  \item \textup{(Theorem \ref{thm:Mrs_Flambda})}
  For $r\in\ZZ$, $1\leq s\leq p$, and $\lambda\in\CC\setminus L^\circ$, 
 \begin{equation*}
  \cM_{r,s}\tens\cF_\lambda\cong\bigoplus_{\ell=0}^{s-1} \cF_{\lambda+\alpha_{r,s}+\ell\alpha_-}.
 \end{equation*}
  
  \item \textup{(Theorem \ref{thm:Prs_Flambda})}
  For $r\in\ZZ$, $1\leq s\leq p-1$, and $\lambda\in\CC\setminus L^\circ$,
 \begin{align*}
  \cP_{r,s}\tens\cF_\lambda 
   \cong\bigoplus_{\ell =0}^{p-1} \left(\cF_{\lambda+\alpha_{r,s}+\ell\alpha_-}\oplus\cF_{\lambda+\alpha_{r-1,p-s}+\ell\alpha_-}\right).
  \end{align*}
  
  \item \textup{(Theorem \ref{thm:typ_typ_atyp_fusion})}
  For $\lambda,\mu\in\CC\setminus L^\circ$ such that $\lambda+\mu =\alpha_++\alpha_-+\alpha_{r,s}\in L^\circ$ for some $r\in\ZZ$, $1\leq s\leq p$,
  \begin{equation*}
  \cF_\lambda\tens\cF_{\mu}\cong\bigoplus_{\substack{s'= s\\ s'\equiv s\,\,(\mathrm{mod}\,2)\\}}^p \cP_{r,s'}\oplus\bigoplus_{\substack{s'=p+2-s\\s'\equiv p-s\,\,(\mathrm{mod}\,2)\\}}^p \cP_{r-1,s'}.
 \end{equation*}

  \item \textup{(Theorem \ref{thm:typ_typ_typ_fusion})}
  For $\lambda,\mu\in\CC\setminus L^\circ$ such that $\lambda+\mu\notin L^\circ$,
\begin{equation*}
\cF_{\lambda}\tens\cF_\mu \cong \bigoplus_{\ell=0}^{p-1} \cF_{\lambda + \mu + \ell \alpha_-}.
\end{equation*}
 \end{enumerate}

\end{thm}

The image of these fusion rules in the Grothendieck ring was predicted earlier in \cite{CM} using a conjectural Verlinde formula, and our results confirm that conjecture. Moreover, Theorems \ref{thm:intro_projective} and \ref{thm:intro_fus_rules} show that $\cO_{\cM(p)}^T$ is equivalent as an abelian category to the category of finite-dimensional weight modules for the unrolled restricted quantum group of $\mathfrak{sl}_2$ at $q = e^{\pi i /p}$, and that under this equivalence, tensor product decompositions agree. See \cite{CGP2} for the detailed structure of the unrolled quantum group category.

Our last result on $\cO_{\cM(p)}$, proved in Section \ref{sec:rigidity}, is rigidity:
\begin{thm} \textup{(Theorems \ref{thm:rig:O} and  \ref{thm:rig:OT})} 
The tensor categories $\cO_{\cM(p)}$ and $\cO_{\cM(p)}^T$  are rigid and ribbon. 
\end{thm}

To prove this, we use \cite[Theorem 4.4.1]{CMY-singlet} to reduce rigidity for the entire category of finite-length $\cM(p)$-modules to rigidity for all simple modules. Since we already proved in \cite{CMY-singlet} that the atypical category $\cC_{\cM(p)}$ is rigid, it is then enough to prove that the typical modules $\cF_{\lambda}$, $\lambda\in\CC\setminus L^\circ$ are rigid (as $\cM(p)$-modules). Our rigidity proof for $\cF_\lambda$ is new and completely different from the explicit calculational proofs of rigidity for typical modules of the $\beta\gamma$-vertex algebra \cite{AW} and of affine $\mathfrak{gl}_{1\vert 1}$ \cite{CMY3}. The idea is to choose evaluations $e_\lambda: \cF_\lambda'\tens\cF_\lambda\rightarrow\cM(p)$ and coevaluations $i_\lambda: \cM(p)\rightarrow\cF_\lambda\tens\cF_\lambda'$ (where $\cF_\lambda'$ is the $\cM(p)$-module contragredient of $\cF_\lambda$, also a typical Fock module) in such a way that the rigidity composition
\begin{align*}
 \cF_\lambda\xrightarrow{\cong}\cM(p)\tens\cF_\lambda\xrightarrow{i_\lambda\tens\Id} & (\cF_\lambda\tens\cF_{\lambda}')\tens\cF_\lambda\xrightarrow{\cong} \cF_\lambda\tens(\cF_{\lambda}'\tens\cF_\lambda)\xrightarrow{\Id\tens e_\lambda} \cF_\lambda\tens\cM(p)\xrightarrow{\cong} \cF_\lambda
\end{align*}
depends analytically on the Heisenberg weight $\lambda$. Then $\cF_\lambda$ is either rigid for a dense open set of $\lambda$ or non-rigid for all $\lambda$. The latter is impossible because we already know from \cite{CMY-singlet} that the modules $\cF_{\alpha_{r,p}}$, $r\in\ZZ$, are rigid, so rigidity for $\cF_\lambda$ holds for generic $\lambda$. Then we use the fusion rules of Theorem \ref{thm:intro_fus_rules}(4) to prove rigidity for all $\lambda\in\CC\setminus L^\circ$.

To show that the rigidity composition indeed depends analytically on $\lambda$, we revisit Huang's derivation \cite{Hu-diff-eqn} of regular-singular-point differential equations for conformal-field-theoretic four-point functions coming from intertwining operators among $C_1$-cofinite modules for a vertex operator algebra. Using generic Fock modules, on which the Heisenberg zero-mode acts by a polynomial variable, we show that such differential equations associated to intertwining operators among typical $\cM(p)$-modules can be chosen to depend analytically on $\lambda$. Then because matrix coefficients of the rigidity composition appear as coefficients of suitable four-point functions, the theory of ordinary differential equations combined with some additional complex analysis shows that the rigidity composition also depends analytically on $\lambda$, as desired.

\subsection{Applications and outlook}

To summarize, we show in this paper that the category $\cO_{\cM(p)}$ of $C_1$-cofinite grading-restricted generalized $\cM(p)$-modules is the same as the category of finite-length grading-restricted generalized $\cM(p)$-modules, and that this category has braided ribbon tensor category structure. We also compute all tensor products of simple modules in $\cO_{\cM(p)}$, identify a subcategory $\cO_{\cM(p)}^T$ which contains enough projective objects, and describe all projective objects in $\cO_{\cM(p)}^T$. Since all $\mathbb{N}$-gradable simple $\cM(p)$-modules are objects of $\cO_{\cM(p)}$ \cite{Ad, CMR}, our results mean that the full representation theory of the singlet vertex algebras is now `understood' in a strong sense. However, there are some remaining open problems related to the singlet vertex algebras which we discuss in this subsection. We also discuss some applications of our results.

The first major application is that we can now construct ribbon tensor (super)categories of modules for interesting vertex operator (super)algebras that contain $\cM(p)$ as a subalgebra. We give one example in this paper, namely the finite cyclic orbifolds of the triplet vertex operator algebra $\cW(p)$ studied previously in \cite{ALM, ALM2}, which are simple current extensions of $\cM(p)$. Thus we can use the vertex operator algebra extension theory of \cite{CKM-exts, CMY-completions} to show that the module category of any cyclic orbifold of $\cW(p)$ is a rigid non-degenerate braided tensor category. This gives new examples of $C_2$-cofinite vertex operator algebras whose representation categories are non-semisimple modular tensor categories. We also confirm the conjectural classification from \cite{ALM} of simple modules for these algebras, and we describe all projective modules.

We can also obtain ribbon tensor (super)categories which are both non-finite and non-semisimple, and which contain modules with infinite-dimensional conformal weight spaces, and even modules without lower bounds on their conformal weights. The first vertex operator algebras with such module categories that we can study are the $\mathcal B_p$-algebras, $p \in \ZZ_{>1}$, of \cite{CRW}. The $\mathcal B_2$-algebra is the $\beta\gamma$-vertex algebra (already analyzed in \cite{AW}), while the $\mathcal B_3$-algebra is the simple affine vertex algebra of $\mathfrak{sl}_2$ at the admissible level $- \frac{4}{3}$, first studied by Adamovi\'{c} \cite{Ad3}. For larger $p$, the $\mathcal B_p$-algebra is isomorphic to the simple subregular $W$-algebra of $\mathfrak{sl}_{p-1}$ at level $-(p-1)+ \frac{p-1}{p}$ \cite{ACGY, ACKM}. The singlet algebra $\cM(p)$ is a coset of $\mathcal B_p$ by a rank-one Heisenberg vertex operator algebra, which means that $\mathcal B_p$ is an extension of the tensor product of these two commuting subalgebras. Subregular $W$-algebras of type $A$ such as $\mathcal{B}_p$ enjoy a duality with certain principal $W$-superalgebras, first conjectured by Feigin and Semikhatov \cite{FS} and hence called Feigin-Semikhatov duality. This duality was proved in \cite{CGN, CL}, and its representation-theoretic consequences are explored in \cite{CGNS}. Let us denote the  Feigin-Semikhatov dual of $\mathcal B_p$ by $\mathcal S_p$; for $p>2$ it is the simple principal $W$-superalgebra of $\mathfrak{sl}_{p-1|1}$ at level $-(p-2) + \frac{p}{p-1}$.  The case $p=2$ is special: we set $\mathcal S_2$ to be the affine vertex superalgebra of $\mathfrak{gl}_{1|1}$. Then
$\mathcal S_p$ is also an extension of $\cM(p)$ times a  Heisenberg algebra.

Since the first version of this paper was completed, we have combined our results on $\cM(p)$ with vertex operator superalgebra extension theory to construct non-finite and non-semisimple rigid tensor (super)categories of modules for $\mathcal{B}_p$ and $\mathcal{S}_p$, $p>1$, as well as for the finite cyclic orbifolds of $\mathcal{B}_2$ and their Feigin-Semikhatov duals \cite{CMY-Bp}. These include the first examples, other than the tensor category for the $\beta\gamma$-vertex algebra $\mathcal{B}_2$ constructed by Allen and Wood \cite{AW}, of vertex algebraic tensor categories containing modules with infinite-dimensional conformal weight spaces and modules with no lower bounds on their conformal weights. 

Another family of vertex algebras that contains $\mathcal M(2)$ as a subalgebra is the minimal $W$-algebras of $\mathfrak{sl}_{2n}$ at level $-n -\frac{1}{2}$ \cite{ACPV}. Using this result together with our understanding of the full representation category of $\mathcal M(2)$ from the present paper, the main result of \cite{ACPV} is then that the Kazhdan-Lusztig category for the simple affine vertex operator algebra of  $\mathfrak{sl}_{2n}$ at the non-admissible level $-n -\frac{1}{2}$, $n\in\mathbb{Z}_{\geq 2}$, is semisimple, and hence a braided tensor category by \cite{CY}. Moreover, all simple objects of this tensor category are simple currents and hence rigid. This example is interesting because so far there are no good techniques in general for studying representation categories of affine vertex operator algebras at non-admissible levels. 
Further vertex operator (super)algebras containing singlet subalgebras will be analyzed in forthcoming papers; examples will include most affine vertex operator superalgebras at level $1$.

Further applications of the singlet algebras and their extensions appear in physics and low-dimensional topology. For example, the $\mathcal B_p$-algebras are the chiral algebras of certain four-dimensional superconformal field theories called Argyres-Douglas theories of type $(A_1, A_{2p-3})$ \cite{C, ACGY}. The singlet algebra itself has recently appeared in connection with new invariants of $3$-manifolds called $\hat{Z}$-invariants introduced in \cite{CCFGH}. These are formal power series associated to a $3$-manifold together with a connected component of the moduli space of flat $G$-connections for some complex Lie group $G$. For $G=SL(2)$, the power series for certain $3$-manifolds  coincide with characters of certain $\cM(p)$-modules, and it is proposed that the modules themselves are associated to the flat connections. Also in low-dimensional topology,  Costantino, Geer and Patureau-Mirand have introduced axiomatic TQFTs associated to non-finite and non-semisimple categories \cite{CGP}. In \cite{CDGG}, non-semisimple topological field theories in the physics sense are introduced. These theories support vertex operator algebras at certain two-dimensional \emph{corners} of the theory. These vertex operator algebras  are closely related to $\cM(p)$, and in particular a TQFT (in the sense of \cite{CGP}) associated to  $\cO_{\cM(p)}$  should appear (see \cite[Conjecture 1]{CDGG}). Actually, in that conjecture, the unrolled restricted quantum group of $\mathfrak{sl}_2$ appears; this leads to an interesting open problem related to the singlet algebra. 

The logarithmic Kazhdan-Lusztig correspondence refers to equivalences of non-semisimple braided tensor categories associated to quantum groups and vertex operator algebras. The best-known example is the correspondence between a quasi-Hopf modification of the restricted quantum group of $\mathfrak{sl}_2$ at a $2p$th root of unity and the triplet algebra $\cW(p)$ \cite{FGST1, FGST2, FHST, NT, CGR, CLM, GN}. But there is also a conjectural correspondence between the category of finite-dimensional weight modules for the unrolled restricted quantum group of $\mathfrak{sl}_2$ at a $2p$th root of unity and a suitable category of $\cM(p)$-modules \cite{CGP2, CMR}. Based on our results in this paper, especially Theorems \ref{thm:intro_projective} and \ref{thm:intro_fus_rules}, it is clear that $\cO_{\cM(p)}^T$ is the correct category of $\cM(p)$-modules for this correspondence.

So far, existence of a braided tensor equivalence between $\cO_{\cM(p)}^T$ and the corresponding category for the unrolled restricted quantum group remains an open problem. Even though we have described $\cO_{\cM(p)}^T$ in great detail in this paper, to the extent that we can see that the singlet and quantum group categories are equivalent as categories and have the same Grothendieck rings, this does not necessarily mean there is an equivalence of categories that also preserves braided tensor structure. Recently, Gannon and Negron have achieved a major partial result on this Kazhdan-Lusztig correspondence \cite{GN}: they proved that there is a braided tensor equivalence between the atypical subcategory $\cC_{\cM(p)}^0$ (which we constructed in \cite{CMY-singlet}) and the corresponding subcategory on the unrolled restricted quantum group side. However, their methods do not extend to proving a tensor equivalence between $\cO_{\cM(p)}^T$ and the full quantum group category since they relied on the existence of a tensor generator for $\cC_{\cM(p)}^0$, namely $\cM_{1,2}$, whereas $\cO_{\cM(p)}^T$ is not singly-generated as a tensor category.

A different method for proving logarithmic Kazhdan-Lusztig correspondences was introduced recently in \cite{CLM}. We believe that this new method can be adapted to prove the full correspondence between $\cO_{\cM(p)}^T$ and the unrolled restricted quantum group category. If so, then the tensor category structure on $\cO_{\cM(p)}^T$ derived in this paper will surely be needed in the proof, since the techniques of \cite{CLM} require \textit{a priori} rigid tensor category structure on $\cO_{\cM(p)}^T$. (In particular, although \cite{CLM} proves the logarithmic Kazhdan-Lusztig correspondence for the triplet algebra $\cW(2)$, it does not address rigidity or fusion rules for typical singlet modules in the $p=2$ case.)

We conclude this introduction with one more open problem related to the singlet vertex algebras. In \cite{ALM-Dm}, it was shown that $\cM(p)$ has an order-$2$ automorphism $\Psi$, and partial results on the classification of simple modules for the $\ZZ/2\ZZ$-fixed-point subalgebra $\cM(p)^+$ were achieved. However, except for small values of $p$, the full classification of simple $\mathcal{M}(p)^+$-modules, equivalently of simple $\Psi$-twisted $\cM(p)$-modules, is still conjectural. It is also not known whether there is a braided tensor category of $\cM(p)^+$-modules which contains all simple modules. (If such a tensor category exists, then one could use vertex operator algebra extension theory to study the dihedral group orbifolds of the triplet algebra $\cW(p)$, since these orbifolds contain $\cM(p)^+$ as subalgebra.) Unfortunately, constructing such a tensor category for $\cM(p)^+$ seems to be difficult, since some simple $\cM(p)^+$-modules (the ones obtained from $\Psi$-twisted $\cM(p)$-modules in \cite{ALM-Dm})  seem not to be $C_1$-cofinite. But on the other hand, this means that $\cM(p)^+$ could be an interesting test case for extending vertex algebraic tensor category theory beyond $C_1$-cofinite modules.

\vspace{3mm}
 
 \noindent{\bf Acknowledgments.} We thank Shashank Kanade for discussions, and we thank the referee for suggestions on the presentation of this paper. T. Creutzig  is supported by  NSERC Grant Number RES0048511. R. McRae is supported by a startup grant from Tsinghua University. J. Yang is supported by a startup grant from Shanghai Jiaotong University.

\section{Preliminaries}

In this section, we briefly recall notation for vertex operator algebras, their modules, and intertwining operators, and then we discuss in more detail the representation theories of the Virasoro and singlet vertex operator algebras at central charge $c_{p,1}=13-6p-6p^{-1}$.

\subsection{Vertex operator algebras and intertwining operators}\label{subsec:VOAs_and_intw_ops}

We use the definition of vertex operator algebra from \cite{FLM, LL}. In particular, a vertex operator algebra $V$ has a conformal weight grading $V=\bigoplus_{n\in\ZZ} V_{(n)}$ given by eigenvalues of the Virasoro zero-mode $L(0)$, a vertex operator map $Y: V\otimes V\rightarrow V((x))$, a vacuum vector $\vac\in V_{(0)}$, and a conformal vector $\omega\in V_{(2)}$. Given a vertex operator algebra $V$, we use the definition of generalized $V$-module from \cite{HLZ1}. In particular, a generalized $V$-module $W=\bigoplus_{h\in\CC} W_{[h]}$ is a graded vector space such that each $W_{[h]}$ is the generalized $L(0)$-eigenspace with generalized eigenvalue $h$. A generalized $V$-module is \textit{grading restricted} if each $W_{[h]}$ is finite dimensional and for any $h\in\CC$, $W_{[h-n]}=0$ for all sufficiently positive $n\in\ZZ$. We use the notation
\begin{align*}
Y_W: V\otimes W & \rightarrow W((x))\nonumber\\
v\otimes w & \mapsto Y_W(v,x)w=\sum_{n\in\ZZ} v_n w\,x^{-n-1}
\end{align*}
for the vertex operator action of $V$ on a (grading-restricted) generalized $V$-module $W$. We will sometimes call grading-restricted generalized $V$-modules simply \textit{$V$-modules} for short.

A \textit{weak $V$-module} is a module for $V$ considered as a vertex algebra, that is, no grading is assumed. Then an \textit{$\NN$-gradable weak $V$-module} is a weak $V$-module $W$ that admits an $\NN$-grading $W=\bigoplus_{n\in\NN} W(n)$ such that
\begin{align}\label{eqn:N-gradable}
\deg v_n w = \mathrm{wt}\,v+\deg w-n-1
\end{align}
for any $n\in\ZZ$ and homogeneous $v\in V$, $w\in W$ (here we are using the notation $\deg w = m$ for $w\in W(m)$ to distinguish the $\NN$-grading of $W$ from the conformal weight grading of a generalized $V$-module). It is easy to see that any grading-restricted generalized $V$-module is $\NN$-gradable (see for example \cite[Remark 2.4]{CMY-completions}).

For $W$ an $\NN$-gradable weak $V$-module, we define its \textit{top level} to be
\begin{equation*}
T(W)=\lbrace w\in W\,\vert\,v_n w =0\,\,\text{if}\,\,v\,\,\text{is homogeneous and}\,\,\mathrm{wt}\,v-n-1<0\rbrace.
\end{equation*}
Clearly $W(0)\subseteq T(W)$, though the reverse inclusion might not hold if $W$ is not simple. In \cite{Zh}, Y. Zhu showed that $T(W)$ is a module for the \textit{Zhu algebra} $A(V)$. We will not need the precise definition of $A(V)$ here; we just recall that $A(V)$ is a unital associative algebra structure on $V/O(V)$ for a certain subspace $O(V)\subseteq V$. For $v\in V$, we use the notation $[v]=v+O(V)\in A(V)$; the action of $[v]$ on the top level of an $\NN$-gradable weak $V$-module $W$ is given by $[v]\cdot w=o(v)w$ for $v\in V$, $w\in T(W)$, where \begin{align*}
o(v)=\mathrm{Res}_x\,x^{-1}Y_W(x^{L(0)} v,x)
\end{align*}
is the degree-preserving component of $Y_W$. If $v$ is homogeneous, then $o(v)=v_{\mathrm{wt}\,v-1}$.

The forgetful functor $T$ from $\NN$-gradable weak $V$-modules to $A(V)$-modules has a left adjoint: As in \cite[Definition 2.7]{Li-intw-ops}, the \textit{generalized Verma $V$-module} $G(M)$ induced from an $A(V)$-module $M$ is an $\NN$-gradable weak $V$-module equipped with a homomorphism $M\rightarrow T(G(M))$ such that for any $A(V)$-module homomorphism $f: M\rightarrow T(W)$, where $W$ is an $\NN$-gradable weak $V$-module, there is a unique $V$-module homomorphism $F: G(M)\rightarrow W$ making the diagram
\begin{equation*}
\xymatrix{
G(M) \ar[rd]^{F} & \\
M \ar[u] \ar[r]_{f} & W\\
}
\end{equation*}
commute. In particular, if $W$ is a simple $V$-module, with conformal weight grading necessarily of the form $\bigoplus_{n\in\NN} W_{[h+n]}$ for some $h\in\CC$, then $T(W)=W_{[h]}$ and $W$ is the unique irreducible quotient of $G(W_{[h]})$.

More generally, for $W=\bigoplus_{n\in\NN} W(n)$ an $\NN$-gradable weak $V$-module and any $N\in\NN$, each subspace $W(n)$ for $0\leq n\leq N$ is a module for the higher-level Zhu algebra $A_N(V)=V/O_N(V)$ constructed in \cite{DLM}. For $v\in V$, the action of $[v]=v+O_N(V)\in A_N(V)$ on each $W(n)$ is again given by $o(v)$. Moreover, \cite[Theorem 4.1]{DLM} shows that for any $A_N(V)$-module $M$, there is an $\NN$-gradable generalized Verma $V$-module $G_N(M)$ such that $[G_N(M)](N)\cong M$ as $A_N(V)$-modules and such that $G_N(M)$ is generated by this subspace. (The assertion in \cite[Theorem 4.1]{DLM} that $[G_N(M)](0)\neq 0$ only applies if $M$ does not factor through an $A_{N-1}(V)$-module).

If $M$ is a finite-dimensional simple $A_N(V)$-module for some $N\in\NN$, then $L(0)=o(\omega)$ (where $\omega$ is the conformal vector of $V$) acts by a scalar on $M$, since $[\omega]$ is central in $A_N(V)$ by \cite[Theorem 2.3(iii)]{DLM}. Then \eqref{eqn:N-gradable} and the fact that $M$ generates $G_N(M)$ as a $V$-module imply that $G_N(M)$ has a conformal weight grading which is just a shift of its $\NN$-grading. In particular, every $V$-submodule of $G_N(M)$ is $\NN$-graded. This means that $G_N(M)$ has a maximal proper $V$-submodule, which is the sum of all ($\NN$-graded) submodules which intersect the generating subspace $[G_N(M)](N)\cong M$ trivially. Thus $G_N(M)$ has a unique simple (weak) $V$-module quotient when $M$ is a finite-dimensional simple $A_N(V)$-module.

We now recall some elements of the (logarithmic) vertex algebraic tensor category theory developed in \cite{HLZ1}-\cite{HLZ8}. We use the definition of (logarithmic) intertwining operator from \cite{HLZ2}. In particular, if $W_1$, $W_2$, and $W_3$ are three modules for a vertex operator algebra $V$, an intertwining operator of type $\binom{W_3}{W_1\,W_2}$ is a linear map
\begin{align*}
\cY: W_1\otimes W_2 & \rightarrow W_3[\log x]\lbrace x\rbrace\nonumber\\
w_1\otimes w_2 & \mapsto\cY(w_1,x)w_2=\sum_{h\in\CC}\sum_{k\in\NN} (w_1)_{h;k}w_2\,x^{-h-1}(\log x)^k
\end{align*}
which satisfies lower truncation, the $L(-1)$-derivative property, and the Jacobi identity. We will need the following two consequences of the Jacobi identity: the commutator formula
\begin{equation}\label{eqn:intw_op_comm}
v_n\cY(w_1,x) = \cY(w_1,x)v_n +\sum_{k\geq 0}\binom{n}{k} x^{n-k}\cY(v_k w_1,x)
\end{equation}
for $v\in V$, $n\in\ZZ$, and $w_1\in W_1$; and the iterate formula
\begin{equation}\label{eqn:intw_op_iterate}
\cY(v_n w_1,x)=\sum_{k\geq 0}(-1)^k\binom{n}{k}\left(v_{n-k} x^k\cY(w_1,x)-(-1)^n x^{n-k}\cY(w_1,x)v_k\right)
\end{equation}
for $v\in V$, $n\in\ZZ$, and $w_1\in W_1$. Given $V$-modules $W_1$, $W_2$, and $W_3$, we use $I_V\binom{W_3}{W_1\,W_2}$ to denote the vector space of intertwining operators of type $\binom{W_3}{W_1\,W_2}$. An intertwining operator $\cY\in I_V\binom{W_3}{W_1\,W_2}$ is called surjective if $W_3$ is spanned by vectors of the form $(w_1)_{h;k} w_2$ for $w_1\in W_1$, $w_2\in W_2$, $h\in\CC$, and $k\in\NN$.

For $z\in\CC^\times$, a $P(z)$-intertwining map is a linear map $W_1\otimes W_2\rightarrow \overline{W}_3:=\prod_{h\in\CC} (W_3)_{[h]}$ obtained by substituting $x\mapsto z$ in an intertwining operator $\cY$ of type $\binom{W_3}{W_1\,W_2}$, using any choice of branch of $\log z$. Given a branch $\ell(z)$ of logarithm, $\cY(w_1,e^{\ell(z)})$ denotes the intertwining map specified by this branch. In \cite{HLZ3}, tensor products of $V$-modules are defined using intertwining maps, but one can equivalently use intertwining operators:
\begin{defi}
Let $\cC$ be a category of generalized $V$-modules, and let $W_1$, $W_2$ be objects of $\cC$. A \textit{tensor product} of $W_1$ and $W_2$ in $\cC$ is a pair $(W_1\tens W_2, \cY_\tens)$, where $W_1\tens W_2$ is an object of $\cC$ and $\cY_\tens$ is an intertwining operator of type $\binom{W_1\tens W_2}{W_1\,W_2}$, that satisfies the following universal property: For any object $W_3$ of $\cC$ and intertwining operator $\cY$ of type $\binom{W_3}{W_1\,W_2}$, there is a unique $V$-module homomorphism $f_\cY: W_1\tens W_2\rightarrow W_3$ such that $f_\cY\circ\cY_\tens=\cY$.
\end{defi}

If a tensor product $(W_1\tens W_2,\cY_\tens)$ of $W_1$ and $W_2$ in $\cC$ exists, then it is unique up to unique isomorphism, and the tensor product intertwining operator $\cY_\tens$ is surjective (see \cite[Proposition 4.23]{HLZ3}). Under suitable conditions on $\cC$ specified in \cite{HLZ1}-\cite{HLZ8} (such as closure under tensor products), vertex algebraic tensor products give $\cC$ braided tensor category structure. For a detailed description of the left and right unit isomorphisms $l$ and $r$, the associativity isomorphisms $\cA$, and the braiding isomorphisms $\cR$ in a braided tensor category of grading-restricted generalized modules for a vertex operator algebra, see \cite{HLZ8} or the exposition in \cite[Section 3.3]{CKM-exts}. 

Although the conditions imposed on $\cC$ in \cite{HLZ1}-\cite{HLZ8} are extensive, most of them are satisfied by the category of $C_1$-cofinite $V$-modules (especially convergence of compositions of intertwining operators \cite{Hu-diff-eqn} and closure under tensor products \cite{Mi}). A $V$-module $W$ is $C_1$-cofinite if $\dim W/C_1(W)<\infty$, where $C_1(W)$ is the span of vectors $v_{-1} w$ for $w\in W$ and $v\in V$ homogeneous such that $\mathrm{wt}\,v>0$. In \cite{CJORY, CY}, it is essentially shown that the category of $C_1$-cofinite grading-restricted $V$-modules indeed satisfies all conditions in \cite{HLZ1}-\cite{HLZ8} (and thus is a vertex algebraic braided tensor category) if it is closed under contragredient modules. Recall from \cite{FHL} that the contragredient of a $V$-module $W$ is a $V$-module structure on the graded dual $W'=\bigoplus_{h\in\CC} W_{[h]}^*$. The category of $C_1$-cofinite grading-restricted generalized $V$-modules is indeed closed under contragredients if it equals the category of finite-length grading-restricted generalized $V$-modules, and this in turn holds if for all finite-dimensional irreducible $A(V)$-modules $M$, the generalized Verma $V$-module $G(M)$ has finite length (see \cite[Theorems 3.3.4 and 3.3.5]{CY}). In Section \ref{sec:tens_cats}, we will use this criterion to show that the category of $C_1$-cofinite modules for the singlet vertex operator algebra $\cM(p)$, $p>1$, is a vertex algebraic braided tensor category.

\subsection{The Virasoro algebra at central charge \texorpdfstring{$c_{p,1}$}{c(p,1)}}

As usual, the Virasoro algebra $\cV ir$ is the Lie algebra with basis $\lbrace L(n)\,\vert\,n\in\ZZ\rbrace\cup\lbrace\mathbf{c}\rbrace$ with $\mathbf{c}$ central and commutators
\begin{equation*}
 [L(m),L(n)]=(m-n)L(m+n)+\frac{m^3-m}{12}\delta_{m+n,0}\mathbf{c}
\end{equation*}
for $m,n\in\ZZ$. Let $\cV ir_{\geq 0}=\mathrm{span}\lbrace L(n),\mathbf{c}\,\vert\,n\geq 0\rbrace$ and $\cV ir_-=\mathrm{span}\lbrace L(n)\,\vert\,n<0\rbrace$.

A $\cV ir$-module $W$ has \textit{central charge} $c\in\CC$ if $\mathbf{c}$ acts on $W$ as scalar multiplication by $c$. In this work, we only consider $\cV ir$-modules of central charge
\begin{equation*}
 c_{p,1}:=13-6p-6p^{-1}=1-6\frac{(p-1)^2}{p}
\end{equation*}
for $p\in\ZZ_{>1}$. At this central charge, the Verma module $\cV_h$ for $h\in\CC$ is the induced module
\begin{equation*}
 \cV_h=U(\cV ir)\otimes_{U(\cV ir_{\geq 0})} \CC v_h,
\end{equation*}
where $\CC v_h$ is the one-dimensional $\cV ir_{\geq 0}$-module on which $\mathbf{c}$ acts by the central charge $c_{p,1}$, $L(0)$ acts by the conformal weight $h$, and $L(n)$ for $n>0$ acts by $0$. By the Feigin-Fuchs criterion for the existence of singular vectors in Verma modules \cite{FF}, $\cV_h$ is reducible if and only if $h=h_{r,s}$ for some $r,s\in\ZZ_+$, where
\begin{equation}\label{eqn:hrs_def}
 h_{r,s}:=\frac{r^2-1}{4}p-\frac{rs-1}{2}+\frac{s^2-1}{4}p^{-1}=\frac{(pr-s)^2-(p-1)^2}{4p}.
\end{equation}
Due to the conformal weight symmetries $h_{r+1,s+p}=h_{r,s}$ and $h_{r,s}=h_{-r,-s}$, we may assume $1\leq s\leq p$. For $r\geq 1$ and $1\leq s\leq p$, we use the notation $\cV_{r,s}=\cV_{h_{r,s}}$, and the notation $v_{r,s}$ for a generating vector of conformal weight $h_{r,s}$ in $\cV_{r,s}$. Non-zero (necessarily injective) homomorphisms between reducible Verma modules $\cV_{r,s}$ are completely described by the following embedding diagrams (see for example \cite[Section 5.3]{IK}):
\begin{itemize}
 \item For $r\geq 1$ and $1\leq s\leq p-1$, we have the diagram
 \begin{equation}\label{eqn:embedding_rs}
  \cV_{r,s}\longleftarrow\cV_{r+1,p-s}\longleftarrow\cV_{r+2,s}\longleftarrow\cV_{r+3,p-s}\longleftarrow\cdots
 \end{equation}

 \item For $r\geq 1$ and $s=p$, we have the diagram
 \begin{equation}\label{eqn:embedding_rp}
  \cV_{r,p}\longleftarrow\cV_{r+2,p}\longleftarrow\cV_{r+4,p}\longleftarrow\cV_{r+6,p}\longleftarrow\cdots
 \end{equation} 
\end{itemize}
For each $r\geq 1$ and $1\leq s\leq p$, every non-zero submodule of $\cV_{r,s}$ is generated by its singular vectors, that is, its $L(0)$-eigenvectors which are annihilated by $L(n)$ for $n>0$ (see \cite[Theorem 6.5]{IK}). Since a singular vector of conformal weight $h$ in $\cV_{r,s}$ induces a non-zero $\cV ir$-module homomorphism $\cV_h\rightarrow\cV_{r,s}$, we can thus make the following observations based on the above embedding diagrams:
\begin{prop}\label{prop:Verma_structure}
 For $r\geq 1$ and $1\leq s\leq p$,
 \begin{enumerate}
  \item Each non-zero submodule $W$ in $\cV_{r,s}$ or one of its quotients is generated by a unique (up to scale) singular vector of minimal conformal weight in $W$.
  
  \item The unique irreducible quotient of $\cV_{r,s}$ is
  \begin{equation*}
   \cL_{r,s}=\left\lbrace\begin{array}{lll}
                          \cV_{r,s}/\cV_{r+1,p-s} & \text{if} & 1\leq s\leq p-1\\
                          \cV_{r,p}/\cV_{r+2,p} & \text{if} & s=p\\
                         \end{array}
\right. .
  \end{equation*}
 \end{enumerate}
\end{prop}

Taking $(r,s)=(1,1)$, the maximal proper submodule of $\cV_{1,1}$ is generated by the singular vector $L(-1)v_{1,1}$, and the irreducible quotient $\cL_{1,1}$ is a vertex operator algebra with vacuum $\vac=v_{1,1}+\langle L(-1)v_{1,1}\rangle$ and conformal vector $\omega=L(-2)\vac$ \cite{FZ1}. We use $\cL(p)$ to denote $\cL_{1,1}$ considered as a vertex operator algebra. Irreducible $\cL(p)$-modules are precisely the irreducible quotients of Verma modules, that is, the $\cL_{r,s}$ for $r,s\in\ZZ_+$ and the $\cV_h$ for $h\in\CC\setminus\lbrace h_{r,s}\,\vert\,r,s\in\ZZ_+\rbrace$. It was shown in \cite{CJORY} that the category of $C_1$-cofinite grading-restricted generalized $\cL(p)$-modules equals the category of finite-length central-charge-$c_{p,1}$ $\cV ir$-modules whose composition factors come from the $\cL_{r,s}$ for $r\geq 1$, $1\leq s\leq p$. Further, $\cO_p$ admits the vertex algebraic braided tensor category structure of \cite{HLZ1}-\cite{HLZ8}; the detailed structure of this tensor category was determined in \cite{MY}, where it was shown in particular that $\cO_p$ is rigid.

The Zhu algebra $A(\cL(p))$ is isomorphic to $\CC[x]$, with the isomorphism given by $[\omega]\mapsto x$ \cite{FZ1}. Intertwining operators among $\cL(p)$-modules $W$ can be studied using the $A(\cL(p))$-bimodules $A(W)$; see \cite{FZ1} for their precise definition. For our purposes here, we will only need a special case of \cite[Proposition 2.5]{MY}; to prepare for its statement, note that a conformal weight space of an $\cL(p)$-module with minimal conformal weight is an $A(\cL(p))$-module on which $[\omega]$ acts by $L(0)$: 
\begin{prop}\label{prop:piY_surjective}
 Let $\cY$ be a surjective intertwining operator of type $\binom{W_3}{W_1\,W_2}$, where $W_1$, $W_2$, and $W_3$ are generalized $\cL(p)$-modules such that $W_2$ is a quotient of a Verma module $\cV_{h_2}$, $h_2\in\CC$, and the conformal weight grading of $W_3$ has the form $W_3=\bigoplus_{n\in\NN} (W_3)_{[h_3+n]}$ for some $h_3\in\CC$. Then there is a surjective $A(\cL(p))$-module homomorphism
 \begin{equation*}
  \pi(\cY): A(W_1)\otimes_{A(\cL(p))}\CC v_{h_2}\rightarrow (W_3)_{[h_3]}.
 \end{equation*}
In particular, if $A(W_1)\otimes_{A(\cL(p))}\CC v_{h_2}$ is finite dimensional and $(W_3)_{[h_3]}$ is non-zero, then $h_3$ is an eigenvalue for the action of $[\omega]$ on $A(W_1)\otimes_{A(\cL(p))}\CC v_{h_2}$. 
\end{prop}

\begin{rem}
 In \cite[Proposition 2.5]{MY}, it was assumed that $W_3$ is finitely generated, but this assumption was only used to ensure that the conformal weights of $W_3$ are contained in finitely many cosets of $\CC/\ZZ$.
\end{rem}

We will need to use Proposition \ref{prop:piY_surjective} in the special cases $W_1=\cL_{1,2},\cL_{3,1}$. In these cases, the $A(\cL(p))$-bimodule $A(W_1)$ was determined in \cite{FZ2} (under the unnecessary assumption $p\notin\QQ$); see \cite[Example 2.12]{FZ2} or the calculations in \cite[Sections 3.1 and 5.1]{MY}. The $A(\cL(p))\cong\CC[x]$-bimodule $A(\cL_{1,2})$ is given by
\begin{equation*}
 A(\cL_{1,2})\cong\CC[x,y]/(f_{1,2}(x,y))
\end{equation*}
where
\begin{equation*}
 f_{1,2}(x,y)=\left(x-y-\left(h_{1,2}+1-p^{-1}\right)\right)(x-y-h_{1,2})-p^{-1} y
\end{equation*}
and the left and right actions of $[\omega]$ are multiplication by $x$ and $y$, respectively. Similarly,
\begin{equation*}
 A(\cL_{3,1})\cong\CC[x,y]/(f_{3,1}(x,y))
\end{equation*}
where
\begin{equation*}
 f_{3,1}(x,y)=(x-y)\left((x-y-h_{3,1})(x-y-1)-4py\right).
\end{equation*}
Thus for any $h\in\CC$,
\begin{align*}
 A(\cL_{1,2})\otimes_{A(\cL(p))} \CC v_h &\cong \CC[x]/(f_{1,2}(x,h)),\nonumber\\
 A(\cL_{3,1})\otimes_{A(\cL(p))} \CC v_h & \cong \CC[x]/(f_{3,1}(x,h)).
\end{align*}
After finding the roots of $f_{1,2}(x,h)$ and $f_{3,1}(x,h)$, we use Proposition \ref{prop:piY_surjective} to conclude:
\begin{cor}\label{cor:conf_wts}
 Let $W_2$ and $W_3$ be grading-restricted generalized $\cL(p)$-modules such that $W_2$ is a quotient of a Verma module $\cV_{h_2}$ for some $h_2\in\CC$, and the conformal weight grading of $W_3$ has the form $W_3=\bigoplus_{n\in\NN} (W_3)_{[h_3+n]}$ for some $h_3\in\CC$ with $(W_3)_{[h_3]}\neq 0$.
 \begin{enumerate}
  \item If there is a surjective intertwining operator of type $\binom{W_3}{\cL_{1,2}\,W_2}$, then
  \begin{equation*}
   h_3\in\left\lbrace h_2+\frac{p^{-1}}{4}\pm\frac{p^{-1}}{2}\sqrt{4ph_2+(p-1)^2} \right\rbrace .
  \end{equation*}

  \item If there is a surjective intertwining operator of type $\binom{W_3}{\cL_{3,1}\,W_2}$, then
  \begin{equation*}
   h_3\in\left\lbrace h_2, h_2+p\pm\sqrt{4ph_2+(p-1)^2}\right\rbrace.
  \end{equation*}
 \end{enumerate}
\end{cor}

We will also use a special case of \cite[Proposition 2.10]{Li-intw-ops}, which says that under certain conditions, the map $\pi(\cY)$ in the statement of Proposition \ref{prop:piY_surjective} vanishes if and only if $\cY$ does:
\begin{prop}\label{prop:pi_injective}
 In the setting of Proposition \ref{prop:piY_surjective}, if $W_1$, $W_2$, and $W_3$ are irreducible $\cL(p)$-modules, then 
 \begin{equation*}
  \dim I_{\cL(p)}\binom{W_3}{W_1\,W_2}\leq\dim\hom_{A(\cL(p))}\left(A(W_1)\otimes_{A(\cL(p))} \CC v_{h_2}, (W_3)_{[h_3]}\right).
 \end{equation*}
\end{prop}

\begin{rem}\label{rem:pi_injective}
 If $W_1$, $W_2=U(\cV ir)\cdot v_{h_2}$, and $W_3$ are irreducible $\cL(p)$-modules, then $\dim (W_3)_{[h_3]} =1$ and $A(W_1)\otimes_{A(\cL_p)}\CC v_{h_2}$ will be some quotient of $\CC[x]$ as an $A(\cL(p))\cong\CC[x]$-module. Thus the above proposition is simply the well-known fact that fusion rules for irreducible $\cV ir$-modules are never greater than $1$ (see in particular \cite[Lemma 2.20]{FZ2}).  
\end{rem}

\subsection{The singlet vertex operator algebras}

In this subsection, we discuss definitions and known results in the representation theory of the singlet vertex operator algebras, mainly using notation from \cite{CRW, CMY-singlet}; see also \cite{Ad} for the first systematic mathematical study of the singlet algebras. Fix an integer $p>1$ and set 
\begin{equation*}
 \alpha_+=\sqrt{2p},\qquad\alpha_-=-\sqrt{2/p}.
\end{equation*}
We define $L=\ZZ\alpha_+$; this is a rank-$1$ even lattice because 
\begin{equation*}
 \langle \alpha_+,\alpha_+\rangle=2p.
\end{equation*}
The dual lattice of $L$ is $L^\circ=\ZZ\frac{\alpha_-}{2}$. 

Viewing $\CC$ as an abelian Lie algebra with symmetric bilinear form $\langle\cdot,\cdot\rangle$, we have the associated rank-$1$ Heisenberg vertex algebra $\mathcal{H}$. We use the symbol $h$ to denote the basis vector $1$ of $\CC$, that is,  $\langle h,h\rangle=1$. Then $\mathcal{H}$ is generated by the degree-one vector $h(-1)\vac$, and the standard conformal vector of $\mathcal{H}$ is $\omega_{Std}=\frac{1}{2}h(-1)^2\vac$. However, we shall consider $\mathcal{H}$ as a vertex operator algebra with respect to the modified conformal vector
\begin{equation*}
 \omega=\frac{1}{2}h(-1)^2\vac+\frac{p-1}{\sqrt{2p}}h(-2)\vac =\omega_{Std}+\frac{\alpha_0}{2}h(-2)\vac,
\end{equation*}
where $\alpha_0=\alpha_++\alpha_-$. With respect to this conformal vector, the Virasoro central charge of $\mathcal{H}$ is $c_{p,1}=13-6p-6p^{-1}$, so the vertex operator subalgebra of $\mathcal{H}$ generated by $\omega$ is $\cL(p)$. 

The irreducible $\mathcal{H}$-modules consist of the Heisenberg Fock modules $\cF_\lambda$ for $\lambda\in\CC$, where the unique (up to scale) lowest conformal weight vector $v_\lambda$ that generates $\cF_\lambda$ satisfies
\begin{equation*}
 h(n)v_\lambda=\delta_{n,0}\lambda v_\lambda
\end{equation*}
for $n\geq 0$. The lowest conformal weight of $\cF_\lambda$ (with respect to $\omega$, not $\omega_{Std}$) is 
\begin{equation}\label{eqn:h_lambda_def}
 h_\lambda=\frac{1}{2}\lambda(\lambda-\alpha_0),
\end{equation}
rather than the usual $\frac{1}{2}\lambda^2$. The modification to $\omega$ also affects the $\mathcal{H}$-module structure of the contragredient $\cF_{\lambda}'$: we have
\begin{equation}\label{eqn:Heis_contras}
 \cF_\lambda'\cong\cF_{\alpha_0-\lambda}
\end{equation}
rather than the usual $\cF_{-\lambda}$.

We will need to focus particularly on the Heisenberg weights in $L^\circ$, that is, $\lambda\in\ZZ\frac{\alpha_-}{2}$. Any such $\lambda$ can be expressed as 
\begin{equation*}
 \alpha_{r,s}=\frac{1-r}{2}\alpha_++\frac{1-s}{2}\alpha_-
\end{equation*}
for certain $r,s\in\ZZ$. Thanks to the periodicity $\alpha_{r+1,s+p}=\alpha_{r,s}$, any $\lambda\in L^\circ$ is equal to a unique $\alpha_{r,s}$ with $1\leq s\leq p$. In particular, the Heisenberg weights $\lambda\in L=\ZZ\alpha_+$ have the form $\alpha_{2n+1,1}$ for $n\in\ZZ$. The minimal conformal weight \eqref{eqn:h_lambda_def} of the Fock module $\cF_{\alpha_{r,s}}$ is
\begin{equation*}
 h_{\alpha_{r,s}}=\frac{r^2-1}{4}p-\frac{rs-1}{2}+\frac{s^2-1}{4}p^{-1}.
\end{equation*}
For $r\geq 1$ and $1\leq s\leq p$, this is exactly the Virasoro conformal weight $h_{r,s}$ defined in \eqref{eqn:hrs_def}. For $r\leq 0$, symmetries of the Virasoro conformal weights show that $h_{\alpha_{r,s}}=h_{1-r,p-s}$. Also, \eqref{eqn:Heis_contras} specializes to
\begin{equation*}
 \cF_{\alpha_{r,s}}'\cong\cF_{\alpha_{-r,-s}} =\left\lbrace\begin{array}{lll}
                              \cF_{\alpha_{1-r,p-s}} & \text{if} & 1\leq s\leq p-1\\
                              \cF_{\alpha_{2-r,p}} & \text{if} & s=p
                                      \end{array}
\right. 
\end{equation*}
 for $r\in\ZZ$ and $1\leq s\leq p$.
 
 The Heisenberg vertex operator algebra $\mathcal{H}$ is not semisimple as a $\cV ir$-module. In fact, by \cite[Theorem 3.1]{Ad}, one way to define the singlet vertex operator algebra $\cM(p)$ for any integer $p>1$ is that $\cM(p)$ is the Virasoro socle of $\mathcal{H}$. See \cite{Ad} also for another definition of $\cM(p)$ as the kernel of a certain screening operator $\til{Q}: \mathcal{H}\rightarrow\cF_{\alpha_-}$. As a $\cV ir$-module,
 \begin{equation}\label{eqn:sing_decomp_Vir}
  \cM(p)\cong\bigoplus_{n= 0}^\infty\cL_{2n+1,1},
 \end{equation}
while as a vertex algebra, $\cM(p)$ is generated by $\omega$ together with a Virasoro singular vector $H$ of conformal weight $h_{3,1}=2p-1$ (see \cite[Theorem 3.2]{Ad}). That is, $H$ generates the $\cV ir$-submodule $\cL_{3,1}\subseteq\cM(p)$. We will use two notations for the modes of the vertex operator $Y_M(H,x)$ acting on any $\cM(p)$-module $M$:
\begin{equation*}
 Y_M(H,x) =\sum_{n\in\ZZ} H_n\,x^{-n-1}=\sum_{n\in\ZZ} H(n)\,x^{-n-2p+1}.
\end{equation*}
That is, $H(n)=H_{n+2p-2}$ is the mode of $Y_M(H,x)$ that lowers conformal weight by $n$.

Irreducible $\cM(p)$-modules were first classified in \cite{Ad}; see also \cite[Section 2]{CRW}. There is a one-to-one correspondence between irreducible $\cM(p)$-modules and Heisenberg Fock modules: For $\lambda\in\CC\setminus L^\circ$, $\cF_\lambda$ remains irreducible as an $\cM(p)$-module, while for $r\in\ZZ$, $1\leq s\leq p$, we define $\cM_{r,s}=\mathrm{Soc}(\cF_{\alpha_{r,s}})$. The $\cM(p)$-module $\cM_{r,s}$ is irreducible and:
\begin{itemize}
 \item For $s=p$, $\cM_{r,p}=\cF_{\alpha_{r,p}}$.
 \item For $1\leq s\leq p-1$, there is a non-split exact sequence
 \begin{equation}\label{eqn:Frs_structure}
  0\longrightarrow\cM_{r,s}\longrightarrow\cF_{\alpha_{r,s}}\longrightarrow\cM_{r+1,p-s}\longrightarrow 0.
 \end{equation}
\end{itemize}
The irreducible $\cM(p)$-modules which are Fock modules are the \textit{typical} irreducible $\cM(p)$-modules, while the $\cM(p)$-modules $\cM_{r,s}$ for $r\in\ZZ$, $1\leq s\leq p-1$ are \textit{atypical}.

For $\lambda\in\CC$, \eqref{eqn:Heis_contras} shows that $\cF_{\alpha_0-\lambda}$ is the $\cM(p)$-module contragredient of $\cF_\lambda$. In particular, $\cF_{\alpha_{r,p}}'\cong\cF_{\alpha_{2-r,p}}$ for $r\in\ZZ$. For $r\in\ZZ$ and $1\leq s\leq p-1$, dualizing \eqref{eqn:Frs_structure} shows that $\cM_{r,s}'$ is the unique irreducible quotient of $\cF_{\alpha_{r,s}}'\cong\cF_{\alpha_{1-r,p-s}}$. Then substituting $(r,s)\mapsto(1-r,p-s)$ in \eqref{eqn:Frs_structure} shows that
\begin{equation}\label{Mrs_contra}
 \cM_{r,s}'\cong\cM_{2-r,s}
\end{equation}
is valid for all $r\in\ZZ$ and $1\leq s\leq p$.

The $\cM(p)$-modules $\cM_{r,s}$ are semisimple as $\cV ir$-modules:
\begin{equation}\label{eqn:Mrs_decomp}
 \cM_{r,s}\cong\bigoplus_{n= 0}^\infty \cL_{r+2n,s} \cong\cM_{2-r,s}
\end{equation}
for $r\geq 1$, $1\leq s\leq p$. The isomorphism of $\cM_{r,s}$ and $\cM_{2-r,s}$ as $\cV ir$-modules is expected because these are a contragredient pair, and because semisimple $\cL(p)$-modules are self-contragredient; but $\cM_{r,s}$ and $\cM_{2-r,s}$ are \textit{not} isomorphic as $\cM(p)$-modules unless $r=1$. From \eqref{eqn:Mrs_decomp}, we see that the minimal conformal weight of $\cM_{r,s}$ is $h_{r,s}$ if $r\geq 1$ and $h_{2-r,s}$ if $r\leq 1$, and that $\cM(p)$ itself is identified with the atypical simple (and self-contragredient) module $\cM_{1,1}$. From now on, we will typically use the notation $\cM_{1,1}$ when we are considering $\cM(p)$ as a module for itself.

In \cite{CMY-singlet}, we found two categories of grading-restricted generalized $\cM(p)$-modules that admit the vertex algebraic braided tensor category structure of \cite{HLZ1}-\cite{HLZ8}. We also showed that these categories are rigid, so they are in particular braided ribbon tensor categories. The first category, denoted $\cC_{\cM(p)}$, is the category of all finite-length grading-restricted generalized $\cM(p)$-modules whose composition factors come from the $\cM_{r,s}$ for $r\in\ZZ$, $1\leq s\leq p$. This category is rather wild; for example, the irreducible modules $\cM_{r,s}$ do not have projective covers in $\cC_{\cM(p)}$.

To define the second braided tensor category of $\cM(p)$-modules, we need to recall the direct limit completion $\ind(\cO_p)$, constructed in \cite{CMY-completions}, of the category $\cO_p$ of $C_1$-cofinite grading-restricted generalized $\cL(p)$-modules, and we also need to recall the triplet vertex operator algebra $\cW(p)$. First, $\ind(\cO_p)$ is the category of generalized $\cL(p)$-modules which are the unions, equivalently the sums, of their $C_1$-cofinite $\cL(p)$-submodules; it has the vertex algebraic braided tensor category structure of \cite{HLZ1}-\cite{HLZ8} by \cite[Theorem 7.1]{CMY-completions}. The decomposition \eqref{eqn:sing_decomp_Vir} shows that $\cM(p)$, as an $\cL(p)$-module, is an object of $\ind(\cO_p)$, and thus $\cM(p)$ is a commutative algebra in the braided tensor category $\ind(\cO_p)$ by \cite[Theorem 3.2]{HKL}, or more precisely \cite[Theorem 7.5]{CMY-completions}. Consequently, as in \cite{KO, CKM-exts}, we have a tensor category $\rep\cM(p)$ of not-necessarily-local $\cM(p)$-modules which, as $\cL(p)$-modules, are objects of $\ind(\cO_p)$. The subcategory $\rep^0\cM(p)$ of (local) generalized $\cM(p)$-modules which, as $\cL(p)$-modules, are objects of $\ind(\cO_p)$ is a vertex algebraic braided tensor category by \cite[Theorem 3.65]{CKM-exts}, or more precisely \cite[Theorem 7.7]{CMY-completions}. 

The tensor categories $\cO_p$ and $\rep\cM(p)$ are related by the induction functor
\begin{align*}
 \cF_{\cM(p)}: \cO_p & \rightarrow \rep\cM(p)\nonumber\\
  W & \mapsto \cM(p)\boxtimes W\nonumber\\
  f & \mapsto \Id_{\cM(p)}\boxtimes f
\end{align*}
where $\tens$ denotes the tensor product bifunctor on $\ind(\cO_p)$. Induction is a monoidal functor by \cite[Theorem 1.6(3)]{KO} or \cite[Theorem 2.59]{CKM-exts}. Induction is also exact because $\cO_p$ is a rigid tensor category (see the proof of \cite[Proposition 3.2.4]{CMY-singlet}). Further, induction satisfies Frobenius reciprocity, that is, there is a natural isomorphism
\begin{equation*}
 \hom_{\cL(p)}(W, M) \xrightarrow{\cong}\hom_{\cM(p)}(\cF_{\cM(p)}(W),M)
\end{equation*}
for grading-restricted generalized modules $\cL(p)$-modules $W$ in $\cO_p$ and generalized $\cM(p)$-modules $M$ in $\rep\cM(p)$.

Now the triplet vertex operator algebra $\cW(p)$ is a $C_2$-cofinite but non-rational simple current extension of $\cM(p)$. Triplet vertex operator algebras have been studied extensively (see for example \cite{FHST, FGST1, FGST2, CF, GR, AM-trip, AM-log-mods, NT, TW, MY}, but we will not need too much from the representation theory of $\cW(p)$ here.  By \cite[Theorem 1.1]{AM-trip},
\begin{equation}\label{eqn:trip_decomp_Vir}
 \cW(p)\cong\bigoplus_{n\geq 0} \,(2n+1)\cdot\cL_{2n+1,1}
\end{equation}
as an $\cL(p)$-module, while
\begin{equation}\label{eqn:trip_decomp_sing}
 \cW(p)\cong\bigoplus_{n\in\ZZ} \cM_{2n+1,1}
\end{equation}
as an $\cM(p)$-module (see for example \cite[Section 3.2]{CMY-singlet}). 

By \eqref{eqn:trip_decomp_Vir}, $\cW(p)$ as an $\cL(p)$-module is an object of $\ind(\cO_p)$, and thus $\cW(p)$ as an $\cM(p)$-module is an object of $\rep^0\cM(p)$. Thus $\cW(p)$ is a commutative algebra in $\rep^0\cM(p)$, there is a tensor category $\rep\cW(p)$ of not-necessarily-local $\cW(p)$-modules which, as $\cM(p)$-modules, are objects of $\rep^0\cM(p)$, and we have the exact monoidal induction functor
\begin{align*}
 \cF_{\cW(p)}: \rep^0\cM(p) &\rightarrow \rep\cW(p)\nonumber\\
  M & \mapsto \cW(p)\tens M\nonumber\\
  f & \mapsto \Id_{\cW(p)}\tens f
\end{align*}
where $\tens$ now denotes the tensor product bifunctor on $\rep^0\cM(p)$. In \cite{CMY-singlet}, we showed that every grading-restricted generalized $\cW(p)$-module is an object of $\rep\cW(p)$.

Now we define, as in \cite{CMY-singlet}, the category $\cC_{\cM(p)}^0$ of $\cM(p)$-modules to be the full subcategory of $\rep^0\cM(p)$ whose objects induce to grading-restricted generalized $\cW(p)$-modules. By \cite[Theorem 3.3.1]{CMY-singlet}, $\cC_{\cM(p)}^0$ is a tensor subcategory of $\cC_{\cM(p)}$; in particular, $\cC_{\cM(p)}^0$ is closed under submodules, quotients, and tensor products, and every object of $\cC_{\cM(p)}^0$ is a finite-length grading-restricted generalized $\cM(p)$-module. The irreducible $\cM(p)$-modules $\cM_{r,s}$ for $r\in\ZZ$, $1\leq s\leq p$ are objects of $\cC_{\cM(p)}^0$ (see \cite[Proposition 3.2.5]{CMY-singlet}), and each $\cM_{r,s}$ has a projective cover $\cP_{r,s}$ in $\cC_{\cM(p)}^0$, although not in $\cC_{\cM(p)}$ (see \cite[Section 5.1]{CMY-singlet}):
\begin{itemize}
 \item For $s=p$, $\cP_{r,p} =\cM_{r,p}$ for all $r\in\ZZ$ (recall that this is also the Fock module $\cF_{\alpha_{r,p}}$).
 
 \item For $1\leq s\leq p-1$, $\cP_{r,s}$ is a length-$4$ indecomposable $\cM(p)$-module with Loewy diagram
 \begin{equation}\label{eqn:Prs_Loewy_diag}
 \begin{matrix}
  \begin{tikzpicture}[->,>=latex,scale=1.5]
\node (b1) at (1,0) {$\cM_{r, s}$};
\node (c1) at (-1, 1){$\cP_{r, s}$:};
   \node (a1) at (0,1) {$\cM_{r-1, p-s}$};
   \node (b2) at (2,1) {$\cM_{r+1, p-s}$};
    \node (a2) at (1,2) {$\cM_{r,s}$};
\draw[] (b1) -- node[left] {} (a1);
   \draw[] (b1) -- node[left] {} (b2);
    \draw[] (a1) -- node[left] {} (a2);
    \draw[] (b2) -- node[left] {} (a2);
\end{tikzpicture}
\end{matrix} .
 \end{equation}
 The rows of the Loewy diagram indicate the socle series of $\cP_{r,s}$, so in particular $\mathrm{Soc}(\cP_{r,s})\cong\cM_{r,s}$ and $\mathrm{Soc}(\cP_{r,s}/\cM_{r,s})\cong\cM_{r-1,p-s}\oplus\cM_{r+1,p-s}$. An arrow between two nodes of the diagram indicates that the corresponding length-$2$ subquotient of $\cP_{r,s}$ is indecomposable, with the arrow pointing towards the quotient of the length-$2$ subquotient.
\end{itemize}

\begin{rem}
 From the projectivity of $\cP_{r,s}$ in $\cC_{\cM(p)}^0$, the Loewy diagram \eqref{eqn:Prs_Loewy_diag}, and the surjection $\cF_{\alpha_{r-1,p-s}}\rightarrow\cM_{r,s}$ in $\cC_{\cM(p)}^0$, it is easy to see that there is a non-split exact sequence
 \begin{equation*}
  0\longrightarrow\cF_{\alpha_{r,s}}\longrightarrow\cP_{r,s}\longrightarrow\cF_{\alpha_{r-1,p-s}}\longrightarrow 0
 \end{equation*}
of $\cM(p)$-modules.
\end{rem}

In the next section, we will need the $\cM(p)$-module inductions of simple $\cL(p)$-modules in $\cO_p$; a couple cases of the following proposition were also obtained in \cite[Lemma 11.4]{GN}:
\begin{prop}\label{prop:Vir_to_sing_ind}
 For $r\geq 1$ and $1\leq s\leq p$,
 \begin{equation*}
  \cF_{\cM(p)}(\cL_{r,s})\cong\bigoplus_{k=1}^r \cM_{2k-r,s}.
 \end{equation*}
\end{prop}
\begin{proof}
 From \eqref{eqn:Mrs_decomp}, there is a non-zero (and one-dimensional) space of $\cL(p)$-module homomorphisms $\cL_{r,s}\rightarrow\cM_{r',s'}$ if and only if $s'=s$ and for some $n\in\NN$,
 \begin{equation*}
  r'=\left\lbrace\begin{array}{lll}
                  r-2n & \text{if} & r'\geq 1\\
                  2-r+2n & \text{if} & r'<1\\
                 \end{array}
\right. .
 \end{equation*}
Thus by Frobenius reciprocity, there is a non-zero (necessarily surjective) $\cM(p)$-module homomorphism $f_k: \cF_{\cM(p)}(\cL_{r,s})\rightarrow\cM_{2k-r,s}$ for each $k\in\lbrace 1,\ldots, r\rbrace$. We combine these maps into a homomorphism
\begin{equation*}
 F=\sum_{k=1}^r q_k\circ f_k : \cF_{\cM(p)}(\cL_{r,s})\longrightarrow\bigoplus_{k=1}^r\cM_{2k-r,s},
\end{equation*}
where $q_k$ is the inclusion of $\cM_{2k-r,s}$ into the direct sum. Because the $\cM_{2k-r,s}$ are non-isomorphic simple $\cM(p)$-modules and each $f_k$ is surjective, $F$ is surjective as well.

To show that $F$ is also injective and thus an isomorphism, it is enough to show that $\cF_{\cM(p)}(\cL_{r,s})$ and $\bigoplus_{k=1}^r\cM_{2k-r,s}$ are isomorphic as $\cV ir$-modules, since this will show they are isomorphic as graded vector spaces with finite-dimensional homogeneous subspaces. In fact, \eqref{eqn:sing_decomp_Vir} and the Virasoro fusion rules in \cite[Theorem 4.3]{MY} show that as $\cL(p)$-modules,
\begin{align*}
 \cF_{\cM(p)}(\cL_{r,s}) & \cong\bigoplus_{n= 0}^\infty \cL_{2n+1,1}\boxtimes\cL_{r,s}\cong\bigoplus_{n= 0}^\infty\bigoplus_{\substack{k=\vert 2n+1-r\vert+1\\ k\equiv r\;\mathrm{mod}\; 2}}^{r+2n} \cL_{k,s}.
\end{align*}
We need to determine the multiplicity of $\cL_{r+2m,s}$ in this sum for $m\geq-\frac{r-1}{2}$: $\cL_{r+2m,s}$ appears once for each $n\geq 0$ such that
\begin{equation*}
 \vert 2n+1-r\vert+1\leq r+2m\leq r+2n.
\end{equation*}
Rearranging these inequalities, we see that $\cL_{r+2m,s}$ occurs once for each $n\geq 0$ such that
\begin{equation*}
 \vert m\vert \leq n\leq m+r-1.
\end{equation*}
From this, we can see that as $\cL(p)$-modules,
\begin{align*}
 \cF_{\cM(p)}(\cL_{r,s})&\cong\bigoplus_{-\frac{r-1}{2}\leq m<0} (r-2\vert m\vert)\cdot\cL_{r+2m,s}\oplus\bigoplus_{m=0}^\infty r\cdot\cL_{r+2m,s}\nonumber\\
 &\cong\left\lbrace \begin{array}{lll}
 \bigoplus_{k=0}^{\frac{r-2}{2}} 2\cdot\bigoplus_{m=0}^\infty \cL_{r-2k+2m,s} & \text{if} & r\equiv 0\;\mathrm{mod}\;2\\
\bigoplus_{m=0}^\infty \cL_{2m+1,s}\oplus\bigoplus_{k=0}^{\frac{r-3}{2}} 2\cdot\bigoplus_{m=0}^\infty\cL_{r-2k+2m,s} & \text{if} & r\equiv 1\;\mathrm{mod}\;2\\
                   \end{array}
\right.\nonumber\\
& \cong\bigoplus_{k=1}^r \cM_{2k-r,s},
\end{align*}
where the last isomorphism follows from \eqref{eqn:Mrs_decomp}. This proves the proposition.
\end{proof}

We will also need a criterion for determining when generalized $\cM(p)$-modules are objects of $\cC_{\cM(p)}^0$. Note that \cite[Proposition 2.65]{CKM-exts} shows that an object $M$ of $\rep^0\cM(p)$ is an object of $\cC_{\cM(p)}^0$ if and only if the monodromy, or double braiding, of $\cW(p)$ with $M$ is trivial. By \eqref{eqn:trip_decomp_sing}, this is equivalent to 
\begin{equation*}
 \cR^2_{\cM_{2n+1,1},M}:=\cR_{M,\cM_{2n+1,1}}\circ\cR_{\cM_{2n+1,1},M} =\Id_{\cM_{2n+1,1}\tens M}
\end{equation*}
for all $n\in\ZZ$. But since $\cM_{3,1}$ generates the group of simple currents $\lbrace\cM_{2n+1,1}\rbrace_{n\in\ZZ}$ by \cite[Lemma 3.2.1]{CMY-singlet}, it is actually necessary and sufficient that
\begin{equation*}
 \cR^2_{\cM_{3,1},M}=\Id_{\cM_{3,1},M}
\end{equation*}
(see \cite[Theorem 2.11(2)]{CKL}).

\begin{lem}\label{lem:untwisted_induction}
 Suppose we have a short exact sequence
 \begin{equation*}
  0\longrightarrow A\longrightarrow M\longrightarrow B\longrightarrow 0
 \end{equation*}
of generalized $\cM(p)$-modules such that $A$ and $B$ are objects of $\cC_{\cM(p)}^0$ having no composition factors in common. Then $M$ is also an object of $\cC_{\cM(p)}^0$.
\end{lem}
\begin{proof}
 Since $A$ and $B$ are (finite-length) objects of $\cC_{\cM(p)}^0$, $M$ is a finite-length $\cM(p)$-module in $\cC_{\cM(p)}$. Thus $M$ is also an object of $\rep^0\cM(p)$ by \cite[Proposition 3.1.3]{CMY-singlet}. By the discussion preceding the proposition, it is now enough to show that $\cR^2_{\cM_{3,1},M}=\Id_{\cM_{3,1},M}$.

 We will show that if $\cR_{\cM_{3,1},M}^2\neq\Id_{\cM_{3,1},M}$, then $A$ and $B$ must share a common composition factor. Since $A$ and $B$ are objects of $\cC_{\cM(p)}^0$, since the tensoring functor $\cM_{3,1}\boxtimes\bullet$ is exact, and since the monodromy isomorphisms are natural, we have a commuting diagram
 \begin{equation*}
\begin{matrix}
  \xymatrix{
0\ar[r] &  \cM_{3,1}\boxtimes A \ar[r] \ar[d]^{\Id_{\cM_{3,1}\boxtimes A}} & \cM_{3,1}\boxtimes M \ar[r] \ar[d]^{\cR^2_{\cM_{3,1},M}} &\cM_{3,1}\boxtimes B \ar[d]^{\Id_{\cM_{3,1}\boxtimes B}} \ar[r] & 0\\
0\ar[r] & \cM_{3,1}\boxtimes A \ar[r]  & \cM_{3,1}\boxtimes M \ar[r] & \cM_{3,1}\boxtimes B\ar[r] & 0\\
  } 
  \end{matrix}
 \end{equation*}
with exact rows. Thus $\cM_{3,1}\tens A$ is a submodule of $\cM_{3,1}\tens M$. Moreover,$\mathcal{N}:=\cR_{\cM_{3,1},M}^2-\Id_{\cM_{3,1}\boxtimes M}$ is an $\cM(p)$-module endomorphism of $\cM_{3,1}\boxtimes M$ such that
\begin{equation*}
 \im\mathcal{N}\subseteq \cM_{3,1}\boxtimes A\subseteq\ker\mathcal{N}.
\end{equation*}
Thus if $\mathcal{N}\neq 0$, there is a non-zero homomorphism
\begin{align*}
 \cM_{3,1}\boxtimes B & \xrightarrow{\cong} (\cM_{3,1}\boxtimes M)/(\mathcal{M}_{3,1}\boxtimes A)\nonumber\\
 &\twoheadrightarrow (\cM_{3,1}\boxtimes M)/\ker\mathcal{N}\xrightarrow{\cong}\im\mathcal{N}\nonumber\\
 &\hookrightarrow \cM_{3,1}\boxtimes A.
\end{align*}
Since this composition is non-zero, any irreducible submodule of its image (in $\cM_{3,1}\boxtimes A$) occurs as a subquotient of $\cM_{3,1}\boxtimes B$, showing that $\cM_{3,1}\boxtimes A$ and $\cM_{3,1}\boxtimes B$ share a common composition factor.

Now because $\cM_{3,1}$ is a simple current $\cM(p)$-module with inverse $\cM_{0,1}$, the composition factors of $A$, respectively $B$, are obtained by tensoring the composition factors of $\cM_{3,1}\boxtimes A$, respectively $\cM_{3,1}\boxtimes B$, with $\cM_{0,1}$ (see for example \cite[Proposition 2.5(5)]{CKLR}). Thus $A$ and $B$ also share a common composition factor.
\end{proof}

\section{Tensor categories of \texorpdfstring{$\cM(p)$}{M(p)}-modules}\label{sec:tens_cats}

In this section we prove that the category $\cO_{\cM(p)}$ of $C_1$-cofinite grading-restricted generalized $\cM(p)$-modules has the vertex algebraic braided tensor category structure of \cite{HLZ1}-\cite{HLZ8}. By \cite[Theorem 13]{CMR}, all irreducible $\cM(p)$-modules are $C_1$-cofinite, so \cite[Theorem 3.3.4]{CY} implies that $\cO_{\cM(p)}$ is a
braided tensor category if it is equal to the category of finite-length generalized $\cM(p)$-modules. For this, \cite[Theorem 3.3.5]{CY} implies that it is enough to show that the generalized Verma $\cM(p)$-module induced from any finite-dimensional irreducible $A(\cM(p))$-module has finite length. We prove this next.

\subsection{Generalized Verma \texorpdfstring{$\cM(p)$}{M(p)}-modules}\label{sec:genVerma}

We first show that the generalized Verma $\cM(p)$-module $\cG_\lambda$ induced from the simple $A(\cM(p))$-module $T(\cF_\lambda)$, $\lambda\in\CC\setminus L^\circ$, is simple; the argument is based on the proof of \cite[Theorem 4.4]{AM-trip}. Note that $T(\cF_\lambda)=\CC v_\lambda$, where $v_\lambda$ is a generating vector in $\cF_\lambda$ of minimal conformal weight $h_\lambda$.

\begin{thm}\label{thm:typical_GVM}
 For $\lambda\in\CC\setminus L^\circ$, the generalized Verma $\cM(p)$-module $\cG_\lambda$ is isomorphic to the typical irreducible Fock module $\cF_\lambda$.
\end{thm}
\begin{proof}
First observe that as a $\cV ir$-module, $\cF_\lambda$ is isomorphic to the irreducible Verma module $\cV_{h_\lambda}$: Since $\cF_\lambda$ contains a Virasoro singular vector of conformal weight $h_\lambda$, there is a non-zero $\cV ir$-module homomorphism $\cV_{h_\lambda}\rightarrow\cF_\lambda$. This map is injective because $\cV_{h_\lambda}$ is irreducible for $h_\lambda\neq h_{r,s}$, and then it is also surjective because $\cV_{h_\lambda}$ and $\cF_\lambda$ have the same graded dimension.

Now by the universal property of generalized Verma $\cM(p)$-modules, there is a surjection $\cG_\lambda\rightarrow\cF_\lambda$ and thus a short exact sequence
\begin{equation*}
 0\longrightarrow K\longrightarrow \cG_\lambda\longrightarrow \cF_\lambda\longrightarrow 0
\end{equation*}
of $\cM(p)$-modules, with $K$ denoting the kernel. Because $\cG_\lambda$ contains a Virasoro singular vector of conformal weight $h_\lambda$ and $\cF_\lambda\cong\cV_{h_\lambda}$ as $\cV ir$-modules, this exact sequence splits when considered as a sequence of $\cV ir$-module homomorphisms. That is, $\cF_\lambda$ occurs as a $\cV ir$-submodule of $\cG_\lambda$ and thus also as a $\cV ir$-module direct summand. So there is a $\cV ir$-module projection $\pi: \cG_\lambda\rightarrow K$.

Recall that the  Virasoro singular vector $H\in\cM(p)$ generates the $\cV ir$-submodule $\cL_{3,1}\subseteq\cM(p)$. Thus we may consider the $\cL(p)$-module intertwining operator
\begin{equation*}
 \cY=\pi\circ Y_{\cG_\lambda}\vert_{\cL_{3,1}\otimes\cF_\lambda}
\end{equation*}
of type $\binom{K}{\cL_{3,1}\,\cF_\lambda}$. If the $\cV ir$-submodule $\im\cY\subseteq K$ is non-zero, its minimal conformal weight is one of
\begin{equation*}
 h_\lambda,\quad h_\lambda+p\pm\sqrt{4ph_\lambda+(p-1)^2}
\end{equation*}
by Corollary \ref{cor:conf_wts}(2). Since the surjection $\cG_\lambda\rightarrow\cF_\lambda$ is an isomorphism on top levels, the conformal weights of $\im\cY\subseteq K$ are contained in $h_\lambda+\ZZ_+$, so that $h_\lambda$ cannot be the minimal conformal weight of $\im\cY$. The remaining two options are also impossible because
\begin{equation*}
p\pm\sqrt{4ph_\lambda+(p-1)^2}\in\ZZ\longleftrightarrow h_\lambda=\frac{n^2-(p-1)^2}{4p}
\end{equation*}
for some $n\in\ZZ$, that is, $h_\lambda=h_{r,s}$ for some $r\in\ZZ$, $1\leq s\leq p$. We conclude that $\im\cY$ has no minimal conformal weight, that is, $\im\cY=0$.

We have now shown that $Y_{\cG_\lambda}(H,x)$ preserves the $\cV ir$-submodule $\cF_\lambda\subseteq\cG_\lambda$. Since $\omega$ and $H$ generate $\cM(p)$, this shows that $\cF_\lambda$ is actually the $\cM(p)$-submodule of $\cG_\lambda$ generated by $v_\lambda$. Since $v_\lambda$ generates $\cG_\lambda$ as an $\cM(p)$-module, it follows that $\cG_\lambda =\cF_\lambda$.
\end{proof}

Now for $r\in\ZZ$ and $1\leq s\leq p$, we use $\cG_{r,s}$ to denote the generalized Verma $\cM(p)$-module induced from $T(\cM_{r,s})$. To determine $\cG_{r,s}$, we first identify candidates obtained as quotients of the projective covers $\cP_{r,s}$ in $\cC_{\cM(p)}^0$: we define $\til{\cG}_{r,s}$ to be the maximal quotient of $\cP_{r,s}$ having a one-dimensional top level isomorphic to $T(\cM_{r,s})$. More specifically:
\begin{itemize}
 \item If $s=p$, then $\til{\cG}_{r,p} =\cM_{r,p}\,\,(=\cF_{\alpha_{r,p}}=\cP_{r,p})$.
 
 \item If $1\leq s\leq p-1$, then we deduce from \eqref{eqn:Prs_Loewy_diag} and the lowest conformal weights of irreducible $\cM(p)$-modules that there is a non-split short exact sequence
 \begin{equation}\label{eqn:Gtil_structure}
  0\longrightarrow \mathrm{Soc}(\til{\cG}_{r,s})\longrightarrow\til{\cG}_{r,s}\longrightarrow\cM_{r,s}\longrightarrow 0
 \end{equation}
with
\begin{equation}\label{eqn:Gtil_socle}
 \mathrm{Soc}(\til{\cG}_{r,s}) \cong \left\lbrace\begin{array}{lll}
                                \cM_{r+1,p-s} & \text{if} & r>1\\
                                 \cM_{0,s}\oplus\cM_{2,s}& \text{if} & r=1\\
                                \cM_{r-1,p-s} & \text{if} & r<1
                               \end{array}
\right. .
\end{equation}
Note that $\til{\cG}_{r,s}\cong\cF_{\alpha_{r-1,p-s}}$ if $r<1$ and $1\leq s\leq p-1$.
\end{itemize}
For any $r$ and $s$, the universal property of generalized Verma $\cM(p)$-modules yields a surjection $\pi_{r,s}: \cG_{r,s}\rightarrow\til{\cG}_{r,s}$. Our goal is to show that  $\pi_{r,s}$ is actually an isomorphism.

To handle the cases $r\geq 1$ and $r\leq 1$ simultaneously, we fix $r\geq 1$ and use $r'$ to denote either $r$ or $2-r$. Recall that $\cM_{r,s}$ and $\cM_{2-r,s}$ are isomorphic as $\cV ir$-modules and in particular have the same lowest conformal weight $h_{r,s}$. 
\begin{lem}\label{lem:Gtil_Vir_submod}
 Let $\til{v}_{r,s}$ be a generating vector for $\til{\cG}_{r',s}$ of conformal weight $h_{r,s}$. Then the $\cV ir$-submodule of $\til{\cG}_{r',s}$ generated by $\til{v}_{r,s}$ is isomorphic to $\cV_{r,s}/\cV_{r+2,s}$.
\end{lem}
\begin{proof}
 Since $\til{v}_{r,s}$ is a Virasoro singular vector of conformal weight $h_{r,s}$, the $\cV ir$-submodule it generates is a quotient of $\cV_{r,s}$. When $s=p$, the $\cV ir$-module structure \eqref{eqn:Mrs_decomp} of $\til{\cG}_{r',p}=\cM_{r',p}$ shows that $\til{v}_{r,p}$ generates a $\cV ir$-module isomorphic to $\cL_{r,p}$. Then by Proposition \ref{prop:Verma_structure}(2), $\cL_{r,p}\cong\cV_{r,p}/\cV_{r+2,p}$ as required. 
 
 When $1\leq s\leq p-1$, the embedding diagram \eqref{eqn:embedding_rs} shows that the first two non-trivial singular vectors in $\cV_{r,s}$ have the form
 \begin{equation*}
  v_{r+1,p-s} =\sigma_{r+1,p-s}\cdot v_{r,s}\quad\text{and}\quad v_{r+2,s}=\sigma_{r+2,s}\cdot v_{r+1,p-s}
 \end{equation*}
for suitable $\sigma_{r+1,p-s},\sigma_{r+2,s}\in U(\cV ir_-)$, where the three singular vectors $v_{r,s}$, $v_{r+1,p-s}$, and $v_{r+2,s}$ have conformal weights $h_{r,s}$, $h_{r+1,p-s}$, and $h_{r+2,s}$, respectively. Thus to prove the lemma when $1\leq s\leq p-1$, we need to show
\begin{equation}\label{eqn:nonzero_sing_vect}
 \sigma_{r+1,p-s}\cdot \til{v}_{r,s}\neq 0
\end{equation}
while
\begin{equation}\label{eqn:zero_sing_vect}
\sigma_{r+2,s} \sigma_{r+1,p-s}\cdot \til{v}_{r,s}= 0.
\end{equation}

To prove \eqref{eqn:nonzero_sing_vect}, suppose to the contrary that $\sigma_{r+1,p-s}\cdot\til{v}_{r,s}=0$; this would imply that $\til{v}_{r,s}$ generates a $\cV ir$-submodule of $\til{\cG}_{r',s}$ isomorphic to $\cL_{r,s}$. Then Frobenius reciprocity applied to the inclusion $\cL_{r,s}\hookrightarrow\til{\cG}_{r',s}$ yields a non-zero $\cM(p)$-module homomorphism
\begin{equation*}
 \cF_{\cM(p)}(\cL_{r,s})\rightarrow\til{\cG}_{r',s}.
\end{equation*}
Since $\cF_{\cM(p)}(\cL_{r,s})\cong\bigoplus_{k=1}^r \cM_{2k-r,s}$ by Proposition \ref{prop:Vir_to_sing_ind}, we get an $\cM(p)$-module injection $\cM_{2k-r,s}\hookrightarrow\til{\cG}_{r',s}$ for at least one $k\in\lbrace 1,2,\ldots, r\rbrace$. But since the $\cM(p)$-module exact sequence \eqref{eqn:Gtil_structure} does not split, no $\cM_{2k-r,s}$ occurs as a submodule of $\til{\cG}_{r',s}$. Thus $\cL_{r,s}$ cannot occur as a $\cV ir$-submodule of $\til{\cG}_{r',s}$, proving \eqref{eqn:nonzero_sing_vect}.

To prove \eqref{eqn:zero_sing_vect}, note from \eqref{eqn:Mrs_decomp} that the $\cV ir$-submodule of $\til{\cG}_{r',s}/\mathrm{Soc}(\til{\cG}_{r',s})\cong\cM_{r',s}$ generated by $\til{v}_{r,s}+\mathrm{Soc}(\til{\cG}_{r',s})$ is isomorphic to $\cL_{r,s}$. Thus
\begin{equation*}
 \sigma_{r+1,p-s}\cdot\til{v}_{r,s}\in\mathrm{Soc}(\til{\cG}_{r',s}),
\end{equation*}
and this is a Virasoro singular vector of conformal weight $h_{r+1,p-s}$. Since \eqref{eqn:Gtil_socle} and \eqref{eqn:Mrs_decomp} show that $\mathrm{Soc}(\til{\cG}_{r',s})$ is semisimple as a $\cV ir$-module, $\sigma_{r+1,p-s}\cdot\til{v}_{r,s}$ generates a $\cV ir$-submodule isomorphic to $\cL_{r+1,p-s}$ which in particular does not contain any singular vector of conformal weight $h_{r+2,s}$. Thus \eqref{eqn:zero_sing_vect} holds, completing the proof of the lemma.
\end{proof}

Now we begin to study $\cG_{r',s}$. We use $v_{r,s}$ to denote a non-zero vector in the top level $T(\cG_{r',s})$. It is a Virasoro singular vector of conformal weight $h_{r,s}$ and thus generates a $\cV ir$-submodule $\til{\cV}_{r,s}$ which is a quotient of the Verma module $\cV_{r,s}$. From the embedding diagrams \eqref{eqn:embedding_rs} and \eqref{eqn:embedding_rp}, $\cV_{r,s}$ contains a singular vector of conformal weight $h_{r+2,s}$, and we use $v_{r+2,s}$ to denote its image in $\til{\cV}_{r,s}$. Note that $v_{r+2,s}$ could possibly be $0$ (and in fact, we will show that it is $0$). Let $K$ denote the kernel of $\pi_{r',s}:\cG_{r',s}\rightarrow\til{\cG}_{r',s}$, so that we have a short exact sequence
\begin{equation}\label{eqn:K_exact_seq}
 0\longrightarrow K\longrightarrow\cG_{r',s}\longrightarrow\til{\cG}_{r',s}\longrightarrow 0
\end{equation}
of $\cM(p)$-modules.
\begin{lem}\label{lem:tilVrs_cap_K}
 The $\cV ir$-submodule $\til{\cV}_{r,s}\cap K$ in $\cG_{r',s}$ is generated by $v_{r+2,s}$.
\end{lem}
\begin{proof}
The commutative diagram of $\cV ir$-module surjections
\begin{equation*}
\xymatrixcolsep{4pc}
 \xymatrix{
 \cV_{r,s} \ar[rd] \ar[d] & \\
 \til{\cV}_{r,s} \ar[r]_(.4){\pi_{r',s}\vert_{\til{\cV}_{r,s}}} & U(\cV ir)\cdot\til{v}_{r,s}
 }
\end{equation*}
combined with Lemma \ref{lem:Gtil_Vir_submod} shows that $\til{\cV}_{r,s}\cap K$ is the image in $\til{\cV}_{r,s}$ of the Verma submodule $\cV_{r+2,s}\subseteq\cV_{r,s}$. Since this Verma submodule is generated by a preimage of $v_{r+2,s}$, it follows that $v_{r+2,s}$ generates $\til{\cV}_{r,s}\cap K$ as a $\cV ir$-module.
\end{proof}

The crucial next lemma considerably strengthens the preceding lemma:
\begin{lem}\label{lem:K_generator}
 The kernel $K$ is generated by $v_{r+2,s}$ as an $\cM(p)$-module.
\end{lem}
\begin{proof}
The preceding lemma shows that $v_{r+2,s}\in K$, so the $\cM(p)$-submodule $\langle v_{r+2,s}\rangle$ generated by $v_{r+2,s}$ is contained in $K$. Then to prove the reverse inclusion, it is enough to show that the $\cM(p)$-module surjection $\cG_{r',s}/\langle v_{r+2,s}\rangle\twoheadrightarrow\til{\cG}_{r',s}$ induced by $\pi_{r',s}$ is an isomorphism.

We first claim that $\cG_{r',s}/\langle v_{r+2,s}\rangle$, as an $\cL(p)$-module, is an object of the direct limit completion $\ind(\cO_p)$, that is, $\cG_{r',s}/\langle v_{r+2,s}\rangle$ is the sum of $C_1$-cofinite $\cL(p)$-submodules. To prove this, note that the preceding lemma implies 
\begin{equation*}
 \til{\cV}_{r,s}\cap\langle v_{r+2,s}\rangle = \til{\cV}_{r,s}\cap K,
\end{equation*}
so that as $\cV ir$-modules,
\begin{align*}
 (\til{\cV}_{r,s}+\langle v_{r+2,s}\rangle)/\langle v_{r+2,s}\rangle &\cong\til{\cV}_{r,s}/(\til{\cV}_{r,s}\cap\langle v_{r+2,s}\rangle)\nonumber\\ &=\til{\cV}_{r,s}/(\til{\cV}_{r,s}\cap K)\cong\cV_{r,s}/\cV_{r+2,s},
\end{align*}
using Lemma \ref{lem:Gtil_Vir_submod} for the last isomorphism. This shows $\cV_{r,s}/\cV_{r+2,s}$ is a ($C_1$-cofinite) $\cL(p)$-submodule of $\cG_{r',s}/\langle v_{r+2,s}\rangle$, and it generates $\cG_{r',s}/\langle v_{r+2,s}\rangle$ as an $\cM(p)$-module since it contains $v_{r,s}+\langle v_{r+2,s}\rangle$. Thus since $\cM(p)\cong\bigoplus_{n= 0}^\infty \cL_{2n+1,1}$ as an $\cL(p)$-module, we have
\begin{align*}
 \cG_{r'.s}/\langle v_{r+2,s}\rangle & =\im Y_{\cG_{r',s}/\langle v_{r+2,s}\rangle}\vert_{\cM(p)\otimes(\cV_{r,s}/\cV_{r+2,s})}\nonumber\\
 & =\sum_{n=0}^\infty \im Y_{\cG_{r',s}/\langle v_{r+2,s}\rangle}\vert_{\cL_{2n+1,1}\otimes(\cV_{r,s}/\cV_{r+2,s})}.
\end{align*}
Each $\im Y_{\cG_{r',s}/\langle v_{r+2,s}\rangle}\vert_{\cL_{2n+1,1}\otimes(\cV_{r,s}/\cV_{r+2,s})}$ is a $C_1$-cofinite $\cL(p)$-module by \cite[Key Theorem]{Mi}, so $\cG_{r',s}/\langle v_{r+2,s}\rangle$, as an $\cL(p)$-module, is the sum of $\cO_p$-submodules.

We now know that $\cG_{r',s}/\langle v_{r+2,s}\rangle$ is an object in the braided tensor category $\rep^0\cM(p)$ of generalized $\cM(p)$-modules that, as $\cL(p)$-modules, are objects of $\ind(\cO_p)$. Thus the $\cL(p)$-module inclusion $\cV_{r,s}/\cV_{r+2,s}\hookrightarrow\cG_{r',s}/\langle v_{r+2,s}\rangle$ induces, by Frobenius reciprocity, an $\cM(p)$-module homomorphism
\begin{equation*}
 \cF_{\cM(p)}(\cV_{r,s}/\cV_{r+2,s})\longrightarrow\cG_{r',s}/\langle v_{r+2,s}\rangle;
\end{equation*}
this map is surjective because its image contains the generating vector $v_{r,s}+\langle v_{r+2,s}\rangle$. Thus because the $\cM(p)$-module category $\cC_{\cM(p)}^0$ is closed under quotients, $\cG_{r',s}/\langle v_{r+2,s}\rangle$ will be an object of $\cC_{\cM(p)}^0$ if the same holds for $\cF_{\cM(p)}(\cV_{r,s}/\cV_{r+2,s})$.

For $s=p$, $\cF_{\cM(p)}(\cV_{r,p}/\cV_{r+2,p})\cong\cF_{\cM(p)}(\cL_{r,p})$
 is an object of $\cC^0_{\cM(p)}$ by Proposition \ref{prop:Vir_to_sing_ind}. For $1\leq s\leq p-1$, the $\cV ir$-module exact sequence
\begin{equation*}
 0\longrightarrow \cL_{r+1,p-s}\longrightarrow\cV_{r,s}/\cV_{r+2,s}\longrightarrow \cL_{r,s}\longrightarrow 0
\end{equation*}
induces to an $\cM(p)$-module exact sequence
\begin{equation*}
 0\longrightarrow \bigoplus_{k=1}^{r+1} \cM_{2k-r-1,p-s}\longrightarrow\cF_{\cM(p)}(\cV_{r,s}/\cV_{r+2,s})\longrightarrow\bigoplus_{k=1}^{r} \cM_{2k-r,s}\longrightarrow 0
\end{equation*}
by Proposition \ref{prop:Vir_to_sing_ind} and exactness of induction. Since $\lbrace\cM_{2k-r-1,p-s}\rbrace_{k=1}^{r+1}$ and $\lbrace \cM_{2k-r,s}\rbrace_{k=1}^r$ are disjoint sets of irreducible $\cM(p)$-modules which are objects of $\cC_{\cM(p)}^0$, Lemma \ref{lem:untwisted_induction} shows that $\cF_{\cM(p)}(\cV_{r,s}/\cV_{r+2,s})$ is an object of $\cC_{\cM(p)}^0$, and thus so is $\cG_{r',s}/\langle v_{r+2,s}\rangle$. 

Now since both $\cG_{r',s}/\langle v_{r+2,s}\rangle$ and $\cP_{r',s}$ surject onto $\til{\cG}_{r',s}$, since $\cG_{r',s}/\langle v_{r+2,s}\rangle$ is an object of $\cC_{\cM(p)}^0$, and since $\cP_{r',s}$ is projective in $\cC_{\cM(p)}^0$, we get a commuting diagram
\begin{equation*}
 \xymatrixcolsep{3pc}
 \xymatrix{
 & \cP_{r',s} \ar[ld] \ar[d] \\
 \cG_{r',s}/\langle v_{r+2,s}\rangle \ar[r] & \til{\cG}_{r',s}\\
 }
\end{equation*}
of $\cM(p)$-module homomorphisms, with the horizontal and vertical arrows surjective. In addition, the horizontal arrow restricts to an isomorphism of conformal-weight-$h_{r,s}$ spaces, so the generator $v_{r,s}+\langle v_{r+2,s}\rangle$ of $\cG_{r',s}/\langle v_{r+2,s}\rangle$ is in the image of the map $\cP_{r',s}\rightarrow\cG_{r',s}/\langle v_{r+2,s}\rangle$. Thus $\cG_{r',s}/\langle v_{r+2,s}\rangle$ is a quotient of $\cP_{r',s}$ with a one-dimensional top level isomorphic to $T(\cM_{r',s})$. Since we defined $\til{\cG}_{r',s}$ to be a maximal such quotient, the horizontal arrow in the above diagram must be an isomorphism, completing the proof of the lemma.
\end{proof}

Now that we know $K$ is generated by $v_{r+2,s}$ as an $\cM(p)$-module, the next lemma further elucidates the structure of $K$:
\begin{lem}\label{lem:K_top_level}
 The generator $v_{r+2,s}$ of $K$ is contained in the top level $T(K)$ and generates an $A(\cM(p))$-module isomorphic to a quotient of $T(\cM_{r+2,s}\oplus\cM_{-r,s})$.
\end{lem}
\begin{proof}
 To show $v_{r+2,s}\in T(K)$, we need to show $v_n v_{r+2,s}=0$ for any homogeneous $v\in\cM(p)$ and $n\in\ZZ$ such that $\mathrm{wt}\,v-n-1<0$. For $v=\omega$, this is clear because $v_{r+2,s}$ is the image in $\cG_{r',s}$ of a Virasoro singular vector in $\cV_{r,s}$. 
 
 For $v=H$, recall that $H$ generates a $\cV ir$-submodule of $\cM(p)$ isomorphic to $\cL_{3,1}$, so consider the $\cL(p)$-module intertwining operator
 \begin{equation*}
  \cY=Y_{\cG_{r',s}}\vert_{\cL_{3,1}\otimes(\til{\cV}_{r,s}\cap K)}.
 \end{equation*}
By Lemma \ref{lem:tilVrs_cap_K},  $\til{\cV}_{r,s}\cap K$ is a quotient of $\cV_{r+2,s}$, so Corollary \ref{cor:conf_wts}(2) implies that the minimal conformal weight of the $\cV ir$-submodule $\im\cY\subseteq\cG_{r',s}$ is one of
\begin{equation*}
 h_{r+2,s}, h_{r+2,s}+p\pm\sqrt{4ph_{r+2,s}+(p-1)^2}.
\end{equation*}
These conformal weights turn out to be $h_{r,s}$, $h_{r+2,s}$, $h_{r+4,s}$.
Consider the quotient of $\im\cY$ by the $\cV ir$-submodule generated by $(\im\cY)_{[h_{r,s}]}$ (note this submodule might be $0$). Since $\cY$ induces a surjective intertwining operator of type $\binom{\im\cY/\langle(\im\cY)_{[h_{r,s}]}\rangle}{\cL_{3,1}\;\til{\cV}_{r,s}\cap K}$, Corollary \ref{cor:conf_wts}(2) again implies that if $\im\cY/\langle(\im\cY)_{[h_{r,s}]}\rangle\neq 0$, then its minimal conformal weight is one of $h_{r+2,s}$, $h_{r+4,s}$. Since $\mathrm{wt}\,H(n)v_{r+2,s}=h_{r+2,s}-n$, this means that
\begin{equation*}
 H(n)v_{r+2,s}\in\langle(\im\cY)_{[h_{r,s}]}\rangle\subseteq\til{\cV}_{r,s}
\end{equation*}
for all $n>0$. We also have $H(n)v_{r+2,s}\in K$ because $v_{r+2,s}\in K$ and $K$ is an $\cM(p)$-submodule of $\cG_{r',s}$. Then by Lemma \ref{lem:tilVrs_cap_K},
\begin{equation*}
 H(n)v_{r+2,s}\in\til{\cV}_{r,s}\cap K =U(\cV ir_-)\cdot v_{r+2,s},
\end{equation*}
so that $H(n)v_{r+2,s}=0$ for $n>0$.

Now for general homogeneous $v\in\cM(p)$, it will follow that $v_n v_{r+2,s}=0$ for $n\in\ZZ$ such that $\mathrm{wt}\,v-n-1<0$ if we can show that each mode of the vertex operator $Y_{\cG_{r',s}}(v,x)$ can be expressed as an (infinite) linear combination of monomials of the form
\begin{equation*}
 v^{(1)}_{n_1}v^{(2)}_{n_2}\cdots v^{(k)}_{n_k},
\end{equation*}
where $v^{(i)}\in\lbrace\omega, H\rbrace$ for each $i=1,2,\ldots,k$ and all modes that weakly raise conformal weight appear to the left of all modes that strictly lower conformal weight. We can prove this by induction on $\mathrm{wt}\,v$, with the base case $v=\mathbf{1}$ obvious.

Now consider $\mathrm{wt}\,v>0$ and suppose the inductive hypothesis has been proved for all weights less than $\mathrm{wt}\,v$. By \cite[Theorem 4.2(i)]{Ad}, it is enough to consider the two cases $v=L(-m)\til{v}$ for some $m\geq 2$, $\til{v}\in\cM(p)_{(\mathrm{wt}\,v-m)}$ and $v=H_{-n}\til{v}$ for some $n\geq 1$, $\til{v}\in\cM(p)_{(\mathrm{wt}\,v-2p-n+2)}$. For the case $v=H_{-n}\til{v}$, the Jacobi identity iterate \eqref{eqn:intw_op_iterate} and commutator \eqref{eqn:intw_op_comm} formulas yield
\begin{align*}
 Y_{\cG_{r',s}} & (H_{-n}\til{v},x)  =\sum_{i\geq 0} (-1)^i\binom{-n}{i}\left(H_{-n-i}x^iY_{\cG_{r',s}}(\til{v},x)-(-1)^n x^{-n-i}Y_{\cG_{r',s}}(\til{v},x)H_i\right)\nonumber\\
 & =\sum_{i\geq 0} (-1)^i\binom{-n}{i}H_{-n-i}x^i Y_{\cG_{r',s}}(\til{v},x)-\sum_{i\geq 2p-1} (-1)^{n+i}\binom{-n}{i}x^{-n-i}Y_{\cG_{r',s}}(\til{v},x)H_i\nonumber\\
 &\qquad -\sum_{i=0}^{2p-2} (-1)^{n+i}\binom{-n}{i}x^{-n-i} \bigg(H_iY_{\cG_{r',s}}(\til{v},x)-\sum_{j\geq 0} \binom{i}{j}x^{i-j}Y_{\cG_{r',s}}(H_j\til{v},x)\bigg).
\end{align*}
Since the weights of $\til{v}$ and $H_j\til{v}$ for $j\geq 0$ are strictly less than $\mathrm{wt}\,v$, it follows by induction that the modes of $Y_{\cG_{r',s}}(H_{-n}\til{v},x)$ can be expressed in the required manner. The case $v=L(-m)\til{v}$ is similar, so this completes the proof that $v_{r+2,s}\in T(K)$.

Finally, we need to determine the $A(\cM(p))$-submodule of $T(K)$ generated by $v_{r+2,s}$. By \cite[Theorem 6.1]{Ad}, there is an isomorphism
\begin{align*}
 A(\cM(p)) & \rightarrow \CC[x,y]/(P(x,y))\nonumber\\
 [\omega] & \mapsto x+(P(x,y))\nonumber\\
 \quad[H] &\mapsto y+(P(x,y))
\end{align*}
where
\begin{equation*}
P(x,y)=  y^2-C_p(x-h_{1,p})\prod_{s=1}^{p-1}(x-h_{1,s})^2\quad\text{for}\quad C_p=\frac{(4p)^{2p-1}}{((2p-1)!)^2}.
\end{equation*}
Thus
\begin{equation}\label{eqn:H0_square}
 H(0)^2=C_p(L(0)-h_{1,p})\prod_{s=1}^{p-1}(L(0)-h_{1,s})^2
\end{equation}
on $T(K)$. Since $v_{r+2,s}$ is an $L(0)$-eigenvector, this means $A(\cM(p))\cdot v_{r+2,s}$ is spanned by $v_{r+2,s}$ and $H(0)v_{r+2,s}$; moreover, $A(\cM(p))\cdot v_{r+2,s}$ is a quotient of the two-dimensional $A(\cM(p))$-module $T_{r+2,s}$ on which $[\omega]$ and $[H]$ act by the matrices
\begin{equation*}
 [\omega]=\left[\begin{array}{cc}
                 h_{r+2,s} & 0\\
                 0 & h_{r+2,s}\\
                \end{array}
\right],\qquad [H]=\left[\begin{array}{cc}
                    0 & C_p(h_{r+2,s}-h_{1,p})\prod_{s=1}^{p-1}(h_{r+2,s}-h_{1,s})^2\\
                    1 & 0\\
                   \end{array}
\right].
\end{equation*}
The $[H]$-eigenvalue(s) on $T_{r+2,s}$ are the square root(s) of 
\begin{equation*}
 C_p(h_{r+2,s}-h_{1,p})\prod_{s=1}^{p-1}(h_{r+2,s}-h_{1,s})^2.
\end{equation*}
Because $T(\cM_{r+2,s})$ and $T(\cM_{-r,s})$ are two distinct one-dimensional $A(\cM(p))$-modules on which $[H]$ acts by such square root(s) (while $[\omega]$ acts on both by $h_{r+2,s}$), the $[H]$-eigenvalues on $T_{r+2,s}$ must be distinct and we conclude
\begin{equation*}
 T_{r+2,s}\cong T(\cM_{r+2,s})\oplus T(\cM_{-r,s}).
\end{equation*}
Thus $A(\cM(p))\cdot v_{r+2,s}$ is a quotient of  $T(\cM_{r+2,s}\oplus\cM_{-r,s})$.
\end{proof}

Finally, we are ready to prove:
\begin{thm}\label{thm:atypical_GVM}
 For $r\in\ZZ$ and $1\leq s\leq p$, the generalized Verma $\cM(p)$-module $\cG_{r,s}$ is isomorphic to $\til{\cG}_{r,s}$. In particular, $\cG_{r,s}$ is a finite-length $\cM(p)$-module in $\cC_{\cM(p)}^0$.
\end{thm}
\begin{proof}
 To handle the cases $r\geq 1$ and $r\leq 1$ simultaneously, we continue to fix $r\geq 1$ and take $r'$ to be either $r$ or $2-r$. We need to show that the kernel $K$ of the surjection $\pi_{r',s}: \cG_{r',s}\rightarrow\til{\cG}_{r',s}$ is $0$. From Lemmas \ref{lem:K_generator} and \ref{lem:K_top_level}, together with the universal property of generalized Verma $\cM(p)$-modules, we know that $K$ is some quotient of $\cG_{r+2,s}\oplus\cG_{-r,s}$.
 
 Assume towards a contradiction that $K\neq 0$, in which case $K$ has a maximal proper submodule $J$ such that $K/J\cong\cM_{r'',s}$ for either $r''=r+2$ or $r''=-r$.  Then we get an exact sequence
 \begin{equation*}
  0\longrightarrow \cM_{r'',s}\longrightarrow\cG_{r',s}/J\longrightarrow\til{\cG}_{r',s}\longrightarrow 0
 \end{equation*}
of $\cM(p)$-modules. Note that $\cM_{r'',s}$ and $\til{\cG}_{r',s}$ are both objects of $\cC_{\cM(p)}^0$, and that $\cM_{r'',s}$ is not a composition factor of $\til{\cG}_{r',s}$ by \eqref{eqn:Gtil_structure} and \eqref{eqn:Gtil_socle}. Thus $\cG_{r',s}/J$ is an object of $\cC_{\cM(p)}^0$ by Lemma \ref{lem:untwisted_induction}. Then because $\cP_{r',s}$ is projective in $\cC_{\cM(p)}^0$, and because both $\cG_{r',s}/J$ and $\cP_{r',s}$ surject onto $\til{\cG}_{r',s}$, there is a commuting diagram
 \begin{equation*}
 \xymatrixcolsep{3pc}
 \xymatrix{
 & \cP_{r',s} \ar[ld] \ar[d] \\
 \cG_{r',s}/J \ar[r] & \til{\cG}_{r',s}\\
 }
\end{equation*}
of $\cM(p)$-module homomorphisms, with the horizontal and vertical arrows both surjective. As in the proof of Lemma \ref{lem:K_generator}, the generator $v_{r,s}+J$ is in the image of the map $\cP_{r',s}\rightarrow\cG_{r',s}/J$, so $\cG_{r',s}/J$ is a quotient of $\cP_{r',s}$. But this is impossible because $\cM_{r'',s}$ is not a composition factor of $\cP_{r',s}$, so in fact $K$ must be $0$.
\end{proof}

\subsection{Existence of tensor category structure}\label{subsec:tens_cats}

Let $\cO_{\cM(p)}$ denote the category of $C_1$-cofinite grading-restricted generalized $\cM(p)$-modules. In light of Theorems \ref{thm:typical_GVM} and \ref{thm:atypical_GVM}, we can now use \cite[Theorems 3.3.4 and 3.3.5]{CY} to immediately conclude:
\begin{thm}\label{thm:C1_equals_fl}
  The category $\cO_{\cM(p)}$ equals the category of finite-length grading-restricted generalized $\cM(p)$-modules and admits the vertex algebraic braided tensor category structure of \cite{HLZ1}-\cite{HLZ8}.
\end{thm}
\begin{rem}
 Since the irreducible $\cM(p)$-modules $\cM_{r,s}$ are $C_1$-cofinite, the braided tensor categories $\cC_{\cM(p)}$ and $\cC_{\cM(p)}^0$ are subcategories of $\cO_{\cM(p)}$. By \cite[Proposition 3.1.1]{CMY-singlet}, they are also tensor subcategories; in particular, the tensor product formulas of \cite[Theorem 5.2.1]{CMY-singlet} hold in $\cO_{\cM(p)}$.
\end{rem}

Our next goal is to find a tensor subcategory of $\cO_{\cM(p)}$ that contains $\cC_{\cM(p)}^0$ and in which the typical Fock modules $\cF_\lambda$ will be projective. Recall that $\cC_{\cM(p)}^0$ is defined as the subcategory of $\rep^0\cM(p)$ whose objects induce to (untwisted) $\cW(p)$-modules. We also recall that $\cW(p)$ has automorphism group $PSL(2,\CC)$ \cite{ALM} and that $\cM(p)\subseteq\cW(p)$ is the fixed-point subalgebra for the one-dimensional torus $T^{\vee}\subseteq PSL(2,\CC)$. We can identify $T^\vee=\CC/L^\circ$ acting on $\cW(p)\cong\bigoplus_{n\in\ZZ}\cM_{2n+1,1}$ by
\begin{equation*}
 (\beta+L^\circ)\vert_{\cM_{2n+1,1}} = e^{2\pi i\alpha_{2n+1,1}\beta}\Id_{\cM_{2n+1,1}}
\end{equation*}
for $\beta\in\CC$ (recall that $L=\lbrace\alpha_{2n+1,1}\,\vert\,n\in\ZZ\rbrace$). Then we could consider the $T^\vee$-graded subcategory of $\cO_{\cM(p)}$ whose objects $M$ are homogeneous of degree $t^\vee\in T^\vee$ if $\cF_{\cW(p)}(M)$ is a $t^\vee$-twisted $\cW(p)$-module, so that $\cC_{\cM(p)}^0$ would be the subcategory of degree $0$. 

However, we will actually need to grade our subcategory of $\cO_{\cM(p)}$ more finely by the double cover $T=\CC/2L^\circ$ of $T^\vee$, which we can view as the one-dimensional torus of $SL(2,\CC)$. This is the automorphism group of the doublet abelian intertwining algebra $\cA(p)\cong\bigoplus_{r\in\ZZ} \cM_{r,1}$, which is a simple current extension of $\cW(p)$ \cite{AM-doub, ACGY}. Similar to before, $\cM(p)$ is the $T$-fixed-point subalgebra of $\cA(p)$, so for $t\in T$, it would make sense to consider the subcategory of $\cO_{\cM(p)}$ whose objects induce to $t$-twisted $\cA(p)$-modules. However, since the theory of twisted modules for abelian intertwining algebras is not well developed, it is more straightforward to define this subcategory in terms of monodromies with the generator $\cM_{2,1}$ of the group of simple current $\cM(p)$-modules $\lbrace \cM_{r,1}\,\vert\,r\in\ZZ\rbrace$. Note that the monodromy condition of the following definition is a straightforward twisted module generalization of the monodromy condition in \cite[Proposition 2.65]{CKM-exts}:
\begin{defi}\label{defi:Ot}
 For $t=\beta+2L^\circ\in T$, we define $\cO_{\cM(p)}^t$ to be the full subcategory of $\cO_{\cM(p)}$ consisting of $\cM(p)$-modules $M$ such that
 \begin{equation}\label{eqn:twist_mono_cond}
  \cR_{\cM_{2,1},M}^2 =e^{-2\pi i\alpha_{2,1}\beta}\Id_{\cM_{2,1}\tens M}\,\,(=e^{\pi i\alpha_+\beta}\Id_{\cM_{2,1}\tens M} = e^{\pi i\beta\sqrt{2p}}\Id_{\cM_{2,1}\tens M}).
 \end{equation}
\end{defi}
\begin{rem}
 The negative sign in the first exponential of \eqref{eqn:twist_mono_cond} is included to ensure consistency with the usual definition of twisted module for a vertex operator algebra: If $V$ is a vertex operator algebra with automorphism $g$, then one usually defines the vertex operator $Y_W$ of a $g$-twisted $V$-module $W$ to satisfy 
 \begin{equation*}
  Y_W(g\cdot v, e^{2\pi i} x)=Y_W(v,x)
 \end{equation*}
for $v\in V$, or equivalently
\begin{equation*}
 Y_W(v,e^{2\pi i} x) =Y_W(g^{-1}\cdot v,x).
\end{equation*}
That is, the monodromy of $Y_W$ is given by the action of $g^{-1}$ on $V$, not $g$. Thus \eqref{eqn:twist_mono_cond} corresponds to the action of $t=\beta+2L^\circ$ on $\cA(p)$ given by
\begin{equation*}
 (\beta+2L^\circ)\vert_{\cM_{r,1}} =e^{2\pi i\alpha_{r,1}\beta}\Id_{\cM_{r,1}}
\end{equation*}
for $r\in\ZZ$.
\end{rem}

\begin{prop}\label{prop:Ot_abelian}
 For $t=\beta+2L^\circ\in T$, the category $\cO_{\cM(p)}^t$ is closed under submodules and quotients. In particular, $\cO_{\cM(p)}^t$ is an abelian category.
\end{prop}
\begin{proof}
 Suppose $M$ is a module in $\cO_{\cM(p)}^t$ and $N\subseteq M$ is any $\cM(p)$-submodule; $N$ is an object of $\cO_{\cM(p)}$ since $\cO_{\cM(p)}$ is closed under submodules. Then because the monodromy isomorphisms in $\cO_{\cM(p)}$ are natural and because the tensoring functor $\cM_{2,1}\tens\bullet$ is exact, we have a commuting diagram
 \begin{equation*}
  \xymatrixcolsep{3pc}
  \xymatrix{
  0 \ar[r] & \cM_{2,1}\tens N \ar[r] \ar[d]^{\cR_{\cM_{2,1},N}^2} & \cM_{2,1}\tens M \ar[r] \ar[d]^{e^{\pi i\beta\sqrt{2p}}\Id_{\cM_{2,1}\tens M}} & \cM_{2,1}\tens(M/N) \ar[r] \ar[d]^{\cR_{\cM_{2,1},M/N}^2} & 0\\
  0 \ar[r] & \cM_{2,1}\tens N \ar[r] & \cM_{2,1}\tens M \ar[r] & \cM_{2,1}\tens(M/N) \ar[r] & 0
  }
 \end{equation*}
with exact rows. It follows that
\begin{equation*}
 \cR_{\cM_{2,1},N}^2=e^{\pi i\beta\sqrt{2p}}\Id_{\cM_{2,1}\tens N},\qquad\cR_{\cM_{2,1},M/N}^2=e^{\pi i\beta\sqrt{2p}}\Id_{\cM_{2,1}\tens M/N},
\end{equation*}
so both $N$ and $M/N$ are objects of $\cO_{\cM(p)}^t$. 

Closure under submodules and quotients guarantees that every morphism in $\cO_{\cM(p)}^t$ has a kernel and cokernel. The other properties of an abelian category are also easy; $\cO_{\cM(p)}^t$ is closed under finite direct sums, for example, because $\tens$ is an additive functor and monodromy is natural.
\end{proof}

Now we can define a $T$-graded subcategory of $\cO_{\cM(p)}$:
\begin{defi}\label{defi:OT}
 We define $\cO_{\cM(p)}^T$ to be the direct sum subcategory $\bigoplus_{t\in T} \cO_{\cM(p)}^t$ of $\cO_{\cM(p)}$: Objects of $\cO_{\cM(p)}^T$ are $\cM(p)$-modules $\bigoplus_{t\in T} M_t$ such that each $M_t$ is an object of $\cO_{\cM(p)}^t$ and $M_t=0$ for all but finitely many $t\in T$. For objects $\bigoplus_{t\in T} M_t$ and $\bigoplus_{t\in T} N_t$ in $\cO_{\cM(p)}^T$, we define
 \begin{equation*}
  \hom_{\cO_{\cM(p)}^T}\bigg(\bigoplus_{t\in T} M_t,\bigoplus_{t\in T} N_t\bigg)=\bigoplus_{t\in T}\hom_{\cM(p)}(M_t,N_t).
 \end{equation*}
\end{defi}

It is not immediate from the definition that $\cO_{\cM(p)}^T$ is a full subcategory of $\cO_{\cM(p)}$, but we will now prove this along with some other basic properties:
\begin{thm}\label{thm:OT_properties}
 With $\cO_{\cM(p)}^T$ defined as above,
 \begin{enumerate}
  \item Let $M_1$ be an object of $\cO_{\cM(p)}^{t_1}$ and $M_2$ an object $\cO_{\cM(p)}^{t_2}$ for $t_1,t_2\in T$. If $t_1\neq t_2$, then $\hom_{\cM(p)}(M_1,M_2)=0$. In particular, $\cO_{\cM(p)}^T$ is a full subcategory of $\cO_{\cM(p)}$.
  
\item  The category $\cO_{\cM(p)}^T$ is closed under submodules and quotients. In particular, $\cO_{\cM(p)}^T$ is an abelian category.

\item If $M_1$ is an object of $\cO_{\cM(p)}^{t_1}$ and $M_2$ is an object $\cO_{\cM(p)}^{t_2}$ for $t_1,t_2\in T$, then $M_1\boxtimes M_2$ is an object of $\cO_{\cM(p)}^{t_1+t_2}$. In particular, $\cO_{\cM(p)}^T$ is a tensor subcategory of $\cO_{\cM(p)}$.
\end{enumerate}
\end{thm}
\begin{proof}
 To prove (1) and (2), we will need the open Hopf link map $h_M\in\Endo_{\cM(p)}(M)$ associated to $\cM_{2,1}$, for any object $M$ in $\cO_{\cM(p)}$. Since $\cM_{2,1}$ is a rigid simple current $\cM(p)$-module with tensor inverse $\cM_{0,1}$, any evaluation homomorphism
 \begin{equation*}
  e: \cM_{0,1}\tens\cM_{2,1}\longrightarrow\cM_{1,1}
 \end{equation*}
is an isomorphism. Then for an object $M$ in $\cO_{\cM(p)}$, we define $h_M$ to be the composition
\begin{align*}
 M\xrightarrow{l_{M}^{-1}}  \cM_{1,1}\tens M & \xrightarrow{e^{-1}\tens\Id_M} (\cM_{0,1}\tens\cM_{2,1})\tens M\xrightarrow{\cA_{\cM_{0,1},\cM_{2,1}, M}^{-1}} \cM_{0,1}\tens(\cM_{2,1}\tens M)\nonumber\\
 &\xrightarrow{\Id_{\cM_{0,1}}\tens\cR_{\cM_{2,1},M}^2} \cM_{0,1}\tens(\cM_{2,1}\tens M)\xrightarrow{\cA_{\cM_{0,1},\cM_{2,1}, M}} (\cM_{0,1}\tens\cM_{2,1})\tens M\nonumber\\
 &\xrightarrow{e\tens\Id_M}\cM_{1,1}\tens M\xrightarrow{l_M} M.
\end{align*}
Since the unit, associativity, and braiding isomorphisms in $\cO_{\cM(p)}$ are natural, the $\cM(p)$-module isomorphisms $h_M$ define a natural automorphism of the identity functor on $\cO_{\cM(p)}$. Moreover, if $M$ is an object of $\cO_{\cM(p)}^t$ for $t=\beta+2L^\circ\in T$, then $h_M=e^{\pi i\beta\sqrt{2p}}\Id_M$.

Now to prove (1), let $M_1$ be an object of $\cO_{\cM(p)}^{t_1}$ and $M_2$ an object $\cO_{\cM(p)}^{t_2}$ for $t_1=\beta_1+2L^\circ$ and $t_2=\beta_2+2L^\circ$. If $f\in\hom_{\cM(p)}(M_1,M_2)$, then naturality of the open Hopf link automorphisms implies
\begin{equation*}
 f=f\circ h_{M_1}\circ h_{M_1}^{-1} = h_{M_2}\circ f\circ h_{M_1}^{-1}= e^{\pi i(\beta_2-\beta_1)\sqrt{2p}} f.
\end{equation*}
Thus $f=0$ unless perhaps $\beta_1-\beta_2\in\sqrt{2/p}\ZZ=2L^\circ$, that is, $t_1=t_2$.

To prove (2), let $M=\bigoplus_{t\in T} M_t$ be an object of $\cO_{\cM(p)}^T$. The non-zero $M_t$ are the distinct eigenspaces of $h_M$, so $h_M$ is diagonalizable with finitely many eigenvalues.  Now if $N\subseteq M$ is an $\cM(p)$-submodule, then $N$ is an object of $\cO_{\cM(p)}$, so naturality of the open Hopf link automorphisms implies that $h_N=h_M\vert_{N}$. In particular, $h_M$ preserves the submodule $N$, $h_N$ is also diagonalizable with finitely many eigenvalues, and the eigenspaces of $h_N$ are subspaces of the corresponding eigenspaces of $h_M$. Since the $h_M$-eigenspaces are the non-zero $M_t$, this shows that $N=\bigoplus_{t\in T} N_t$ where $N_t\subseteq M_t$ and $N_t\neq 0$ for only finitely many $t\in T$. Each $N_t$ is an $\cM(p)$-submodule of $M_t$ because it is an eigenspace of the $\cM(p)$-module endomorphism $h_N$. Proposition \ref{prop:Ot_abelian} now shows that $N_t$ is an object of $\cO_{\cM(p)}^t$, so $N$ is an object of $\cO_{\cM(p)}^T$. This proves $\cO_{\cM(p)}^T$ is closed under $\cM(p)$-submodules.

To show that $\cO_{\cM(p)}^T$ is also closed under quotients, we have just shown that any quotient $M/N$, with $M$ an object of $\cO_{\cM(p)}^T$ and $N$ an $\cM(p)$-submodule, has the form $\bigoplus_{t\in T} M_t/N_t$ where $M_t/N_t=0$ for all but finitely many $t\in T$ and both $M_t$ and $N_t$ are objects of $\cO_{\cM(p)}^t$ for $t=\beta+2L^\circ\in T$. Then Proposition \ref{prop:Ot_abelian} shows that $M_t/N_t$ is an object of $\cO_{\cM(p)}^t$, so $M/N$ is an object of $\cO_{\cM(p)}^T$.

To prove (3), suppose $M_1$ is an object of $\cO_{\cM(p)}^{t_1}$ and $M_2$ is an object of $\cO_{\cM(p)}^{t_2}$ for $t_1=\beta_1+2L^\circ$ and $t_2=\beta_2+2L^\circ$. Then by the hexagon axiom,
\begin{align*}
 \cR_{\cM_{2,1},M_1\tens M_2}^2 & =\cA_{\cM_{2,1},M_1,M_2}^{-1}\circ(\cR_{M_1,\cM_{2,1}}\tens\Id_{M_2})\circ\cA_{M_1,\cM_{2,1},M_2}\circ(\Id_{M_1}\tens\cR_{\cM_{2,1},M_2}^2)\circ\nonumber\\
 &\qquad\qquad\circ\cA_{M_1,\cM_{2,1},M_2}^{-1}\circ(\cR_{\cM_{2,1},M_1}\tens\Id_{M_2})\circ\cA_{\cM_{2,1},M_1,M_2}\nonumber\\
 & =e^{\pi i\beta_2\sqrt{2p}}\left[\cA_{\cM_{2,1},M_1,M_2}^{-1}\circ(\cR_{\cM_{2,1},M_1}^2\tens\Id_{M_2})\circ\cA_{\cM_{2,1},M_1,M_2}\right]\nonumber\\
 & =e^{\pi i(\beta_1+\beta_2)\sqrt{2p}}\Id_{\cM_{2,1}\tens(M_1\tens M_2)},
\end{align*}
which means that $M_1\tens M_2$ is an object of $\cO_{\cM(p)}^{t_1+t_2}$.
\end{proof}

\begin{rem}\label{rem:open_Hopf_link}
The open Hopf link map $h_M$ associated to $\cM_{2,1}$ used in the preceding proof is somewhat different from the usual open Hopf link map defined in references such as \cite[Section 3.1.3]{CG}, since we have defined $h_M$ using the map $e^{-1}:\cM_{1,1}\rightarrow\cM_{0,1}\tens\cM_{2,1}$, rather than using the composition
\begin{equation*}
\til{i}: \cM_{1,1} \xrightarrow{i_{\cM_{2,1}}} \cM_{2,1}\tens\cM_{0,1}\xrightarrow{\theta_{\cM_{2,1}}\tens\Id_{\cM_{0,1}}} \cM_{2,1}\tens\cM_{0,1}\xrightarrow{\cR_{\cM_{2,1},\cM_{0,1}}} \cM_{0,1}\tens\cM_{2,1}
\end{equation*}
(where $i_{\cM_{2,1}}$ is the coevaluation and $\theta=e^{2\pi i L(0)}$ is the ribbon twist). Since $e\circ\til{i}$ is by definition the categorical dimension $\dim_{\cM(p)} \cM_{2,1}$, $h_M$ is related to the usual open Hopf link map (denoted $\Phi_{\cM_{2,1},M}$ in \cite{CG}) by $\Phi_{\cM_{2,1},M} =(\dim_{\cM(p)}\cM_{2,1})h_M$.
\end{rem}

Our next goal is to show that $\cO_{\cM(p)}^T$ contains all simple $\cM(p)$-modules in $\cO_{\cM(p)}$. First we need a lemma:
\begin{lem}\label{lem:simple_curr_tens}
 For $r\in\ZZ$ and $\cF_\lambda$ a typical irreducible $\cM(p)$-module,
 \begin{equation}\label{eqn:simple_curr_tens}
  \cM_{r,1}\tens\cF_\lambda\cong\cF_{\lambda+\alpha_{r,1}}\,\,\left(=\cF_{\lambda-(r-1)\frac{\alpha_+}{2}}=\cF_{\lambda-(r-1)\sqrt{\frac{p}{2}}}\right)
 \end{equation}
in $\cO_{\cM(p)}$.
\end{lem}
\begin{proof}
 There is a non-zero $\mathcal{H}$-module intertwining operator $\cY$ of type $\binom{\cF_{\lambda+\alpha_{r,1}}}{\cF_{\alpha_{r,1}}\,\cF_\lambda}$. Thus if $f$ denotes the inclusion $\cM_{r,1}\hookrightarrow\cF_{\alpha_{r,1}}$, then $\cY\circ(f\otimes\Id_{\cF_\lambda})$ is an $\cM(p)$-module intertwining operator of type $\binom{\cF_{\lambda+\alpha_{r,1}}}{\cM_{r,1}\,\cF_\lambda}$. As $\cF_{\alpha_{r,1}}$ and $\cF_\lambda$ are irreducible $\mathcal{H}$-modules, \cite[Proposition 11.9]{DL} says $\cY(w_1,x)w_2\neq 0$ for any non-zero $w_1\in\cF_{\alpha_{r,1}}$, $w_2\in\cF_\lambda$, and this means $\cY\circ(f\otimes\Id_{\cF_\lambda})\neq 0$. Then the universal property of tensor products in $\cO_{\cM(p)}$ induces a non-zero $\cM(p)$-module homomorphism $F:\cM_{r,1}\tens\cF_\lambda\rightarrow\cF_{\lambda+\alpha_{r,1}}$. 
 
 Since $\cM_{r,1}$ is a simple current $\cM(p)$-module, the domain of $F$ is irreducible. The codomain is also irreducible because $\lambda+\alpha_{r,1}\neq\alpha_{r',s'}$ for $r'\in\ZZ$, $1\leq s'\leq p-1$ when $\lambda\in(\CC\setminus L^\circ)\cup\lbrace\alpha_{r,p}\,\vert\,r\in\ZZ\rbrace$. Thus $F$ is an isomorphism by Schur's Lemma.
\end{proof}

Now we can prove:
\begin{prop}\label{prop:irred_mod_grading}
 For $\lambda\in(\CC\setminus L^\circ)\cup\lbrace\alpha_{r,p}\,\vert\,r\in\ZZ\rbrace$, the typical irreducible $\cM(p)$-module $\cF_\lambda$ is an object of $\cO_{\cM(p)}^{\lambda+2L^\circ}$. Moreover, for $r\in\ZZ$ and $1\leq s\leq p$, the irreducible $\cM(p)$-module $\cM_{r,s}$ is an object of $\cO_{\cM(p)}^{\alpha_{r,s}+2L^\circ}$ 
\end{prop}
\begin{proof}
For $\lambda\in(\CC\setminus L^\circ)\cup\lbrace\alpha_{r,p}\,\vert\,r\in\ZZ\rbrace$, we  use the balancing equation for monodromy, the tensor product formula \eqref{eqn:simple_curr_tens}, and the conformal weight \eqref{eqn:h_lambda_def} to calculate
 \begin{align*}
  \cR_{\cM_{2,1},\cF_\lambda}^2 & = \theta_{\cM_{2,1}\tens\cF_\lambda}\circ(\theta_{\cM_{2,1}}^{-1}\tens\theta_{\cF_\lambda}^{-1})\nonumber\\
  &= e^{2\pi i \left(h_{\lambda+\alpha_{2,1}}-h_{\alpha_{2,1}}-h_\lambda\right)}\Id_{\cM_{2,1}\tens\cF_\lambda}\nonumber\\
  &= e^{\pi i\left[(\lambda+\alpha_{2,1})(\alpha_0-\lambda-\alpha_{2,1})-\alpha_{2,1}(\alpha_0-\alpha_{2,1})-\lambda(\alpha_0-\lambda)\right]}\Id_{\cM_{2,1}\tens\cF_\lambda}\nonumber\\
  & =e^{-2\pi i\alpha_{2,1}\lambda}\Id_{\cM_{2,1}\tens\cF_\lambda}.
 \end{align*}
Thus $\cF_\lambda$ is an object of $\cO_{\cM(p)}^{\lambda+2L^\circ}$ by \eqref{eqn:twist_mono_cond}. For $r\in\ZZ$ and $1\leq s\leq p$, the same calculation using the formula $\cM_{2,1}\tens\cM_{r,s}\cong \cM_{r+1,s}$ from \cite[Theorem 3.2.8(1)]{CMY-singlet} shows
\begin{equation*}
 \cR_{\cM_{2,1},\cM_{r,s}}^2 =e^{2\pi i(h_{r+1,s}-h_{2,1}-h_{r,s})}\Id_{\cM_{2,1}\tens\cM_{r,s}} = e^{\pi i[(r-1)p-(s-1)]}\Id_{\cM_{2,1}\tens\cM_{r,s}}.
\end{equation*}
(Note that although the lowest conformal weight of $\cM_{r,s}$ equals $h_{r,s}$ only for $r\geq 1$, it is always congruent to $h_{r,s}$ modulo $\ZZ$.) Since
\begin{equation*}
 e^{\pi i[(r-1)p-(s-1)]}=e^{\pi i[(1-r)p-(1-s)]}=e^{2\pi i\left(\frac{1-r}{2}p-\frac{1-s}{2}\right)\sqrt{\frac{2}{p}}\sqrt{\frac{p}{2}}} =e^{\pi i\alpha_{r,s}\sqrt{2p}},
\end{equation*}
it follows from \eqref{eqn:twist_mono_cond} that $\cM_{r,s}$ is an object of $\cO_{\cM(p)}^{\alpha_{r,s}+2L^\circ}$.
\end{proof}

Next we determine the relation between $\cC_{\cM(p)}^0$ and $\cO_{\cM(p)}^T$:
\begin{prop}\label{prop:C0_and_OT}
 The category $\cC_{\cM(p)}^0$ of \cite[Section 3.1]{CMY-singlet} equals $\cO_{\cM(p)}^{0+2L^\circ}\oplus\cO_{\cM(p)}^{\alpha_-/2+2L^\circ}$.
\end{prop}
\begin{proof}
 Recalling that $\frac{\alpha_-}{2}$ spans $L^\circ$, we denote $\cO_{\cM(p)}^{0+2L^\circ}\oplus\cO_{\cM(p)}^{\alpha_-/2+2L^\circ}$ by $\cO_{\cM(p)}^{L^\circ}$. Since $\frac{\alpha_-}{2}+\frac{\alpha_-}{2}\in 2L^\circ$, Theorem \ref{thm:OT_properties} and its proof show that $\cO_{\cM(p)}^{L^\circ}$ is a full tensor subcategory of $\cO_{\cM(p)}$; in particular, $\cO_{\cM(p)}^{L^\circ}$ is closed under tensor products, finite direct sums, and subquotients. Any object in $\cO_{\cM(p)}^{L^\circ}$ has finite length, and its composition factors are objects of $\cO_{\cM(p)}^{L^\circ}$ since $\cO_{\cM(p)}^{L^\circ}$ is closed under subquotients. Proposition \ref{prop:irred_mod_grading} shows that an irreducible $\cM(p)$-module is an object of $\cO_{\cM(p)}^{L^\circ}$ if and only if it is one of the $\cM_{r,s}$ for $r\in\ZZ$, $1\leq s\leq p$, so $\cO_{\cM(p)}^{L^\circ}$ is a full tensor subcategory of $\cC_{\cM(p)}$ defined in \cite[Section 3.1]{CMY-singlet}.
 
 On the other hand, $\cC_{\cM(p)}^0$ is also a full tensor subcategory of $\cC_{\cM(p)}$ that contains every irreducible $\cM_{r,s}$ for $r\in\ZZ$, $1\leq s\leq p$ (see Proposition 3.2.5 and Theorem 3.3.1 of \cite{CMY-singlet}). Moreover, \cite[Section 5]{CMY-singlet} shows that $\cC_{\cM(p)}^0$ has enough projectives and that every indecomposable projective object occurs as a direct summand of a tensor product of two simple objects. This means that $\cC_{\cM(p)}^0$ is the smallest full tensor subcategory of $\cC_{\cM(p)}$ that contains all the $\cM_{r,s}$, so $\cC_{\cM(p)}^0$  is a subcategory of  $\cO_{\cM(p)}^{L^\circ}$.
 
 Conversely, to show that $\cO_{\cM(p)}^{L^\circ}$ is a subcategory of $\cC_{\cM(p)}^0$, note that $\cO_{\cM(p)}^{L^\circ}$ is the subcategory of modules in $\cC_{\cM(p)}$ whose indecomposable summands $M$ satisfy $\cR_{\cM_{2,1},M}^2 =\pm\Id_{\cM_{2,1},M}$, while from the discussion preceding Lemma \ref{lem:untwisted_induction}, $\cC_{\cM(p)}^0$ is the subcategory of modules $M$ such that $\cR_{\cM_{3,1},M}^2=\Id_{\cM_{3,1}\tens M}$. Now let $M$ be an indecomposable object of $\cO_{\cM(p)}^{L^\circ}$. To show that $M$ is also an object of $\cC_{\cM(p)}^0$, we fix an isomorphism $f:\cM_{2,1}\tens\cM_{2,1}\rightarrow\cM_{3,1}$ (guaranteed by \cite[Theorem 3.2.8(1)]{CMY-singlet}) and calculate
 \begin{align*}
  \cR_{\cM_{3,1},M}^2 & =(f\tens\Id_M)\circ\cR_{\cM_{2,1}\tens\cM_{2,1},M}^2\circ (f^{-1}\tens\Id_M)\nonumber\\
  & =(f\tens\Id_M)\circ\cA_{\cM_{2,1},\cM_{2,1},M}\circ(\Id_{\cM_{2,1}}\tens\cR_{M,\cM_{2,1}})\circ\cA_{\cM_{2,1},M,\cM_{2,1}}^{-1}\circ\nonumber\\
  &\qquad\qquad\circ(\cR_{\cM_{2,1},M}^2\tens\Id_{\cM_{2,1}})\circ\cA_{\cM_{2,1},M,\cM_{2,1}}\circ(\Id_{\cM_{2,1}}\tens\cR_{\cM_{2,1},M})\circ\nonumber\\
  &\qquad\qquad\circ\cA_{\cM_{2,1},\cM_{2,1},M}^{-1}\circ(f^{-1}\tens\Id_M)\nonumber\\
  & =\pm(f\tens\Id_M)\circ\cA_{\cM_{2,1},\cM_{2,1},M}\circ(\Id_{\cM_{2,1}}\tens\cR_{\cM_{2,1},M}^2)\circ\cA_{\cM_{2,1},\cM_{2,1},M}^{-1}\circ(f^{-1}\tens\Id_M)\nonumber\\
  & =\Id_{\cM_{3,1}\tens M}
 \end{align*}
as required.
\end{proof}

\subsection{Projective \texorpdfstring{$\cM(p)$}{M(p)}-modules}\label{sec:projective}

In this subsection, we classify projective objects in the tensor category $\cO_{\cM(p)}^T$ defined in the previous subsection. In particular, we will show that $\cO_{\cM(p)}^T$ has enough projectives, that is, every irreducible $\cM(p)$-module has a projective cover in $\cO_{\cM(p)}^T$. For the atypical modules, this is an easy consequence of Proposition \ref{prop:C0_and_OT} combined with the results of \cite{CMY-singlet}:
\begin{prop}\label{prop:Prs_proj}
For $r\in\ZZ$ and $1\leq s\leq p$, the indecomposable $\cM(p)$-module $\cP_{r,s}$ is a projective cover of $\cM_{r,s}$ in $\cO_{\cM(p)}^T$.
\end{prop}

\begin{proof}
 It is shown in \cite{CMY-singlet} that $\cP_{r,s}$ is a projective object of $\cC_{\cM(p)}^0$, so by Proposition \ref{prop:C0_and_OT}, $\cP_{r,s}$ is an object of $\cO_{\cM(p)}^T$ which is projective in the subcategory $\cO_{\cM(p)}^{L^\circ}=\cO_{\cM(p)}^{0+2L^\circ}\oplus\cO_{\cM(p)}^{\alpha_-/2+2L^\circ}$. To show that $\cP_{r,s}$ is still projective in $\cO_{\cM(p)}^T$, consider a surjection $p: M\rightarrow N$ and a morphism $q: \cP_{r,s}\rightarrow N$, where $M=\bigoplus_{t\in T} M_t$ and $N=\bigoplus_{t\in T} N_t$ are objects of $\cO_{\cM(p)}^T$. We need to show that there exists $f:\cP_{r,s}\rightarrow M$ such that $p\circ f=q$.
 
 Assuming as we may that $q\neq 0$, Theorem \ref{thm:OT_properties}(1) implies $N_{L^\circ}=N_{0+2L^\circ}\oplus N_{\alpha_-/2+2L^\circ}$ is non-zero. Then because $p$ is surjective, $M_{L^\circ}=M_{0+2L^\circ}\oplus M_{\alpha_-/2+2L^\circ}$ is also non-zero and
 \begin{equation*}
  \im p\vert_{M_{L^\circ}} = N_{L^\circ}.
 \end{equation*}
Thus because $\cP_{r,s}$ is projective in $\cO_{\cM(p)}^{L^\circ}$, there is a morphism $f: \cP_{r,s}\rightarrow M_{L^\circ}\hookrightarrow M$ such that $p\circ f=q$, showing $\cP_{r,s}$ is also projective in $\cO_{\cM(p)}^T$. Moreover, the same argument as concludes the proof of \cite[Proposition 3.3.5]{CMY-singlet} now shows that $\cP_{r,s}$ is a projective cover of $\cM_{r,s}$ in $\cO_{\cM(p)}^T$.
\end{proof}

The typical irreducible $\cM(p)$-modules are their own projective covers in $\cO_{\cM(p)}^T$:
\begin{thm}\label{thm:Flambda_proj}
 For $\lambda\in(\CC\setminus L^\circ)\cup\lbrace\alpha_{r,p}\,\vert\,r\in\ZZ\rbrace$, the irreducible $\cM(p)$-module $\cF_\lambda$ is projective in $\cO_{\cM(p)}^T$. In particular, $\cF_\lambda$ is its own projective cover in $\cO_{\cM(p)}^T$.
\end{thm}
\begin{proof}
 The case $\lambda=\alpha_{r,p}$ is covered by Proposition \ref{prop:Prs_proj}, so we assume $\lambda\in\CC\setminus L^\circ$. Since $\cF_\lambda$ is simple and every module in $\cO_{\cM(p)}^T$ has finite length, a straightforward induction on length implies that it is enough to show that any short exact sequence
 \begin{equation}\label{eqn:Flambda_extension}
  0\longrightarrow M \longrightarrow N \longrightarrow\cF_\lambda\longrightarrow 0
 \end{equation}
splits when $M$ is simple and $N$ is an object of $\cO_{\cM(p)}^T$.

Let $h$ be the minimal conformal weight of $M$, so that \eqref{eqn:Flambda_extension} can fail to split only when $h-h_\lambda\in\ZZ$. If $h-h_\lambda\in\ZZ_+$, then the top level $T(N)$ contains $T(\cF_\lambda)$, so the universal property of generalized Verma $\cM(p)$-modules implies that $N$ contains a non-zero homomorphic image of $\cG_\lambda\cong\cF_\lambda$ (using Theorem \ref{thm:typical_GVM}). That is, $N$ contains $\cF_\lambda$ as a submodule and \eqref{eqn:Flambda_extension} splits.

If $h_\lambda-h\in\ZZ_+$, then we can dualize \eqref{eqn:Flambda_extension} to get an exact sequence
\begin{equation*}
 0\longrightarrow\cF_{\alpha_0-\lambda}\longrightarrow N'\longrightarrow M'\longrightarrow 0.
\end{equation*}
In this case $N'$ contains a non-zero quotient of $G(T(M'))$; this quotient is isomorphic to either $M'$ or $N'$ since $N'$ has length $2$. But since $M'\ncong\cF_{\alpha_0-\lambda}$, Theorems \ref{thm:typical_GVM} and \ref{thm:atypical_GVM} show that $\cF_{\alpha_0-\lambda}$ is not a composition factor of $G(T(M'))$. Thus $N'$ contains a submodule isomorphic to $M'$, and then $N'\cong\cF_{\alpha_0-\lambda}\oplus M'$. Dualizing again shows that \eqref{eqn:Flambda_extension} splits.

It remains to consider $h=h_\lambda$; in this case \eqref{eqn:h_lambda_def} implies that $M$ is isomorphic to either $\cF_\lambda$ or its contragredient $\cF_{\alpha_0-\lambda}$. If $M\cong\cF_{\alpha_0-\lambda}$, note that $\cF_{\alpha_0-\lambda}\ncong\cF_\lambda$ as $\cM(p)$-modules (since $\lambda\neq\alpha_{1,p}$) while $L(0)$ acts identically on $T(\cF_{\alpha_0-\lambda})$ and $T(\cF_\lambda)$. Thus $H(0)$ must act by different eigenvalues on $T(\cF_{\alpha_0-\lambda})$ and $T(\cF_\lambda)$. If we denote these $H(0)$-eigenvalues by $H_{\alpha_0-\lambda}$ and $H_\lambda$, respectively, then the top level $T(M)$ has a basis with respect to which $H(0)$ and $L(0)$ act by the matrices
\begin{equation*}
 H(0)=\left[\begin{array}{cc}
              H_{\alpha_0-\lambda} & 0\\
              0 & H_\lambda\\
             \end{array}
\right],\qquad L(0)=\left[\begin{array}{cc}
                     h_\lambda & a\\
                     0 & h_\lambda\\
                    \end{array}
\right]
\end{equation*}
for some $a\in\CC$. Since $H(0)$ commutes with $L(0)$ and $H_{\alpha_0-\lambda}\neq H_\lambda$, we get $a=0$, so $T(M)\cong T(\cF_{\alpha_0-\lambda})\oplus T(\cF_\lambda)$. Thus $M$ is a length-$2$ homomorphic image of the generalized Verma $\cM(p)$-module $\cG_{\alpha_0-\lambda}\oplus\cG_\lambda\cong\cF_{\alpha_0-\lambda}\oplus\cF_\lambda$, and again we see that \eqref{eqn:Flambda_extension} splits.

Finally, we need to consider the possibility $M\cong\cF_\lambda$. In this case, $\cF_\lambda$ does have a non-split self-extension in $\cO_{\cM(p)}$, namely, the $\mathcal{H}$-module $\cF_\lambda^{(2)}$ with two-dimensional top level on which $h(0)$ acts by the matrix
\begin{equation*}
 \left[\begin{array}{cc}
        \lambda & 1\\
        0 & \lambda\\
       \end{array}
\right].
\end{equation*}
Thus it is enough to show that any non-split self-extension of $\cF_\lambda$ in $\cO_{\cM(p)}$ is isomorphic to $\cF_{\lambda}^{(2)}$, and that $\cF_{\lambda}^{(2)}$ is not an object of $\cO_{\cM(p)}^T$. 

First consider any non-split exact sequence
\begin{equation*}
 0\longrightarrow\cF_\lambda\longrightarrow M\longrightarrow \cF_\lambda\longrightarrow 0
\end{equation*}
in $\cO_{\cM(p)}$. Then the top level $T(M)$ has a basis with respect to which $L(0)$ and $H(0)$ act by the matrices
\begin{equation*}
 L(0)=\left[\begin{array}{cc}
        h_\lambda & a\\
        0 & h_\lambda\\
       \end{array}
\right],\qquad H(0)=\left[\begin{array}{cc}
                           H_\lambda & b\\
                           0 & H_\lambda\\
                          \end{array}
\right]
\end{equation*}
for certain $a,b\in\CC$. Since $H(0)^2$ is a polynomial in $L(0)$ by \eqref{eqn:H0_square}, and since $H_\lambda\neq 0$ (because \eqref{eqn:H0_square} shows that $H_\lambda=0$ only for $\lambda=\alpha_{1,s}$, $1\leq s\leq p$), it follows that $b$ is completely determined by $a$. In particular, $a=0$ would imply both $L(0)$ and $H(0)$ are diagonalizable on $T(M)$, which would then imply $T(M)=T(\cF_\lambda)\oplus T(\cF_\lambda)$. By the universal property of generalized Verma $\cM(p)$-modules, this would imply $M=\cF_\lambda\oplus\cF_\lambda$. Consequently, $a\neq 0$ since $M$ is indecomposable, and by adjusting the basis of $T(M)$ if necessary, we may assume $a=1$. Thus up to isomorphism, there is only one possible $A(\cM(p))$-module structure on $T(M)$. In fact, $T(M)\cong T(\cF_{\lambda}^{(2)})$ since $\cF_\lambda^{(2)}$ is one possible non-split self-extension of $\cF_\lambda$. 

Since $M$ has length $2$ and both its composition factors intersect $T(M)$, we see that $M$ is generated by $T(M)$. Then the universal property of generalized Verma $\cM(p)$-modules applied to the isomorphism $T(M)\cong T(\cF_{\lambda}^{(2)})$ shows that $M$ is a length-$2$ quotient of $G(T(\cF_\lambda^{(2)}))$. Thus if we can show that $G(T(\cF_\lambda^{(2)}))\cong\cF_\lambda^{(2)}$, then it will follow that $M\cong\cF_\lambda^{(2)}$ as required. In fact, since $T(\cF_\lambda^{(2)})$ contains $T(\cF_\lambda)$ as an $A(\cM(p))$-submodule, $G(T(\cF_\lambda^{(2)}))$ has a submodule isomorphic to $\cG_\lambda\cong\cF_\lambda$. Then since $G(T(\cF_\lambda^{(2)}))$ is generated by its top level, so is $G(T(\cF_\lambda^{(2)}))/\cF_\lambda$. But since
\begin{equation*}
T\left( G(T(\cF_\lambda^{(2)}))/\cF_\lambda\right)\cong T(\cF_\lambda),
\end{equation*}
this means that $G(T(\cF_\lambda^{(2)}))/\cF_\lambda$ is also a homomorphic image of $\cG_\lambda\cong\cF_\lambda$. This shows that $G(T(\cF_\lambda^{(2)}))$ is a length-$2$ self-extension of $\cF_\lambda$. Since by definition $G(T(\cF_\lambda^{(2)}))$ surjects onto the length-$2$ self-extension $\cF_\lambda^{(2)}$, we must have $G(T(\cF_\lambda^{(2)}))\cong\cF_\lambda^{(2)}$, as required.

We have now shown that up to isomorphism, $\cF_\lambda^{(2)}$ is the only non-split length-$2$ extension of $\cF_\lambda$ in $\cO_{\cM(p)}$. So to complete the proof of the theorem, we just need to show that $\cF_\lambda^{(2)}$ is not an object of $\cO_{\cM(p)}^T$. If $\cF_\lambda^{(2)}$ were an object of $\cO_{\cM(p)}^T$, then it would be an object of $\cO_{\cM(p)}^{\lambda+2L^\circ}$ because it is indecomposable and its composition factors are objects of $\cO_{\cM(p)}^{\lambda+2L^\circ}$ by Proposition \ref{prop:irred_mod_grading}. Then Theorem \ref{thm:OT_properties}(3) would imply that $\cF_{\alpha_{1,p}-\lambda}\tens\cF_\lambda^{(2)}$ is an object of $\cO_{\cM(p)}^{\alpha_{1,p}+2L^\circ}$. Now, there is a surjective $\mathcal{H}$-module intertwining operator
\begin{equation*}
 \cY: \cF_{\alpha_{1,p}-\lambda}\otimes\cF_\lambda^{(2)}\longrightarrow\cF_{\alpha_{1,p}}^{(2)}[\log x]\lbrace x\rbrace.
\end{equation*}
As $\cY$ is also an $\cM(p)$-module intertwining operator, it induces an $\cM(p)$-module surjection
\begin{equation*}
 \cF_{\alpha_{1,p}-\lambda}\tens\cF_\lambda^{(2)}\longrightarrow\cF_{\alpha_{1,p}}^{(2)}.
\end{equation*}
Thus because $\cO_{\cM(p)}^T$ is closed under quotients by Theorem \ref{thm:OT_properties}(2), $\cF_\lambda^{(2)}$ an object of $\cO_{\cM(p)}^T$ would imply the same for $\cF_{\alpha_{1,p}}^{(2)}$. But $\cF_{\alpha_{1,p}}^{(2)}$ is not an object of $\cO_{\cM(p)}^T$ because $\cF_{\alpha_{1,p}}$ is projective in $\cO_{\cM(p)}^T$ by Proposition \ref{prop:Prs_proj}. So $\cF_\lambda^{(2)}$ is not an object of $\cO_{\cM(p)}^T$ either.
\end{proof}

We conclude this subsection with a useful observation about the subcategories $\cO_{\cM(p)}^{\lambda+2L^\circ}$ for $\lambda\in\CC\setminus L^\circ$. By Propositions \ref{prop:Ot_abelian} and \ref{prop:irred_mod_grading}, any composition factor of an object in $\cO_{\cM(p)}^{\lambda+2L^\circ}$ is isomorphic to some $\cF_\mu$ such that $\lambda-\mu\in 2L^\circ$. Thus an easy induction on length using Theorem \ref{thm:Flambda_proj} yields:
\begin{cor}\label{cor:Olambda_structure}
 For $\lambda\in\CC\setminus L^\circ$, any object of $\cO_{\cM(p)}^{\lambda+2L^\circ}$ is isomorphic to a finite direct sum of Fock modules $\cF_\mu$ such that $\lambda-\mu\in 2L^\circ$.
\end{cor}

\section{Fusion rules}\label{sec:fusion}

In this section, we compute tensor products in $\cO_{\cM(p)}$ involving typical irreducible modules; see \cite[Theorem 5.2.1]{CMY-singlet} for all tensor products involving only atypical irreducible modules and their projective covers.

\subsection{Atypical-typical fusion}

We first determine how $\cM_{1,2}$ tensors with typical modules:
\begin{lem}\label{lem:M12_Flambda}
 For any $\lambda\in\CC\setminus L^\circ$,
 \begin{equation*}
  \cM_{1,2}\tens\cF_\lambda\cong\cF_{\lambda+\alpha_{1,2}}\oplus\cF_{\lambda-\alpha_{1,2}}.
 \end{equation*}
\end{lem}
\begin{proof}
By Theorem \ref{thm:OT_properties}(3), Proposition \ref{prop:irred_mod_grading}, and Corollary \ref{cor:Olambda_structure}, $\cM_{1,2}\tens\cF_\lambda$ is a finite direct sum of Fock modules $\cF_{\mu}$ such that $\lambda+\alpha_{1,2}-\mu\in 2L^\circ$. Since $\cM_{1,2}$ contains $\cL_{1,2}$ as a $\cV ir$-submodule by \eqref{eqn:Mrs_decomp}, we have for any such $\mu$ a linear map
\begin{align*}
 I_{\cM(p)}\binom{\cF_\mu}{\cM_{1,2}\,\cF_\lambda} & \rightarrow I_{\cL(p)}\binom{\cF_\mu}{\cL_{1,2}\,\cF_\lambda}\nonumber\\
 \cY &\mapsto \cY\vert_{\cL_{1,2}\otimes\cF_\lambda}
\end{align*}
Since $\cM_{1,2}$ is an irreducible $\cM(p)$-module, \cite[Proposition 11.9]{DL} says that this map is injective. Thus because $\cF_\lambda$ and $\cF_\mu$ are irreducible $\cV ir$-modules (since $\lambda,\mu\notin L^\circ)$, Proposition \ref{prop:pi_injective} and Remark \ref{rem:pi_injective} imply that the multiplicity of $\cF_\mu$ in $\cM_{1,2}\tens\cF_\lambda$ cannot be greater than $1$. Moreover, by Corollary \ref{cor:conf_wts}, the multiplicity of $\cF_\mu$ can be non-zero only if 
\begin{align*}
 h_\mu &= h_\lambda+\frac{p^{-1}}{4}\pm\frac{p^{-1}}{2}\sqrt{4ph_\lambda+(p-1)^2}\nonumber\\
 & = \frac{1}{2}\lambda(\lambda-\alpha_0)+\frac{1}{2}\alpha_{1,2}^2\pm\frac{\alpha_{1,2}}{\sqrt{2p}}\left(\sqrt{2p}\lambda-p+1\right) = h_{\lambda\pm\alpha_{1,2}}
\end{align*}
(for this calculation, note that $\alpha_{1,2}=-\frac{1}{2}\alpha_-=\frac{1}{\sqrt{2p}}$). Thus if $\cF_\mu$ has non-zero multiplicity (equal to $1$) in $\cM_{1,2}\tens\cF_\lambda$, then $\mu$ is one of $\lambda\pm\alpha_{1,2}$ or $\alpha_0-(\lambda\pm\alpha_{1,2})$. We can rule out the latter two possibilities because
\begin{equation*}
 \lambda+\alpha_{1,2}-(\alpha_0-\lambda\mp\alpha_{1,2})\in 2\lambda+2L^\circ,
\end{equation*}
and $2\lambda+2L^\circ$ is disjoint from $2L^\circ$ since by assumption $\lambda\notin L^\circ$.

We have now shown that $\cM_{1,2}\tens\cF_{\lambda}$ is a submodule of $\cF_{\lambda+\alpha_{1,2}}\oplus\cF_{\lambda-\alpha_{1,2}}$. To show that this direct sum is indeed the tensor product module, we need to demonstrate non-zero $\cM(p)$-module homomorphisms $\cM_{1,2}\tens\cF_\lambda\rightarrow\cF_{\lambda\pm\alpha_{1,2}}$ for both sign choices. First, as in the proof of Lemma \ref{lem:simple_curr_tens}, there is a non-zero map $\cM_{1,2}\tens\cF_\lambda\rightarrow\cF_{\lambda+\alpha_{1,2}}$ induced by a non-zero $\mathcal{H}$-module intertwining operator of type $\binom{\cF_{\lambda+\alpha_{1,2}}}{\cF_{\alpha_{1,2}}\,\,\cF_\lambda}$, for any $\lambda\in\CC\setminus L^\circ$. To get the second non-zero map, the first case implies that $\cF_\lambda$ is a direct summand of $\cM_{1,2}\tens\cF_{\lambda-\alpha_{1,2}}$. So because $\cM_{1,2}$ is a rigid $\cM(p)$-module (see \cite[Section 4.2]{CMY-singlet}),
\begin{equation*}
 \dim\hom_{\cM(p)}(\cM_{1,2}\tens\cF_\lambda,\cF_{\lambda-\alpha_{1,2}}) =\dim\hom_{\cM(p)}(\cF_\lambda,\cM_{1,2}\tens\cF_{\lambda-\alpha_{1,2}})=1.
\end{equation*}
More concretely, using $q$ to denote an injection $\cF_\lambda\hookrightarrow\cM_{1,2}\tens\cF_{\lambda-\alpha_{1,2}}$, the composition 
\begin{align*}
 \cM_{1,2}\tens\cF_\lambda\xrightarrow{\Id_{\cM_{1,2}}\tens q} &\cM_{1,2}\tens(\cM_{1,2}\tens\cF_{\lambda-\alpha_{1,2}})\xrightarrow{\cA_{\cM_{1,2},\cM_{1,2},\cF_{\lambda-\alpha_{1,2}}}} (\cM_{1,2}\tens\cM_{1,2})\tens\cF_{\lambda-\alpha_{1,2}}\nonumber\\
 &\xrightarrow{e_{\cM_{1,2}}\tens\Id_{\cF_{\lambda-\alpha_{1,2}}}} \cM_{1,1}\tens\cF_{\lambda-\alpha_{1,2}}\xrightarrow{l_{\cF_{\lambda-\alpha_{1,2}}}} \cF_{\lambda-\alpha_{1,2}}
\end{align*}
is non-zero; here $e_{\cM_{1,2}}$ is the evaluation homomorphism.
\end{proof}

We now use Lemmas \ref{lem:simple_curr_tens} and \ref{lem:M12_Flambda} to compute how each $\cM_{r,s}$ tensors with the typical Fock modules:
\begin{thm}\label{thm:Mrs_Flambda}
 For $r\in\ZZ$, $1\leq s\leq p$, and $\lambda\in\CC\setminus L^\circ$, 
 \begin{equation}\label{eqn:Mrs_Flambda}
  \cM_{r,s}\tens\cF_\lambda\cong\bigoplus_{\ell=0}^{s-1} \cF_{\lambda+\alpha_{r,s}+\ell\alpha_-}.
 \end{equation}
\end{thm}
\begin{proof}
 We prove the $r=1$ case first by induction on $s$. The base case $s=1$ is obvious because $\cM_{1,1}$ is the unit object of $\cO_{\cM(p)}$ and $\alpha_{1,1}=0$, and the $s=2$ case is Lemma \ref{lem:M12_Flambda}. Now suppose we have proved \eqref{eqn:Mrs_Flambda} up to some $s\in\lbrace 2,3,\ldots p-1\rbrace$ and consider $s+1$. As in the proof of Lemma \ref{lem:M12_Flambda}, $\cM_{1,s+1}\tens\cF_\lambda$ is a direct sum of $\cF_\mu$ such that $\lambda+\alpha_{1,s+1}-\mu\in 2 L^\circ$, so we just need to determine which $\cF_\mu$ appear in the tensor product, and with what multiplicity.
 
 On the one hand, associativity, the fusion rules of \cite[Theorem 3.2.8(2)]{CMY-singlet}, and the inductive hypothesis show that $\cM_{1,s+1}\tens\cF_\lambda$ is a summand of:
 \begin{align*}
  \cM_{1,2}\tens(\cM_{1,s}\tens\cF_\lambda) &\cong (\cM_{1,2}\tens\cM_{1,s})\tens\cF_\lambda\nonumber\\
  &\cong(\cM_{1,s+1}\tens\cF_\lambda)\oplus(\cM_{1,s-1}\tens\cF_\lambda)\nonumber\\
  &\cong (\cM_{1,s+1}\tens\cF_\lambda)\oplus\bigoplus_{\ell=0}^{s-2}\cF_{\lambda+\alpha_{1,s-1}+\ell\alpha_-}.
 \end{align*}
On the other hand, the inductive hypothesis, Lemma \ref{lem:M12_Flambda}, and the observation $\alpha_{1,s}\pm\alpha_{1,2}=\alpha_{1,s\pm1}$ yield
\begin{align*}
 \cM_{1,2}\tens(\cM_{1,s}\tens\cF_\lambda) &\cong\bigoplus_{\ell=0}^{s-1} \cM_{1,2}\tens\cF_{\lambda+\alpha_{1,s}+\ell\alpha_-}\nonumber\\
 &\cong\bigoplus_{\ell=0}^{s-1} \left(\cF_{\lambda+\alpha_{1,s}+\alpha_{1,2}+\ell\alpha_-}\oplus\cF_{\lambda+\alpha_{1,s}-\alpha_{1,2}+\ell\alpha_-}\right)\nonumber\\
 &\cong\bigoplus_{\ell=0}^{s-1} \cF_{\lambda+\alpha_{1,s+1}+\ell\alpha_-}\oplus\bigg(\cF_{\lambda+\alpha_{1,s-1}+(s-1)\alpha_-}\oplus\bigoplus_{\ell=0}^{s-2}\cF_{\lambda+\alpha_{1,s-1}+\ell\alpha_-}\bigg).
\end{align*}
Comparing these two computations, we see that
\begin{align*}
 \cM_{1,s+1}\tens\cF_\lambda\cong\cF_{\lambda+\alpha_{1,s-1}+(s-1)\alpha_-}\oplus\bigoplus_{\ell=0}^{s-1}\cF_{\lambda+\alpha_{1,s+1}+\ell\alpha_-} \cong\bigoplus_{\ell=0}^s \cF_{\lambda+\alpha_{1,s+1}+\ell\alpha_-}
\end{align*}
since $\alpha_{1,s-1}-\alpha_-=\alpha_{1,s+1}$. This completes the proof of the $r=1$ case of \eqref{eqn:Mrs_Flambda}.

For general $r$, we need the identity $\cM_{r,1}\tens\cM_{1,s}\cong\cM_{r,s}$ proved in \cite{CMY-singlet}, the $r=1$ case of \eqref{eqn:Mrs_Flambda}, and Lemma \ref{lem:simple_curr_tens}:
\begin{align*}
 \cM_{r,s}\tens\cF_\lambda & \cong(\cM_{r,1}\tens\cM_{1,s})\tens\cF_\lambda\cong\cM_{r,1}\tens(\cM_{1,s}\tens\cF_\lambda)\nonumber\\
 &\cong \bigoplus_{\ell=0}^{s-1} \cM_{r,1}\tens\cF_{\lambda+\alpha_{1,s}+\ell\alpha_-} \cong\bigoplus_{\ell=0}^{s-1}\cF_{\lambda+\alpha_{r,1}+\alpha_{1,s}+\ell\alpha_-}.
\end{align*}
Since $\alpha_{r,1}+\alpha_{1,s}=\alpha_{r,s}$, this yields \eqref{eqn:Mrs_Flambda}.
\end{proof}

We next compute how the projective covers $\cP_{r,s}$ tensor with typical Fock modules, using the preceding theorem and the fusion rules
\begin{equation}\label{eqn:M12_Prs}
 \cM_{1,2}\tens\cP_{r,s}=\left\lbrace\begin{array}{lll}
                                      \cP_{r,s-1}\oplus\cP_{r,s+1} &\text{if}& 2\leq s\leq p-2\\
                                      \cP_{r,p-2}\oplus2\cdot\cP_{r,p} &\text{if}& s=p-1\\
                                      \cP_{r,p-1} &\text{if}& s=p\\
                                     \end{array}\right.
\end{equation}
from \cite[Theorems 3.2.8(2) and 5.1.4(2)]{CMY-singlet}. Note that the first case in the above fusion rules is vacuous when $p =2, 3$, and the second case is valid for $p\geq 3$.

\begin{thm}\label{thm:Prs_Flambda}
 For $r\in\ZZ$, $1\leq s\leq p-1$, and $\lambda\in\CC\setminus L^\circ$,
 \begin{align}\label{eqn:Prs_Flambda}
  \cP_{r,s}\tens\cF_\lambda & \cong(\cM_{r+1,p-s}\tens\cF_\lambda)\oplus 2\cdot(\cM_{r,s}\tens\cF_\lambda)\oplus(\cM_{r-1,p-s}\tens\cF_\lambda)\nonumber\\
  & \cong\bigoplus_{\ell =0}^{p-1} (\cF_{\lambda+\alpha_{r,s}+\ell\alpha_-}\oplus\cF_{\lambda+\alpha_{r-1,p-s}+\ell\alpha_-}).
 \end{align}

\end{thm}

\begin{proof}
We first note that the two formulas given for $\cP_{r,s}\tens\cF_\lambda$ are equivalent because by the fusion rule \eqref{eqn:Mrs_Flambda} and the identity $\alpha_{r+1,p-s}=\alpha_{r,s}+s\alpha_-$,
\begin{align*}
 (\cM_{r,s}\tens\cF_\lambda) & \oplus(\cM_{r+1,p-s}\tens\cF_\lambda)\nonumber\\
 &\cong \bigoplus_{\ell=0}^{s-1} \cF_{\lambda+\alpha_{r,s}+\ell\alpha_-}\oplus\bigoplus_{\ell=0}^{p-s-1}\cF_{\lambda+\alpha_{r+1,p-s}+\ell\alpha_-}\cong\bigoplus_{\ell=0}^{p-1}\cF_{\lambda+\alpha_{r,s}+\ell\alpha_-}
\end{align*}
for $r\in\ZZ$, $1\leq s\leq p-1$, and $\lambda\in\CC\setminus L^\circ$.

 We now prove the theorem by downward induction on $s$ beginning with $s=p-1$. For this case, we use the third case in \eqref{eqn:M12_Prs}, associativity, the fusion rule \eqref{eqn:Mrs_Flambda}, and Lemma \ref{lem:M12_Flambda} to compute
 \begin{align*}
  \cP_{r,p-1}\tens\cF_\lambda & \cong(\cM_{1,2}\tens\cM_{r,p})\tens\cF_\lambda \cong\cM_{1,2}\tens(\cM_{r,p}\tens\cF_\lambda)\nonumber\\
  &\cong\cM_{1,2}\tens\bigoplus_{\ell=0}^{p-1} \cF_{\lambda+\alpha_{r,p}+\ell\alpha_-} \cong\bigoplus_{\ell=0}^{p-1} (\cF_{\lambda+\alpha_{r,p}-\alpha_{1,2}+\ell\alpha_-}\oplus\cF_{\lambda+\alpha_{r,p}+\alpha_{1,2}+\ell\alpha_-}).
 \end{align*}
Since
\begin{equation*}
 \alpha_{r,p}-\alpha_{1,2}=\alpha_{r,p-1},\qquad\alpha_{r,p}+\alpha_{1,2}=\alpha_{r-1,1},
\end{equation*}
this proves \eqref{eqn:Prs_Flambda} for $s=p-1$. This also proves the theorem in the case $p=2$.

For $p\geq 3$, we now prove the $s=p-2$ case of \eqref{eqn:Prs_Flambda}. On the one hand, the middle case of the fusion rule \eqref{eqn:M12_Prs} together with \eqref{eqn:Mrs_Flambda} yield
\begin{align}\label{eqn:sp2_1}
 (\cM_{1,2}\tens\cP_{r,p-1})\tens\cF_\lambda & \cong(\cP_{r,p-2}\tens\cF_\lambda)\oplus2\cdot(\cM_{r,p}\tens\cF_\lambda).
\end{align}
On the other hand, the $s=p-1$ case of \eqref{eqn:Prs_Flambda} together with the fusion rules of \cite[Theorem 3.2.8(2)]{CMY-singlet} yield
\begin{align}\label{eqn:sp2_2}
 \cM_{1,2}\tens(\cP_{r,p-1}\tens\cF_\lambda) &\cong\cM_{1,2}\tens\left((\cM_{r+1,1}\tens\cF_\lambda)\oplus 2\cdot(\cM_{r,p-1}\tens\cF_\lambda)\oplus(\cM_{r-1,1}\tens\cF_\lambda)\right)\nonumber\\
 &\cong (\cM_{r+1,2}\tens\cF_\lambda)\oplus 2\cdot(\cM_{r,p-2}\tens\cF_\lambda)\nonumber\\
 &\hspace{7.8em}\oplus2\cdot(\cM_{r,p}\tens\cF_\lambda)\oplus(\cM_{r-1,2}\tens\cF_\lambda).
\end{align}
After comparing the right sides of \eqref{eqn:sp2_1} and \eqref{eqn:sp2_2}, the Krull-Schmidt Theorem shows that $\cP_{r,p-2}\tens\cF_\lambda$ has the same indecomposable summands as 
\begin{equation*}
 (\cM_{r+1,2}\tens\cF_\lambda)\oplus 2\cdot(\cM_{r,p-2}\tens\cF_\lambda)\oplus(\cM_{r-1,2}\tens\cF_\lambda),
\end{equation*}
proving \eqref{eqn:Prs_Flambda} in the case $s=p-2$.

Now in general for $p\geq 4$, suppose we have proven the $s$ and $s+1$ cases of \eqref{eqn:Prs_Flambda} for some $s\in\lbrace 2,\ldots,p-2\rbrace$; we will prove \eqref{eqn:Prs_Flambda} for $s-1$ using the first case in the fusion rule \eqref{eqn:M12_Prs}. On the one hand,
\begin{align*}
 (\cM_{1,2} & \tens\cP_{r,s})\tens\cF_\lambda  \cong (\cP_{r,s-1}\tens\cF_\lambda)\oplus(\cP_{r,s+1}\tens\cF_\lambda) 
 \nonumber\\
 &\cong (\cP_{r,s-1}\tens\cF_\lambda)\oplus(\cM_{r+1,p-s-1}\tens\cF_\lambda)\oplus2\cdot(\cM_{r,s+1}\tens\cF_\lambda)\oplus(\cM_{r-1,p-s-1}\tens\cF_\lambda).
\end{align*}
On the other hand,
\begin{align*}
 \cM_{1,2}\tens(\cP_{r,s}\tens\cF_\lambda) & \cong\cM_{1,2}\tens\left((\cM_{r+1,p-s}\tens\cF_\lambda)\oplus2\cdot(\cM_{r,s}\tens\cF_\lambda)\oplus(\cM_{r-1,p-s}\tens\cF_\lambda)\right)\nonumber\\
 & \cong (\cM_{r+1,p-s-1}\tens\cF_\lambda)\oplus(\cM_{r+1,p-s+1}\tens\cF_\lambda)\oplus 2\cdot(\cM_{r,s-1}\tens\cF_\lambda)\nonumber\\
 &\qquad\oplus2\cdot(\cM_{r,s+1}\tens\cF_\lambda)\oplus(\cM_{r-1,p-s-1}\tens\cF_\lambda)\oplus(\cM_{r-1,p-s+1}\tens\cF_\lambda).
\end{align*}
Comparing indecomposable summands as before then shows that \eqref{eqn:Prs_Flambda} indeed holds for $s-1$. This proves the theorem.
\end{proof}

\subsection{Typical-typical fusion}

We now compute $\cF_\lambda\tens\cF_\mu$ for $\lambda,\mu\in\CC\setminus L^\circ$ and either $\lambda+\mu\in L^\circ$ or $\lambda+\mu\notin L^\circ$. The first possibility is covered by the following theorem:
\begin{thm}\label{thm:typ_typ_atyp_fusion}
 For $\lambda,\mu\in\CC\setminus L^\circ$ such that $\lambda+\mu=\alpha_0+\alpha_{r,s}$ for some $r\in\ZZ$ and $1\leq s\leq p$,
 \begin{equation*}
  \cF_\lambda\tens\cF_{\mu}\cong\bigoplus_{\substack{s'= s\\ s'\equiv s\,\,(\mathrm{mod}\,2)\\}}^p \cP_{r,s'}\oplus\bigoplus_{\substack{s'=p+2-s\\s'\equiv p-s\,\,(\mathrm{mod}\,2)\\}}^p \cP_{r-1,s'}.
 \end{equation*}
\end{thm}
\begin{proof}
 We first determine the simple quotients of $\cF_\lambda\tens\cF_\mu$, which by Theorem \ref{thm:OT_properties}(3) and Proposition \ref{prop:irred_mod_grading} have the form $\cM_{r',s'}$ for $\alpha_{r',s'}\in\alpha_0+\alpha_{r,s}+2L^\circ$. By symmetries of intertwining operators and the fusion rule \eqref{eqn:Mrs_Flambda},
 \begin{align}\label{eqn:Srs_calc}
&  \hom_{\cM(p)}(  \cF_\lambda\tens\cF_\mu,\cM_{r',s'}) \cong I_{\cM(p)}\binom{\cM_{r',s'}}{\cF_\lambda\,\,\cF_{\alpha_0-\lambda+\alpha_{r,s}}} \cong I_{\cM(p)}\binom{\cF_{\lambda-\alpha_{r,s}}}{\cM_{2-r',s'}\,\cF_\lambda}\nonumber\\
  &\quad\cong \hom_{\cM(p)}(\cM_{2-r',s'}\tens\cF_\lambda,\cF_{\lambda-\alpha_{r,s}})\cong\bigoplus_{\ell=0}^{s'-1}\hom_{\cM(p)}(\cF_{\lambda+\alpha_{2-r',s'}+\ell\alpha_-},\cF_{\lambda-\alpha_{r,s}}).
 \end{align}
Thus $\hom_{\cM(p)}(\cF_\lambda\tens\cF_\mu,\cM_{r',s'})$ is non-zero (and one-dimensional) if and only if
\begin{equation}\label{eqn:Srs_def}
 \alpha_{2-r',s'}+\ell\alpha_-=-\alpha_{r,s}
\end{equation}
for some $\ell\in\lbrace 0,\ldots,s'-1\rbrace$. Let $S_{r,s}$ denote the set of labels $(r',s')$ such that \eqref{eqn:Srs_def} holds.

From the definitions, \eqref{eqn:Srs_def} holds if and only if
\begin{equation*}
 (r-r')p = s'+s-2(\ell+1).
\end{equation*}
Since $0\leq\ell\leq s'-1<p$, we have
\begin{equation*}
 -p+1\leq s'+s-2(\ell+1)\leq 2p-2,
\end{equation*}
which means that $r$ and $r-1$ are the only possibilities for $r'$. If $r'=r$, then we get $\ell+1=\frac{s'+s}{2}$, which must be an integer no larger than $s'$. That is, $(r,s')\in S_{r,s}$ if and only if $s\leq s'\leq p$ and $s'\equiv s\,\,(\mathrm{mod}\,2)$. If $r'=r-1$, then we get $\ell+1=\frac{s'+s-p}{2}$, which must be a positive integer. That is, $(r-1,s')\in S_{r,s}$ if and only if $p+2-s\leq s'\leq p$ and $s'\equiv p-s\,\,(\mathrm{mod}\,2)$.

Thus the theorem amounts to the claim $\cF_\lambda\tens\cF_\mu\cong\bigoplus_{(r',s')\in S_{r,s}} \cP_{r',s'}$. If $(r',p)\in S_{r,s}$ for $r'=r$ or $r'=r-1$, then the non-zero homomorphism $\cF_\lambda\tens\cF_\mu\rightarrow\cM_{r',p}=\cP_{r',p}$ is surjective. Thus because $\cP_{r',p}$ is projective in $\cO_{\cM(p)}^T$, it occurs as a direct summand of $\cF_\lambda\tens\cF_\mu$. For $(r',s')\in S_{r,s}$ with $1\leq s\leq p-1$, we can repeat the calculation \eqref{eqn:Srs_calc} with $\cM_{r',s'}$ replaced by $\cP_{r',s'}$ and then apply the fusion rule \eqref{eqn:Prs_Flambda}. It follows that there are two linearly independent $\cM(p)$-module homomorphisms $\cF_\lambda\tens\cF_\mu\rightarrow\cP_{r',s'}$. At most one of these homomorphisms has image contained in $\mathrm{Soc}(\cP_{r',s'})\cong\cM_{r',s'}$ since $\dim\hom_{\cM(p)}(\cF_\lambda\tens\cF_\mu,\cM_{r',s'})=1$ for such $(r',s')$. Thus from the Loewy diagram \eqref{eqn:Prs_Loewy_diag} of $\cP_{r',s'}$, the second linearly independent homomorphism $\cF_\lambda\tens\cF_\mu\rightarrow\cP_{r',s'}$ is either surjective or has image containing at least one of $\cM_{r'\pm1,p-s'}$ as simple quotient. But the latter option is impossible: it would imply $\hom_{\cM(p)}(\cF_\lambda\tens\cF_\mu,\cM_{r'\pm1,p-s'})\neq 0$, whereas it is easy to see that if $(r',s')\in S_{r,s}$, then $(r'\pm 1,p-s')\notin S_{r,s}$. Consequently, there is a surjective map $\cF_\lambda\tens\cF_\mu\rightarrow\cP_{r',s'}$ for all $(r',s')\in S_{r,s}$, and then $\cP_{r',s'}$ occurs as a summand of $\cF_\lambda\tens\cF_\mu$ by projectivity.

Since $\cF_\lambda\tens\cF_\mu$ is a finite-length $\cM(p)$-module and $\cP_{r',s'}$ is an indecomposable direct summand for any $(r',s')\in S_{r,s}$, there exists for each such $(r',s')$ a decomposition of $\cF_\lambda\tens\cF_\mu$ into a direct sum of indecomposable submodules that includes $\cP_{r',s'}$. Then because the indecomposable submodules appearing in such a decomposition are unique up to isomorphism by the Krull-Schmidt Theorem, and because the $\cP_{r',s'}$ for different $(r',s')$ are non-isomorphic, $\cF_\lambda\tens\cF_\mu\cong X\oplus\bigoplus_{(r',s')\in S_{r,s}} \cP_{r',s'}$ for some $X$. Thus for any $r''\in\ZZ$ and $1\leq s''\leq p$,
\begin{equation*}
 \hom_{\cM(p)}(\cF_\lambda\tens\cF_\mu,\cM_{r'',s''}) \cong\hom_{\cM(p)}(X,\cM_{r'',s''})\oplus\bigoplus_{(r',s')\in S_{r,s}}\hom_{\cM(p)}(\cP_{r',s'},\cM_{r'',s''}).
\end{equation*}
But since
\begin{equation*}
 \dim\hom_{\cM(p)}(\cF_\lambda\tens\cF_\mu,\cM_{r'',s''}) =\left\lbrace\begin{array}{lll}
                                                                        1 & \text{if} & (r'',s'')\in S_{r,s}\\
                                                                        0 & \text{if} & (r'',s'')\notin S_{r,s}\\
                                                                      \end{array}\right.
\end{equation*}
and since 
$$\dim\hom_{\cM(p)}(\cP_{r',s'},\cM_{r'',s''})=\delta_{(r',s'),(r'',s'')},$$
this means $\hom_{\cM(p)}(X,\cM_{r'',s''})=0$ for all $r''\in\ZZ$ and $1\leq s''\leq p$. Since $X$ is a (finite-length) module in $\cO^{\alpha_0+\alpha_{r,s}+2L^\circ}_{\cM(p)}$, it follows that $X=0$. This proves the theorem.
\end{proof}

It is worth recording the $r=s=1$ case of the preceding theorem as a corollary:
\begin{cor}\label{cor:Flambda_Flambda_prime}
 For $\lambda\in\CC\setminus L^\circ$, $\cF_\lambda\tens\cF_{\alpha_0-\lambda}\cong\bigoplus_{s\,\mathrm{odd}} \cP_{1,s}$.
\end{cor}

Using this corollary and the fusion rule \eqref{eqn:Prs_Flambda}, we also prove the following lemma that we will need for the next theorem:
\begin{lem}\label{lem:trip_F_tens}
 For $\lambda,\mu\in\CC\setminus L^\circ$,
 \begin{equation*}
  (\cF_\lambda\tens\cF_\mu)\tens\cF_{\alpha_0-\mu}\cong\bigoplus_{\ell,\ell'=0}^{p-1} \cF_{\lambda+\alpha_0+(\ell+\ell')\alpha_-}.
 \end{equation*}
\end{lem}
\begin{proof}
 We calculate
 \begin{align*}
  (\cF_\lambda\tens\cF_\mu)\tens\cF_{\alpha_0-\mu} &\cong\cF_\lambda\tens(\cF_\mu\tens\cF_{\alpha_0-\mu})\nonumber\\
  &\cong\bigoplus_{s\,\text{odd}}(\cP_{1,s}\tens\cF_{\lambda}) \cong\bigoplus_{s\,\text{odd}}\bigoplus_{\ell=0}^{p-1}(\cF_{\lambda+\alpha_{1,s}+\ell\alpha_-}[\oplus\cF_{\lambda+\alpha_{0,p-s}+\ell\alpha_-}]),
 \end{align*}
where the summand in brackets does not occur for $s=p$ (in case $p$ is odd). Since
\begin{equation*}
 \alpha_{1,s}=\alpha_0+\left(p-\frac{s+1}{2}\right)\alpha_-,\qquad\alpha_{0,p-s}=\alpha_0+\frac{s-1}{2}\alpha_-,
\end{equation*}
we have
\begin{equation*}
 \bigoplus_{s\,\text{odd}}(\cF_{\lambda+\alpha_{1,s}+\ell\alpha_-}[\oplus\cF_{\lambda+\alpha_{0,p-s}+\ell\alpha_-}]) =\bigoplus_{\ell'=0}^{p-1}\cF_{\lambda+\alpha_0+(\ell+\ell')\alpha_-}
\end{equation*}
for all $\ell$, proving the lemma.
\end{proof}

It remains to compute $\cF_\lambda\tens\cF_\mu$ when $\lambda+\mu\notin L^\circ$:
\begin{thm}\label{thm:typ_typ_typ_fusion}
For $\lambda,\mu\in\CC\setminus L^\circ$ such that $\lambda+\mu\in\CC\setminus L^\circ$,
\begin{equation*}
\cF_{\lambda}\tens\cF_\mu \cong \bigoplus_{\ell=0}^{p-1} \cF_{\lambda + \mu + \ell \alpha_-}.
\end{equation*}
\end{thm} 
\begin{proof}
First we claim that 
\begin{equation}\label{typtyp1}
\cF_{\lambda}\tens\cF_\mu\cong \bigoplus_{\ell=0}^{2(p-1)}  a_\ell(\lambda, \mu) \cF_{\lambda + \mu + \ell \alpha_-}
\end{equation}
for certain multiplicities $a_\ell(\lambda, \mu)$. The theorem will then reduce to showing that $a_\ell(\lambda, \mu) =1$ for $\ell \in \{0, \dots,  p-1\}$ and vanishes otherwise. To prove the claim, note that Heisenberg fusion rules and projectivity of $\cF_\lambda$ in $\cO^T_{\cM(p)}$ imply that $\cF_\lambda$ is a direct summand of $\cF_{\lambda+\mu-\alpha_0}\tens\cF_{\alpha_0-\mu}$:
\begin{equation*}
 \cF_{\lambda+\mu-\alpha_0}\tens\cF_{\alpha_0-\mu}\cong\cF_\lambda\oplus X
\end{equation*}
for some $X$. Then using Lemma \ref{lem:trip_F_tens},
\begin{align*}
 (\cF_\lambda\tens\cF_\mu)\oplus(X\tens\cF_\mu) \cong (\cF_{\lambda+\mu-\alpha_0}\tens\cF_{\alpha_0-\mu})\tens\cF_\mu\cong\bigoplus_{\ell,\ell'=0}^{p-1}\cF_{\lambda+\mu+(\ell+\ell')\alpha_-},
\end{align*}
proving \eqref{typtyp1}. It remains to determine the coefficients $a_\ell(\lambda, \mu)$.

Each $a_\ell(\lambda, \mu) \leq 1$ since $\cF_\lambda$, $\cF_\mu$, and $\cF_{\lambda+\mu+\ell\alpha_-}$ are simple Virasoro Verma modules, and thus $\dim I_{\cL(p)}\binom{\cF_{\lambda+\mu+\ell\alpha_-}}{\cF_\lambda\,\,\cF_\mu}=1$ (recall Remark \ref{rem:pi_injective}). Now by Lemma \ref{lem:trip_F_tens} again,
\begin{align}\label{eq:anotheridentity}
\bigoplus_{\ell , \ell'=0}^{p-1} \cF_{\lambda+\alpha_0+(\ell+\ell')\alpha_-} &\cong  (\cF_{\lambda} \tens \cF_{\mu}) \tens \cF_{\alpha_0-\mu} \nonumber\\
&\cong \bigoplus_{\ell = 0}^{2(p-1)} a_\ell(\lambda, \mu)  (\cF_{\lambda+\mu+\ell \alpha_-} \tens \cF_{\alpha_0-\mu}) \nonumber\\
&\cong \bigoplus_{\ell, \ell' = 0}^{2(p-1)} a_\ell(\lambda, \mu) a_{\ell'}(\lambda+\mu+\ell\alpha_-, \alpha_0-\mu)   \cF_{\lambda+\alpha_0+(\ell+\ell') \alpha_-}.
\end{align}
 In particular, for $0\leq n\leq p-1$, we get the identity 
\begin{equation}\label{eq:anidentity}
\sum_{\ell +\ell'=n}  a_\ell(\lambda, \mu) a_{\ell'}(\lambda+\mu+\ell\alpha_-, \alpha_0-\mu)   =  n+1.
\end{equation}
We prove that $a_\ell(\lambda, \mu) =1$ for $\ell \in \{0, \dots, p-1\}$ by induction on $\ell$. The case $\ell = 0$ follows from the $n=0$ case  of \eqref{eq:anidentity} (that is, from $a_0(\lambda, \mu)  a_{0}(\lambda+\mu, \alpha_0-\mu)=1$), or from the fact that there is a Heisenberg intertwining operator of type  $\binom{\cF_{\lambda+\mu}}{\cF_{{\lambda}}\,\,\cF_{\mu}}$. 
Now assume that $a_\ell(\lambda, \mu) =1$ for $\ell \in \{0, \dots, n-1\}$ and $n \leq p-1$.  Then \eqref{eq:anidentity} together with the induction hypothesis and $a_{0}(\lambda+\mu+n\alpha_-, \alpha_0-\mu)=1$  yields 
\begin{equation}\nonumber
a_n(\lambda, \mu) +   \sum_{\ell'=0}^{n-1}  a_{\ell'}(\lambda+\mu+(n-\ell')\alpha_-, \alpha_0-\mu)   = n+1.
\end{equation}
Since all coefficients are $0$ or $1$, the only possibility is $a_n(\lambda, \mu) = 1 = a_{\ell'}(\lambda+\mu+\ell\alpha_-, \alpha_0-\mu)$. Thus $a_\ell(\lambda, \mu) =1$ for $\ell \in \{0, \dots, p-1\}$ and any $\lambda, \mu$ such that $\lambda, \mu, \lambda+\mu \notin L^\circ$. In particular, $a_{\ell'}(\lambda+\mu+\ell\alpha_-, \alpha_0-\mu)=1$ for $\ell, \ell' \in \{0, \dots, p-1\}$.
Thus \eqref{eq:anotheridentity} can only hold if $a_\ell(\lambda, \mu) = 0$ for $\ell > p-1$.
\end{proof}

\section{Rigidity}\label{sec:rigidity}

In this section, we show that the tensor categories $\cO_{\cM(p)}$ and $\cO_{\cM(p)}^T$ are rigid. Thanks to \cite[Theorem 4.4.1]{CMY-singlet}, it is enough to show that all simple $\cM(p)$-modules are rigid. We already showed in \cite{CMY-singlet} that the atypical irreducible modules $\cM_{r,s}$ for $r\in\ZZ$, $1\leq s\leq p$ are rigid, so it remains to consider the typical Fock modules $\cF_\lambda$ for $\lambda\in\CC\setminus L^\circ$. The idea is to choose evaluations $e_\lambda: \cF_{\alpha_0-\lambda}\tens\cF_\lambda\rightarrow\cM_{1,1}$ and coevaluations $i_\lambda:\cM_{1,1}\rightarrow\cF_\lambda\tens\cF_{\alpha_0-\lambda}$ in such a way that (at least one matrix coefficient of) the rigidity composition
\begin{align*}
 \cF_\lambda\xrightarrow{l_{\cF_\lambda}^{-1}}\cM_{1,1}\tens\cF_\lambda\xrightarrow{i_\lambda\tens\Id_{\cF_\lambda}} & (\cF_\lambda\tens\cF_{\alpha_0-\lambda})\tens\cF_\lambda\nonumber\\
 \xrightarrow{\cA_{\cF_\lambda,\cF_{\alpha_0-\lambda},\cF_\lambda}^{-1}} & \cF_\lambda\tens(\cF_{\alpha_0-\lambda}\tens\cF_\lambda)\xrightarrow{\Id_{\cF_\lambda}\tens e_\lambda} \cF_\lambda\tens\cM_{1,1}\xrightarrow{r_{\cF_\lambda}} \cF_\lambda
\end{align*}
depends analytically on $\lambda$. Since this composition is non-zero for $\lambda=\alpha_{r,p}$, $r\in\ZZ$, it does not vanish identically and thus will be non-zero on a dense open set of $\lambda$. We will then use the fusion rules of the previous section to prove rigidity for all $\lambda$. We begin with the construction of suitable evaluation and coevaluation candidates. 

\subsection{Evaluation and coevaluation for typical Fock modules}\label{subsec:eval_and_coeval}

For all $\lambda\in\CC$, we fix a non-zero lowest conformal weight vector $v_\lambda\in\cF_\lambda$ of conformal weight $h_\lambda=\frac{1}{2}\lambda(\lambda-\alpha_0)$, and we identify $\cF_{\alpha_0-\lambda}$ as the contragredient of $\cF_\lambda$ via the unique non-degenerate $\mathcal{H}$-invariant bilinear form
\begin{equation*}
\langle\cdot,\cdot\rangle: \cF_{\alpha_0-\lambda}\times\cF_\lambda\rightarrow\CC
\end{equation*}
such that $\langle v_{\alpha_0-\lambda},v_\lambda\rangle=1$. We will prove that $\cF_\lambda$ is a rigid $\cM(p)$-module with dual $\cF_{\alpha_0-\lambda}$.

We first need an evaluation $e_\lambda: \cF_{\alpha_0-\lambda}\boxtimes\cF_\lambda\rightarrow\cM_{1,1}$. By intertwining operator symmetries from \cite[Equations 3.77 and 3.87]{HLZ2}, we get an intertwining operator $A_0(\Omega_0(Y_{\cF_{\alpha_0-\lambda}}))$ of type $\binom{\cM_{1,1}'}{\cF_{\alpha_0-\lambda}\,\cF_\lambda}$. Identifying $\cM_{1,1}'\cong\cM_{1,1}$ via the non-degenerate invariant bilinear form $(\cdot,\cdot)$ such that $(\vac,\vac)=1$, $A_0(\Omega_0(Y_{\cF_{\alpha_0-\lambda}}))$ becomes the intertwining operator of type $\binom{\cM_{1,1}}{\cF_{\alpha_0-\lambda}\,\cF_\lambda}$ such that
\begin{align}\label{eqn:unscaled_eval}
 \big(v,A_0(\Omega_0(Y_{\cF_{\alpha_0-\lambda}}))(w',x)w\big) =\big\langle Y_{\cF_{\alpha_0-\lambda}}(v,-x^{-1})e^{xL(1)}x^{-2L(0)}e^{\pi iL(0)} w', e^{xL(1)}w\big\rangle
\end{align}
for $v\in\cM_{1,1}$, $w\in\cF_\lambda$, and $w'\in\cF_{\alpha_0-\lambda}$. To simplify the dependence on $\lambda$, we rescale by setting $\cE_\lambda=e^{-\pi i h_\lambda}A_0(\Omega_0(Y_{\cF_{\alpha_0-\lambda}}))$. Then from \eqref{eqn:unscaled_eval} and the scaling $(\vac,\vac)=1$,
\begin{equation}\label{eqn:scaled_eval}
 \cE_\lambda(v_{\alpha_0-\lambda},x)v_\lambda\in x^{-2h_\lambda}\big(\vac+x\cM_{1,1}[[x]]\big).
\end{equation}
Now the universal property of vertex algebraic tensor products yields a unique evaluation $e_\lambda: \cF_{\alpha_0-\lambda}\tens\cF_\lambda\rightarrow\cM_{1,1}$ such that $e_\lambda\circ\cY_{\alpha_0-\lambda}=\cE_\lambda$, where $\cY_{\alpha_0-\lambda}$ denotes the tensor product intertwining operator of type $\binom{\cF_{\alpha_0-\lambda}\tens\cF_\lambda}{\cF_\lambda\,\cF_{\alpha_0-\lambda}}$. Likewise, $\cY_\lambda$ denotes the tensor product intertwining operator of type $\binom{\cF_\lambda\tens\cF_{\alpha_0-\lambda}}{\cF_\lambda\,\cF_{\alpha_0-\lambda}}$.

Next we construct a coevaluation $i_\lambda:\cM_{1,1}\rightarrow\cF_\lambda\tens\cF_{\alpha_0-\lambda}$ in the case that $\cF_\lambda$ is simple as an $\cM(p)$-module, that is, $\lambda\in(\CC\setminus L^\circ)\cup\lbrace\alpha_{r,p}\,\vert\,r\in\ZZ\rbrace$. Certainly a unique (up to scale) non-zero such map exists because the fusion rule
\begin{equation*}
 \cF_\lambda\tens\cF_{\alpha_0-\lambda} \cong\bigoplus_{s\,\text{odd}} \cP_{1,s}
\end{equation*}
of Corollary \ref{cor:Flambda_Flambda_prime} and the structure of $\cP_{1,1}$ as an $\cM(p)$-module shows that $\cF_\lambda\tens\cF_{\alpha_0-\lambda}$ contains a unique submodule isomorphic to $\cM_{1,1}$. (Corollary \ref{cor:Flambda_Flambda_prime} also holds for $\lambda=\alpha_{r,p}$, $r\in\ZZ$, by the $r'=2-r$, $s=s'=p$ case of the fusion rules in \cite[Theorem 5.2.1(1)]{CMY-singlet}.) However, we would like to choose $i_\lambda$ in such a way that its dependence on $\lambda$ is not too arbitrary. To choose a suitable $i_\lambda$, note first that
\begin{equation*}
L(0)(\cP_{1,1})_{[0]}\subseteq\mathrm{Soc}(\cP_{1,1})_{[0]}\cong(\cM_{1,1})_{[0]}=\CC\vac,
\end{equation*}
so we might at first attempt to define $i_\lambda$ such that $i_\lambda(\vac)$ is the coefficient of $x^{-2h_\lambda}$ in $L(0)\cY_\lambda(v_\lambda,x)v_{\alpha_0-\lambda}$ (since this coefficient is a vector in $(\cF_\lambda\tens\cF_{\alpha_0-\lambda})_{[0]}$).
But this will not work because $(\cP_{1,s})_{[0]}$ might be non-zero for some $s\neq 1$, so that the coefficient of $x^{-2h_\lambda}$ in $\cY_\lambda(v_\lambda,x)v_{\alpha_0-\lambda}$ might involve a contribution from such $\cP_{1,s}$. 
The next lemma, whose proof uses the higher-level Zhu algebras $A_N(\cM(p))$ of \cite{DLM}, provides a way to filter out such unwanted contributions; recall that for $v\in\cM(p)$, $o(v)$ is the component of an $\cM(p)$-module vertex operator that preserves conformal weights:
\begin{lem}
 There exists $v\in \cM(p)$, independent of $\lambda$, such that for all $\lambda\in(\CC\setminus L^\circ)\cup\lbrace\alpha_{r,p}\,\vert\,r\in\ZZ\rbrace$,
 \begin{equation*}
 o(v)\cdot\mathrm{Res}_x\,x^{2h_\lambda-1}\cY_\lambda(v_\lambda,x)v_{\alpha_0-\lambda}\in(\cF_\lambda\tens\cF_{\alpha_0-\lambda})_{[0]}
 \end{equation*}
 generates the unique submodule of $\cF_\lambda\tens\cF_{\alpha_0-\lambda}$ isomorphic to $\cM_{1,1}$.
\end{lem}
\begin{proof}
 First, $(\cF_\lambda\tens\cF_{\alpha_0-\lambda})_{[0]}$ is an $A_N(\cM(p))$-module for $N=\left\lfloor\frac{(p-1)^2}{4p}\right\rfloor$ since $h_{1,p}=-\frac{(p-1)^2}{4p}$ is the minimum of all conformal weights $h_{1,s}$. By Corollary \ref{cor:Flambda_Flambda_prime}, there is an $A_N(\cM(p))$-module isomorphism
 \begin{equation*}
  f_\lambda: (\cF_\lambda\tens\cF_{\alpha_0-\lambda})_{[0]}\longrightarrow\bigoplus_{s\,\text{odd}} (\cP_{1,s})_{[0]}.
 \end{equation*}
Then since
\begin{equation*}
 L(0)(\cP_{1,s})_{[0]}\subseteq\mathrm{Soc}(\cP_{1,s})_{[0]}\cong(\cM_{1,s})_{[0]},
\end{equation*}
for each $s$, we have
\begin{equation*}
 L(0)\cdot\mathrm{Res}_x\,x^{2h_\lambda-1}\cY_\lambda(v_\lambda,x)v_{\alpha_0-\lambda}\in S
\end{equation*}
where $S$ is an $A_N(\cM(p))$-submodule of $(\cF_\lambda\tens\cF_{\alpha_0-\lambda})_{[0]}$ isomorphic to $\bigoplus_{s\,\text{odd}}(\cM_{1,s})_{[0]}$.

As each $\cM_{1,s}$ is a simple $\cM(p)$-module, $S$ is a semisimple $A_N(\cM(p))$-module. Moreover, the distinct non-zero $(\cM_{1,s})_{[0]}$ are non-isomorphic $A_N(\cM(p))$-modules: If $(\cM_{1,s})_{[0]}\cong(\cM_{1,s'})_{[0]}$, then $\cM_{1,s}$ and $\cM_{1,s'}$ are both simple quotients of the generalized Verma $\cM(p)$-module $G_N((\cM_{1,s})_{[0]})$. So recalling from Subsection \ref{subsec:VOAs_and_intw_ops} that $G_N((\cM_{1,s})_{[0]})$ has a unique simple quotient, $\cM_{1,s}\cong\cM_{1,s'}$ as desired. Now by the Jacobson Density Theorem, there exists $[u]\in A_N(\cM(p))$ (independent of $\lambda$) such that $o(u)$ acts on $\bigoplus_{s\,\text{odd}} (\cM_{1,s})_{[0]}$ by $\bigoplus_{s\,\text{odd}} \delta_{s,1}\Id_{(\cM_{1,s})_{[0]}}$. So taking $v\in\cM(p)$ such that $[v]=[u][\omega]$ in $A_N(\cM(p))$,
\begin{align*}
 o(v)\cdot\mathrm{Res}_x\,x^{2h_\lambda-1}\cY_\lambda(v_\lambda,x)v_{\alpha_0-\lambda} & = o(u)L(0)\cdot\mathrm{Res}_x\,x^{2h_\lambda-1}\cY_\lambda(v_\lambda,x)v_{\alpha_0-\lambda}\nonumber\\
 & = (f_\lambda^{-1}\circ o(u)\circ f_\lambda)\left(L(0)\cdot\mathrm{Res}_x\,x^{2h_\lambda-1}\cY_\lambda(v_\lambda,x)v_{\alpha_0-\lambda}\right)\nonumber\\
 & =(f_\lambda^{-1}\circ p_1\circ f_\lambda)\left(L(0)\cdot\mathrm{Res}_x\,x^{2h_\lambda-1}\cY_\lambda(v_\lambda,x)v_{\alpha_0-\lambda}\right)
\end{align*}
where $p_1:\bigoplus_{s\,\text{odd}}\cP_{1,s}\rightarrow\cP_{1,1}$ is the $\cM(p)$-module projection. Consequently,
\begin{equation*}
 o(v)\cdot\mathrm{Res}_x\,x^{2h_\lambda-1}\cY_\lambda(v_\lambda,x)v_{\alpha_0-\lambda}
\end{equation*}
is a vector of conformal weight $0$ in the $\cM(p)$-submodule of $\cF_\lambda\tens\cF_{\alpha_0-\lambda}$ isomorphic to $\cM_{1,1}$.

We still need to show that $o(v)\cdot\mathrm{Res}_x\,x^{2h_\lambda-1}\cY_\lambda(v_\lambda,x)v_{\alpha_0-\lambda}\neq 0$. To do so, observe that
\begin{equation*}
 e_{\alpha_0-\lambda}\left(\mathrm{Res}_x\,x^{2h_\lambda-1}\cY_\lambda(v_\lambda,x)v_{\alpha_0-\lambda}\right) =\vac\neq 0
\end{equation*}
by \eqref{eqn:scaled_eval}. Thus because $f_\lambda^{-1}((\cP_{1,s})_{[0]})\subseteq\ker e_{\alpha_0-\lambda}$ for $s>1$, we must have
\begin{equation*}
 f_\lambda\left(\mathrm{Res}_x\,x^{2h_\lambda-1}\cY_\lambda(v_\lambda,x)v_{\alpha_0-\lambda}\right)\notin\ker p_1.
\end{equation*}
Moreover, $(p_1\circ f_\lambda)\left(\mathrm{Res}_x\,x^{2h_\lambda-1}\cY_\lambda(v_\lambda,x)v_{\alpha_0-\lambda}\right)$ is not in the maximal proper submodule of $\cP_{1,1}$ because this submodule is contained in the kernel of any homomorphism to $\cP_{1,1}\rightarrow\cM_{1,1}$. Thus by the structure of $\cP_{1,1}$, 
\begin{equation*}
 L(0)(p_1\circ f_\lambda)\left(\mathrm{Res}_x\,x^{2h_\lambda-1}\cY_\lambda(v_\lambda,x)v_{\alpha_0-\lambda}\right)\neq 0,
\end{equation*}
and then
\begin{align*}
 o(v)\cdot\mathrm{Res}_x\,x^{2h_\lambda-1}\cY_\lambda(v_\lambda,x)v_{\alpha_0-\lambda} 
 & =(f_\lambda^{-1}\circ p_1\circ f_\lambda)\left(L(0)\cdot\mathrm{Res}_x\,x^{2h_\lambda-1}\cY_\lambda(v_\lambda,x)v_{\alpha_0-\lambda}\right)\nonumber\\
 & =f_\lambda^{-1}\left(L(0)(p_1\circ f_\lambda)(\mathrm{Res}_x\,x^{2h_\lambda-1}\cY_\lambda(v_\lambda,x)v_{\alpha_0-\lambda})\right)\neq 0
\end{align*}
as well.
\end{proof}

For $\lambda\in(\CC\setminus L^\circ)\cup\lbrace\alpha_{r,p}\,\vert\,r\in\ZZ\rbrace$, we now have a non-zero coevaluation candidate
\begin{align*}
 i_\lambda: \cM_{1,1} & \rightarrow \cF_\lambda\tens\cF_{\alpha_0-\lambda}\nonumber\\
 \vac & \mapsto o(v)\cdot\mathrm{Res}_x\, x^{2h_\lambda-1}\cY_\lambda(v_\lambda,x)v_{\alpha_0-\lambda},
\end{align*}
with $v\in\cM(p)$ as in the preceding lemma. We can rewrite the formula for $i_\lambda$ using the commutator formula \eqref{eqn:intw_op_comm}; for $j\in\NN$, we use $\pi_j$ to denote the projection from $\cM(p)$ to the conformal weight space $\cM(p)_{(j)}$:
\begin{align}\label{eqn:i_lambda_vac}
 i_\lambda(\vac) & = o(v)\cdot\mathrm{Res}_x\,x^{2h_\lambda-1}\cY_\lambda(v_\lambda,x)v_{\alpha_0-\lambda}\nonumber\\
 & =\mathrm{Res}_x\,x^{2h_\lambda-1}\bigg(\cY_\lambda(v_\lambda,x) o(v)v_{\alpha_0-\lambda}+\sum_{j\geq 0}\sum_{n\geq 0}\binom{j-1}{n} x^{j-n-1}\cY_\lambda(\pi_j(v)_n v_\lambda,x)v_{\alpha_0-\lambda} \bigg)\nonumber\\
 & =\mathrm{Res}_x\,x^{-1}\bigg(\cY_\lambda(x^{L(0)}v_\lambda,x) x^{L(0)}o(v)v_{\alpha_0-\lambda}\nonumber\\
 &\hspace{9em}+\sum_{j\geq 0}\sum_{n\geq 0}\binom{j-1}{n}\cY_\lambda(x^{L(0)}\pi_j(v)_n v_\lambda,x)x^{L(0)}v_{\alpha_0-\lambda} \bigg)\nonumber\\
 & =\sum_{j=1}^J\mathrm{Res}_x\,x^{L(0)-1}\cY_\lambda(u^{(j)}_{n_j}v_\lambda,1)o(v^{(j)})v_{\alpha_0-\lambda}
\end{align}
for suitable homogeneous $u^{(j)}, v^{(j)}\in\cM(p)$ and $n_j\in\ZZ$ (with $u^{(j)},v^{(j)}=\vac$ as appropriate), and the substitution $x\mapsto 1$ in $\cY_\lambda$ is accomplished using the branch of logarithm $\log 1=0$. Since the $u^{(j)}$, $v^{(j)}$, and $n_j$ are independent of $\lambda$, the right side of \eqref{eqn:i_lambda_vac} is defined for all $\lambda\in\CC$, though it might vanish if $\cF_\lambda$ is not simple.

\subsection{The rigidity argument}

In this subsection, we give the proof that all typical Fock modules, and indeed the entire categories $\cO_{\cM(p)}$ and $\cO_{\cM(p)}^T$, are rigid, modulo some complex analytic results that we will prove in the next subsections.

To show that $\cF_\lambda$ is rigid for $\lambda\in\CC\setminus L^\circ$, \cite[Corollary 4.2.2]{CMY3} shows that it is enough to prove that the rigidity composition $\mathfrak{R}_\lambda$ given by
\begin{align*}
 \cF_\lambda\xrightarrow{l_{\cF_\lambda}^{-1}}\cM_{1,1}\tens\cF_\lambda\xrightarrow{i_\lambda\tens\Id_{\cF_\lambda}} & (\cF_\lambda\tens\cF_{\alpha_0-\lambda})\tens\cF_\lambda\nonumber\\
 \xrightarrow{\cA_{\cF_\lambda,\cF_{\alpha_0-\lambda},\cF_\lambda}^{-1}} & \cF_\lambda\tens(\cF_{\alpha_0-\lambda}\tens\cF_\lambda)\xrightarrow{\Id_{\cF_\lambda}\tens e_\lambda} \cF_\lambda\tens\cM_{1,1}\xrightarrow{r_{\cF_\lambda}} \cF_\lambda
\end{align*}
is non-zero. In particular, it is sufficient to show that
\begin{equation}\label{eqn:rig_crit_1}
 \langle v_{\alpha_0-\lambda},\mathfrak{R}_\lambda(v_\lambda)\rangle =\langle v_{\alpha_0-\lambda}, \cY(i_\lambda(\vac),1)v_\lambda\rangle\neq 0,
\end{equation}
where
\begin{equation*}
 \cY=r_{\cF_\lambda}\circ(\Id_{\cF_\lambda}\tens e_\lambda)\circ\cA_{\cF_\lambda,\cF_{\alpha_0-\lambda},\cF_\lambda}^{-1}\circ\cY_\tens;
\end{equation*}
here $\cY_\tens$ is the tensor product intertwining operator of type $\binom{(\cF_{\lambda}\tens\cF_{\alpha_0-\lambda})\tens\cF_\lambda}{\cF_\lambda\tens\cF_{\alpha_0-\lambda}\,\,\cF_\lambda}$. By \eqref{eqn:i_lambda_vac}, $\langle v_{\alpha_0-\lambda},\mathfrak{R}_\lambda(v_\lambda)\rangle$ is the constant term in the formal series
\begin{equation}\label{eqn:i_lambda_power_series}
\varphi(\lambda,x)=\sum_{j=1}^J \langle v_{\alpha_0-\lambda},\cY(x^{L(0)}\cY_\lambda(u^{(j)}_{n_j}v_\lambda, 1)o(v^{(j)})v_{\alpha_0-\lambda},1)v_\lambda\rangle.
\end{equation}
By convergence of iterates of intertwining operators among $C_1$-cofinite $\cM(p)$-modules (see \cite{Hu-diff-eqn} or \cite[Section 11.2]{HLZ7}), we can substitute $x\mapsto\frac{1-z}{z}$ in this formal series (using any choice of branch of logarithm) and get an absolutely convergent series for $ \vert z\vert>\vert 1-z\vert>0$.

We always use $\log$ to denote the principal branch of logarithm on $\CC\setminus(-\infty,0]$:
\begin{equation*}
 \log\zeta =\ln\vert\zeta\vert+i\arg\zeta
\end{equation*}
where $-\pi<\arg\zeta\leq\pi$. For $\zeta=\frac{1-z}{z}$, note that $\log(\frac{1-z}{z})$ defines a single-valued analytic function on $\CC\setminus((-\infty,0]\cup[1,\infty))$, and that this function agrees with $\log(1-z)-\log z$ since both functions obviously agree on the real interval $(0,1)$. Thus substituting $x\mapsto e^{\log(\frac{1-z}{z})}$ in \eqref{eqn:i_lambda_power_series} and using the definitions of $\cY$ and of the unit and associativity isomorphisms in $\cO_{\cM(p)}$ (see \cite{HLZ8} or the exposition in \cite[Section 3.3]{CKM-exts}), we get
\begin{align}\label{eqn:iterate_equals_product}
&\sum_{j=1}^J \left\langle v_{\alpha_0-\lambda},\cY(e^{\log(\frac{1-z}{z})L(0)}\cY_\lambda(u^{(j)}_{n_j}v_\lambda, 1)o(v^{(j)})v_{\alpha_0-\lambda},1)v_\lambda\right\rangle\nonumber\\
& =\sum_{j=1}^J (1-z)^{2h_\lambda+\mathrm{wt}\,u^{(j)}-n_j-1}\left\langle v_{\alpha_0-\lambda}, \cY(e^{-(\log z)L(0)}\cY_\lambda(u^{(j)}_{n_j}v_\lambda,e^{\log(1-z)})o(v^{(j)})v_{\alpha_0-\lambda},1)v_\lambda\right\rangle\nonumber\\
& =\sum_{j=1}^J (1-z)^{2h_\lambda+\mathrm{wt}\,u^{(j)}-n_j-1}\cdot\nonumber\\
&\hspace{4em}\cdot\left\langle e^{-(\log z)L(0)}v_{\alpha_0-\lambda}, \cY(\cY_\lambda(u^{(j)}_{n_j}v_\lambda,e^{\log(1-z)})o(v^{(j)})v_{\alpha_0-\lambda},e^{\log z})e^{(\log z)L(0)}v_\lambda\right\rangle\nonumber\\
& = \sum_{j=1}^J (1-z)^{2h_\lambda+\mathrm{wt}\,u^{(j)}-n_j-1}\left\langle v_{\alpha_0-\lambda}, \Omega_0(Y_{\cF_\lambda})(u^{(j)}_{n_j} v_\lambda, 1)\cE_\lambda(o(v^{(j)})v_{\alpha_0-\lambda},e^{\log z})v_\lambda\right\rangle
\end{align}
for $z$ such that both sides converge, that is, $1>\vert z\vert>\vert 1-z\vert>0$. For such $z$, the factors $(1-z)^{2h_\lambda+\mathrm{wt}\,u^{(j)}-n_j-1}$ agree with their binomial expansions as power series in $z$. Thus the right side of \eqref{eqn:iterate_equals_product} is a series in powers of $z$ which converges absolutely to a multivalued analytic function on the punctured disk $B_1(0)\setminus\lbrace 0\rbrace$ of radius $1$ centered at $0$, or to a single-valued analytic function on the simply-connected open set $B_1(0)\setminus(-1,0]$.

To prove \eqref{eqn:rig_crit_1} and thus show $\cF_\lambda$ is rigid, we will show that $\langle v_{\alpha_0-\lambda},\mathfrak{R}_\lambda(v_\lambda)\rangle$ is an analytic function of $\lambda$ which does not vanish identically, and thus is generically non-zero. For this, we need to enhance the analyticity in $z$ of the right side of \eqref{eqn:iterate_equals_product} to analyticity in both $\lambda$ and $z$; this is the content of the following theorem, which we will prove in the next subsection:
\begin{thm}\label{thm:prod_analytic_modified}
 For any homogeneous $u,v\in\cM(p)$, $n\in\ZZ$, and analytic function $r(\lambda)$ on $\CC$, there is a finite set $S\subseteq\CC$ such that the series
 \begin{equation}\label{eqn:intw_op_prod}
(1-z)^{r(\lambda)}\big\langle v_{\alpha_0-\lambda}, \Omega_0(Y_{\cF_\lambda})(u_n v_\lambda,1)\cE_\lambda(o(v)v_{\alpha_0-\lambda},e^{\log z})v_\lambda\big\rangle
 \end{equation}
converges absolutely to a function analytic in both $\lambda$ and $z$ on $(\CC\setminus S)\times(B_1(0)\setminus(-1,0])$. Moreover, this analytic function is the solution of a differential equation of the form
\begin{equation}\label{eqn:diff_eqn_form}
\frac{d^N\varphi}{dz^N} =\sum_{n=0}^{N-1} a_n(\lambda,z)\dfrac{d^n\varphi}{dz^n}
\end{equation}
whose coefficient functions $a_n(\lambda,z)$ are analytic in $\lambda$ and $z$ on $(\CC\setminus S)\times(\CC\setminus\lbrace 0,1\rbrace)$.
\end{thm}

Recall that $\langle v_{\alpha_0-\lambda},\mathfrak{R}_\lambda(v_\lambda)\rangle$ is the constant term in the formal series $\varphi(\lambda, x)$ of \eqref{eqn:i_lambda_power_series}. Substituting $x\mapsto\zeta$ using the principal branch of logarithm $\log\zeta$ yields a function $\varphi(\lambda,\zeta)$ which is analytic in $\zeta$ for $\zeta\in B_1(0)\setminus(-1,0]$. Moreover, the calculation \eqref{eqn:iterate_equals_product} shows that
\begin{align*}
 \varphi\left(\lambda,\frac{1-z}{z}\right)&=\sum_{j=1}^J (1-z)^{2h_\lambda+\mathrm{wt}\,u^{(j)}-n_j-1}\cdot\nonumber\\
 &\hspace{6em}\cdot\big\langle v_{\alpha_0-\lambda}, \Omega_0(Y_{\cF_\lambda})(u^{(j)}_{n_j} v_\lambda, 1)\cE_\lambda(o(v^{(j)})v_{\alpha_0-\lambda},e^{\log z})v_\lambda\big\rangle
\end{align*}
for $1>\vert z\vert>\vert 1-z\vert>0$. For $1\leq j\leq J$, let $S_j\subseteq\CC$ be the finite set of Theorem \ref{thm:prod_analytic_modified} corresponding to $u^{(j)}, v^{(j)}\in\cM(p)$, $n_j\in\ZZ$, and the analytic function $2h_\lambda+\mathrm{wt}\,u^{(j)}-n_j-1$.

 Now applying Theorem \ref{thm:prod_analytic_modified} and the variable change $z=(1+\zeta)^{-1}$, $\frac{d}{dz}=-(1+\zeta)^2\frac{d}{d\zeta}$, we see that $\varphi(\lambda,\zeta)$ is analytic in both $\lambda$ and $\zeta$ for $\lambda\in\CC\setminus S$, where $S=\cup_{j=1}^J S_j$, and for $\zeta$ such that $\vert 1+\zeta\vert>1>\vert\zeta\vert>0$ (a non-empty simply-connected open subset of $B_1(0)\setminus(-1,0]$ that we will call $V_0$). Moreover, the analytic function $\varphi(\lambda,\zeta)$ on $(\CC\setminus S)\times V_0$ is a finite sum of solutions to differential equations of the form
\begin{equation}\label{eqn:diff_eqn_zeta}
\frac{d^N\varphi}{d\zeta^N} =\sum_{n=0}^{N-1} b_n(\lambda,\zeta)\dfrac{d^n\varphi}{d\zeta^n}
\end{equation}
whose coefficient functions $b_n(\lambda,\zeta)$ are analytic for $\lambda\in\CC\setminus S$ and $\zeta\in\CC\setminus\lbrace -1,0\rbrace$ (when changing variables from $z$ to $\zeta$ in the differential equation \eqref{eqn:diff_eqn_form}, the singularities at $z=0,1,\infty$ become singularities at $\zeta=\infty,0,-1$, respectively).

Since each summand of $\varphi(\lambda,\zeta)$ is analytic in $\zeta$ for all $\zeta\in B_1(0)\setminus(-1,0]$ and also solves a differential equation with analytic coefficients for $\zeta$ in the non-empty open subset $V_0\subseteq B_1(0)\setminus(-1,0]$, each summand of $\varphi(\lambda,\zeta)$ is actually a solution to the differential equation on the entire connected set $B_1(0)\setminus(-1,0]$, for all $\lambda\in\CC\setminus S$. It is not immediately evident that $\varphi(\lambda,\zeta)$ is also analytic in $\lambda$ at any $\zeta\in B_1(0)\setminus(-1,0]$ (outside of $V_0$), but this is proved as part of the proof of the following theorem:
\begin{thm}\label{thm:diff_eqn_thm_2}
 Suppose $\varphi(\lambda,\zeta)=\sum_{m=1}^M\sum_{k=0}^K f_{m,k}(\lambda,\zeta) e^{h_m\log\zeta} (\log\zeta)^k$ is a series solution to a differential equation of the form \eqref{eqn:diff_eqn_zeta} such that:
\begin{itemize}
 \item The $h_m\in\CC$ for $1\leq m\leq M$ are pairwise non-congruent mod $\ZZ$.
 \item For each $m$ and $k$, $f_{m,k}(\lambda,\zeta)=\sum_{n\in\ZZ} f_{m,k,n}(\lambda)\,\zeta^n$ is a Laurent series in $\zeta$ whose coefficients are functions of $\lambda$ defined on a non-empty open subset $U\subseteq\CC$.
 \item The coefficient functions $b_n(\lambda,\zeta)$ in \eqref{eqn:diff_eqn_zeta} are analytic in both $\lambda$ and $\zeta$ on the open set $U\times(B_1(0)\setminus\lbrace 0\rbrace)$.
 \item The series $\varphi(\lambda,\zeta)$ converges absolutely on $U\times(B_1(0)\setminus(-1,0])$, and thus for any $\lambda\in U$, $\varphi(\lambda,\zeta)$ is analytic in $\zeta$ on $B_1(0)\setminus(-1,0]$.
 \item For some non-empty open subset $V_0\subseteq B_1(0)\setminus(-1,0]$, $\varphi(\lambda,\zeta)$ is analytic in both $\lambda$ and $\zeta$ on $U\times V_0$.
\end{itemize}
Then for $1\leq m\leq M$, $0\leq k\leq K$, and all $n\in\ZZ$, the function $f_{m,k,n}(\lambda)$ is analytic on $U$.
\end{thm}

We defer the proof of this theorem to Subsection \ref{subsec:analytic_thm_proof}. We take $\varphi(\lambda,\zeta)$ in the theorem to be the individual summands of \eqref{eqn:i_lambda_power_series} (with $x\mapsto e^{\log\zeta}$). Then the $h_m\in\CC$ for $1\leq m\leq M$ are a maximal set of $h_{1,s}$, $s$ odd, that are pairwise non-congruent mod $\ZZ$ (since $\cF_\lambda\tens\cF_{\alpha_0-\lambda}\cong\bigoplus_{s\,\text{odd}}\cP_{1,s}$) and $K=1$ (since the nilpotent part of $L(0)$ squares to $0$ on any $\cP_{1,s}$). The differential equation of the theorem is that of Theorem \ref{thm:prod_analytic_modified} (with a change of variables from $z$ to $\zeta=\frac{1-z}{z}$), the open set $U$ is $\CC\setminus S$, and $V_0=\lbrace\zeta\in\CC\,\vert\,\vert1+\zeta\vert>1>\vert\zeta\vert>0\rbrace$. Thus taking $h_1=h_{1,1}=0$, Theorem \ref{thm:diff_eqn_thm_2} implies that
\begin{equation*}
 f_{1,0,0}(\lambda)=\langle v_{\alpha_0-\lambda},\mathfrak{R}_\lambda(v_\lambda)\rangle
\end{equation*}
is analytic on $U=\CC\setminus S$. (Although we explicitly defined $\mathfrak{R}_\lambda$ only for $\lambda$ such that $\cF_\lambda$ is simple, the function $\varphi(\lambda,\zeta)$, and thus also the coefficient function $f_{1,0,0}(\lambda)$ is defined for all $\lambda\in\CC$, as noted at the end of Subsection \ref{subsec:eval_and_coeval}.) We can now prove:
\begin{thm}\label{thm:F_lambda_rigid}
 For all $\lambda\in\CC$, the $\cM(p)$-module $\cF_\lambda$ is rigid.
\end{thm}
\begin{proof}
 Since $\langle v_{\alpha_0-\lambda},\mathfrak{R}_\lambda(v_\lambda)\rangle$ is analytic on the connected open set $\CC\setminus S$, it is either identically $0$ or its zeros form a discrete subset of $\CC\setminus S$. We claim that $\langle v_{\alpha_0-\lambda},\mathfrak{R}_\lambda(v_\lambda)\rangle$ is not identically $0$: Indeed, it was shown in \cite{CMY-singlet} that $\cF_{\alpha_{r,p}}$ for $r\in\ZZ$ is rigid and therefore
 \begin{align*}
  &\mathrm{Hom}_{\cM(p)}(\cF_{\alpha_{2-r},p}\tens\cF_{\alpha_{r,p}},\cM_{1,1})\cong\mathrm{End}_{\cM(p)}(\cF_{\alpha_{r,p}})\cong\CC,\nonumber\\
  &\mathrm{Hom}_{\cM(p)}(\cM_{1,1},\cF_{\alpha_{r,p}}\tens\cF_{\alpha_{2-r,p}})\cong\mathrm{End}_{\cM(p)}(\cF_{\alpha_{r,p}})\cong\CC.
 \end{align*}
Thus because $e_{\alpha_{r,p}}$ and $i_{\alpha_{r,p}}$ are non-zero, the actual evaluation and coevaluation for $\cF_{\alpha_{r,p}}$ have to be non-zero multiples of $e_{\alpha_{r,p}}$ and $i_{\alpha_{r,p}}$, respectively. Consequently, $\mathrm{Id}_{\cF_{\alpha_{r,p}}}$ is a non-zero multiple of $\mathfrak{R}_{\alpha_{r,p}}$, so that $\mathfrak{R}_{\alpha_{r,p}}\neq 0$ for $r\in\ZZ$. Since $S$ is a finite set, $\alpha_{r,p}\notin S$ for infinitely many $r\in\ZZ$, and therefore $\langle v_{\alpha_0-\lambda},\mathfrak{R}_\lambda(v_\lambda)\rangle\neq 0$ for infinitely many $\lambda\in\CC\setminus S$. 

The above argument combined with the analyticity of $\langle v_{\alpha_0-\lambda},\mathfrak{R}_\lambda(v_\lambda)\rangle$ shows that there exists $r\in\ZZ$ and $\varepsilon>0$ such that for all $\lambda$ in the open ball $B_\varepsilon(\alpha_{r,p})$ of radius $\varepsilon$ around $\alpha_{r,p}$, $\cF_\lambda$ is rigid as an $\cM(p)$-module (and we may assume $\varepsilon$ is small enough so that $\cF_\lambda$ is simple for all $\lambda\in B_\varepsilon(\alpha_{r,p})$). Then for any $\lambda$ such that $0<\vert\lambda\vert<\varepsilon$, the tensor product $\cF_{\lambda+\alpha_{r,p}}\tens\cF_{\alpha_{2-r,p}}$ is also rigid and contains $\cF_{\lambda+\alpha_{r,p}+\alpha_{2-r,p}+(p-1)\alpha_-}=\cF_\lambda$ as a direct summand by Theorem \ref{thm:Mrs_Flambda}. Thus $\cF_\lambda$ is rigid for all $\lambda\in B_\varepsilon(0)$. 

Now consider any $\lambda\in\CC\setminus L^\circ$ and any $0<\delta<\varepsilon$. Since the zeros of $\langle v_{\alpha_0-\lambda},\mathfrak{R}_\lambda(v_\lambda)\rangle$ form a discrete set of $\CC\setminus S$, the circle $\lbrace\lambda +\mu\in\CC\,\vert\,\vert\mu\vert=\delta\rbrace$ can contain infinitely many such zeros only if it contains one of the (finitely many) elements of $S$. Thus there is some $\delta<\varepsilon$ and some $\mu$ with $\vert\mu\vert=\delta$ such that $\lambda+\mu\in\CC\setminus L^\circ$ and $\mathfrak{R}_{\lambda+\mu}\neq 0$. Thus $\cF_{\lambda+\mu}\tens\cF_{-\mu}$ is rigid and contains $\cF_\lambda$ as a direct summand, proving that $\cF_\lambda$ is rigid for any $\lambda\in\CC\setminus L^\circ$ (and rigidity of $\cF_\lambda$ for $\lambda\in L^\circ$ was proved in \cite{CMY-singlet}).
\end{proof}

Combined with the rigidity results of \cite{CMY-singlet}, the preceding theorem shows that all simple objects of the tensor categories $\cO_{\cM(p)}$ and $\cO_{\cM(p)}^T$ are rigid. As every object of $\cO_{\cM(p)}$ has finite length, \cite[Theorem 4.4.1]{CMY-singlet} shows that all objects of $\cO_{\cM(p)}$ are rigid:
\begin{thm}\label{thm:rig:O}
 The tensor category $\cO_{\cM(p)}$ is rigid and ribbon, with duals given by contragredient modules and ribbon twist $\theta=e^{2\pi i L(0)}$.
\end{thm}

\begin{thm}\label{thm:rig:OT}
 The tensor category $\cO_{\cM(p)}^T$ is rigid and ribbon.
\end{thm}
\begin{proof}
 We just need to check that $\cO_{\cM(p)}^T$ is closed under contragredients. For the subcategory $\cC^0_{\cM(p)}=\cO_{\cM(p)}^{0+2L^\circ}\oplus\cO_{\cM(p)}^{\alpha_-/2+2L^\circ}$, this was done in \cite[Corollary 4.4.3]{CMY-singlet}. For $\lambda\in\CC\setminus L^\circ$, all objects of $\cO_{\cM(p)}^{\lambda+2L^\circ}$ are direct sums of Fock modules $\cF_\mu$ such that $\lambda-\mu\in 2L^\circ$ by Corollary \ref{cor:Olambda_structure}, and the contragredient of such a direct sum is an object of $\cO_{\cM(p)}^{\alpha_0-\lambda+2L^\circ}$ by Proposition \ref{prop:irred_mod_grading}. 
\end{proof}

\subsection{Generic Fock modules and differential equations}

In this subsection, we prove Theorem \ref{thm:prod_analytic_modified}. The idea is repeat Huang's derivation in \cite{Hu-diff-eqn} of regular-singular-point differential equations satisfied by products of intertwining operators such as \eqref{eqn:intw_op_prod}, but in such a way as to guarantee that the coefficient functions of the differential equation are analytic in the Heisenberg weight $\lambda$. To do so, we replace the Fock modules $\cF_\lambda$ and $\cF_{\alpha_0-\lambda}$ in the derivation of the differential equations with a generic Fock module $\cF_x$ on which the Heisenberg zero-mode $h(0)$ acts by the monomial $x$. We define the generic Fock module more precisely as follows.

Recall that the Zhu algebra $A(\mathcal{H})$ of the Heisenberg vertex operator algebra $\mathcal{H}$ is isomorphic to $\CC[x]$. Thus $\CC[x]$ is the top level of the generic Fock $\mathcal{H}$-module $\cF_x:=G(\CC[x])$; we identify $1\in\CC[x]$ with a generating highest-weight vector $v_x\in\cF_x$ such that
\begin{equation*}
 h(n)v_x=\delta_{n,0}x\cdot v_x
\end{equation*}
for $n\in\NN$. The operator $h(0)$ gives $\cF_x$ the structure of a $\CC[x]$-module, and each homogeneous space $\cF_x(n)$, $n\in\NN$, in the natural $\NN$-grading of $\cF_x$ is a finitely-generated $\CC[x]$-module (generated by vectors of the form
\begin{equation*}
 h(-n_1)\cdots h(-n_k)v_x
\end{equation*}
such that each $n_i\in\ZZ_+$ and $n_1+\cdots+n_k=n$).

For each $\lambda\in\CC$, there is an $A(\mathcal{H})$-module homomorphism $\CC[x]\rightarrow\CC v_\lambda\subseteq\cF_\lambda$ sending $p(x)\in\CC[x]$ to $p(\lambda)v_\lambda$. The universal property of generalized Verma $\mathcal{H}$-modules then induces a unique (and surjective) $\mathcal{H}$-module homomorphism $p_\lambda:\cF_x\rightarrow\cF_\lambda$ such that $p_\lambda(v_x)=v_\lambda$.
\begin{lem}\label{lem:ker_p_lambda}
 For each $\lambda\in\CC$, $\ker p_\lambda=(x-\lambda)\cdot\cF_x$.
\end{lem}
\begin{proof}
 From the definition of $p_\lambda$, $(x-\lambda)\cdot\cF_x\subseteq\ker p_\lambda$, so $p_\lambda$ induces a surjective map 
 \begin{equation*}
  \overline{p_\lambda}: \cF_x/(x-\lambda)\cdot\cF_x\longrightarrow\cF_\lambda
 \end{equation*}
such that $\overline{p_\lambda}(v_x+(x-\lambda)\cdot\cF_x)=v_\lambda$ and $\ker\overline{p_\lambda}=\ker p_\lambda/(x-\lambda)\cdot\cF_x$. Setting $\overline{v_x}=v_x+(x-\lambda)\cdot\cF_x$, note that $\overline{v_x}$ is in the top level $T(\cF_x/(x-\lambda)\cdot\cF_x)$, and that $h(0)\overline{v_x}=\lambda\overline{v_x}$. Thus by the universal property of generalized Verma $\mathcal{H}$-modules, there is an $\mathcal{H}$-module homomorphism
\begin{equation*}
 q_\lambda: \cF_\lambda\longrightarrow\cF_x/(x-\lambda)\cdot\cF_x
\end{equation*}
such that $q_\lambda(v_\lambda)=\overline{v_x}$. Then $q_\lambda\circ\overline{p_\lambda}=\Id_{\cF_x/(x-\lambda)\cdot\cF_x}$ since $\overline{v_x}$ generates $\cF_x/(x-\lambda)\cdot\cF_x$. In particular, $\overline{p_\lambda}$ is injective, which implies $\ker p_\lambda=(x-\lambda)\cdot\cF_x$.
\end{proof}

Using $\cF_x$ and the homomorphisms $p_\lambda$, we now prove that the coefficients of series similar to \eqref{eqn:intw_op_prod} depend polynomially on $\lambda$:
\begin{prop}\label{prop:poly_coeff}
 Fix homogeneous $u,v\in\cM(p)$ and $n\in\ZZ$. Then for all $\lambda\in\CC$,
 \begin{equation}\label{eqn:poly_coeff}
  \left\langle v_{\alpha_0-\lambda},\Omega_0(Y_{\cF_\lambda})(u_n v_{\lambda},1)\cE_\lambda(o(v)v_{\alpha_0-\lambda}, z)v_\lambda\right\rangle =\sum_{m\geq 0} q_m(\lambda)\,z^{-2h_\lambda+m}
 \end{equation}
as formal series in powers of $z$, where $q_m(\lambda)\in\CC[\lambda]$ are polynomials depending only on $u$, $v$, and $n$.
\end{prop}
\begin{proof}
 Let $\lbrace v_i\rbrace_{i\in I}$ be a basis of $\cM(p)$ consisting of homogeneous vectors, and let $\lbrace v_i'\rbrace_{i\in I}$ be the dual basis with respect to the nondegenerate invariant bilinear form $(\cdot,\cdot)$ such that $(\vac,\vac)=1$. The definitions of $\Omega_0(Y_{\cF_\lambda})$ and $\cE_\lambda$ (recall in particular \eqref{eqn:unscaled_eval}) yield
 \begin{align*}
  \langle v_{\alpha_0-\lambda},&\,\Omega_0(Y_{\cF_\lambda})  (u_n v_{\lambda},1)\cE_\lambda(o(v)v_{\alpha_0-\lambda}, z)v_\lambda\rangle\nonumber\\
  &= \sum_{i\in I}\left\langle v_{\alpha_0-\lambda},\Omega_0(Y_{\cF_\lambda})(u_nv_\lambda,1)v_i\right\rangle\left( v_i',\cE_\lambda(o(v)v_{\alpha_0-\lambda},z)v_\lambda\right)\nonumber\\
  & =\sum_{i\in I}\left\langle e^{L(1)} v_{\alpha_0-\lambda}, Y_{\cF_\lambda}(v_i,-1)u_n v_\lambda\right\rangle e^{-\pi ih_\lambda}\cdot\nonumber\\
  &\hspace{6em}\cdot\left\langle Y_{\cF_{\alpha_0-\lambda}}(v_i',-z^{-1})e^{zL(1)} z^{-2L(0)} e^{\pi iL(0)} o(v)v_{\alpha_0-\lambda}, e^{zL(1)}v_{\lambda}\right\rangle\nonumber\\
  & =\sum_{i\in I} z^{-2h_\lambda}\left\langle v_{\alpha_0-\lambda}, Y_{\cF_\lambda}(v_i,-1)u_n p_\lambda(v_x)\right\rangle\left\langle Y_{\cF_{\alpha_0-\lambda}}(v_i',-z^{-1})  o(v)p_{\alpha_0-\lambda}(v_x), v_{\lambda}\right\rangle\nonumber\\
  & =\sum_{i\in I} z^{-2h_\lambda}\left\langle v_{\alpha_0-\lambda}, p_\lambda\left(Y_{\cF_x}(v_i,-1)u_n v_x\right)\right\rangle\left\langle p_{\alpha_0-\lambda}(Y_{\cF_x}(v_i',-z^{-1})o(v)v_x), v_\lambda\right\rangle.
 \end{align*}
Now for each $i\in I$, the projection of $Y_{\cF_x}(v_i,-1)u_n v_x$ to the degree-$0$ space $\cF_x(0)$ has the form $q_i^{(1)}(x)v_x$ for some $q_i^{(1)}(x)\in\CC[x]$ depending on $v_i$, $u$, and $n$, so
\begin{equation*}
 \left\langle v_{\alpha_0-\lambda}, p_\lambda\left(Y_{\cF_x}(v_i,-1)u_n v_x\right)\right\rangle=q_i^{(1)}(\lambda).
\end{equation*}
Similarly,
\begin{align*}
 \left\langle p_{\alpha_0-\lambda}(Y_{\cF_x}(v_i',-z^{-1})o(v)v_x), v_\lambda\right\rangle &=\sum_{k\in\ZZ} (-1)^{k+1}\left\langle p_{\alpha_0-\lambda}((v_i')_k o(v)v_x),v_\lambda\right\rangle z^{k+1}\nonumber\\
 & =(-1)^{\mathrm{wt}\,v_i'}\left\langle p_{\alpha_0-\lambda}(o(v_i')o(v)v_x),v_\lambda\right\rangle z^{\mathrm{wt}\,v_i'}\nonumber\\
 &= q^{(2)}_i(\alpha_0-\lambda)\,z^{\mathrm{wt}\,v_i'}
\end{align*}
where $(-1)^{\mathrm{wt}\,v_i'}o(v_i')o(v)v_x=q^{(2)}_i(x)v_x$. Thus \eqref{eqn:poly_coeff} holds with
\begin{equation*}
 q_m(\lambda)=\sum_{\mathrm{wt}\,v_i=m} q^{(1)}_i(\lambda)q^{(2)}_i(\alpha_0-\lambda)
\end{equation*}
for $m\in\NN$ (since $\mathrm{wt}\,v_i=\mathrm{wt}\,v_i'$ for all $i\in I$).
\end{proof}

We now consider the generic Fock module $\cF_x$ as an $\cM(p)$-module by restriction; recall the $C_1$-quotient $\cF_x/C_1(\cF_x)$ where
\begin{equation*}
 C_1(\cF_x)=\mathrm{span}\left\lbrace v_{-1} w\,\,\vert\,\,w\in\cF_x, v\in\cM(p), \mathrm{wt}\,v\geq 1\right\rbrace.
\end{equation*}
The natural $\NN$-grading on $\cF_x$ restricts to a grading on $C_1(\cF_x)$, so each homogeneous space $[\cF_x/C_1(\cF_x)](n)=\cF_x(n)/[C_1(\cF_x)](n)$, $n\in\NN$, is a finitely-generated $\CC[x]$-module. Although $\cF_x$ is certainly not a $C_1$-cofinite $\NN$-gradable weak $\cM(p)$-module, we do have:
\begin{lem}
 For any sufficiently large $n\in\NN$, there is a non-zero polynomial $d_n(x)\in\CC[x]$ such that $d_n(x)\cdot\cF_x(n)\subseteq [C_1(\cF_x)](n)$. That is, $[\cF_x/C_1(\cF_x)](n)$ is a torsion $\CC[x]$-module.
\end{lem}

\begin{proof}
 Fix any $\lambda\in\CC$. Since $\cF_\lambda$ is a $C_1$-cofinite $\cM(p)$-module by \cite[Theorem 13]{CMR}, $\cF_\lambda(n)=[C_1(\cF_\lambda)](n)$ for $n\in\NN$ sufficiently large. For such $n$ and for any $w\in\cF_x(n)$,
 \begin{equation*}
  p_\lambda(w)=\sum v_{-1}^{(i)} w^{(i)} =\sum v_{-1}^{(i)}p_\lambda(\til{w}^{(i)})=\sum p_\lambda(v^{(i)}_{-1}\til{w}^{(i)})
 \end{equation*}
for suitable $v^{(i)}\in\cM(p)$, $w^{(i)}\in\cF_\lambda$, and $\til{w}^{(i)}\in\cF_x$ such that $p_\lambda(\til{w}^{(i)})=w^{(i)}$ (recall that $p_\lambda$ is surjective). Since $p_\lambda$ preserves the $\NN$-gradings of $\cF_x$ and $\cF_\lambda$, we may assume that each $v_{-1}^{(i)}\til{w}^{(i)}$ has degree $n$, and thus
\begin{equation}\label{eqn:Fxn_torsion}
 \cF_x(n)=[C_1(\cF_x)](n)+(\ker p_\lambda)(n)=[C_1(\cF_x)](n)+(x-\lambda)\cdot\cF_x(n)
\end{equation}
for $n$ sufficiently large, using Lemma \ref{lem:ker_p_lambda}.

Now as a finitely-generated $\CC[x]$-module, $\cF_x(n) =\sum_{i=1}^I \CC[x]\cdot w_i$ for certain $w_i\in\cF_x(n)$. By \eqref{eqn:Fxn_torsion}, each generator $w_i$ satisfies
\begin{equation*}
 w_i = c_i+(x-\lambda)\sum_{j=1}^I p_{ij}(x)\cdot w_j
\end{equation*}
for suitable $c_i\in[C_1(\cF_x)](n)$ and $p_{ij}(x)\in\CC[x]$. Equivalently,
\begin{equation*}
\left[ \begin{array}{cccc}
        1-(x-\lambda)p_{11}(x) & -(x-\lambda)p_{12}(x) & \cdots & -(x-\lambda)p_{1I}(x)\\
        -(x-\lambda)p_{21}(x) & 1-(x-\lambda)p_{22}(x) & \cdots & -(x-\lambda)p_{2I}(x)\\
        \vdots & \vdots & \ddots & \vdots\\
        -(x-\lambda)p_{I1}(x) & -(x-\lambda)p_{I2}(x) & \cdots & 1-(x-\lambda)p_{II}(x)\\
       \end{array}
\right]\left[\begin{array}{c}
              w_1\\
              w_2\\
              \vdots\\
              w_I\\
             \end{array}
\right] =\left[\begin{array}{c}
                c_1\\
                c_2\\
                \vdots\\
                c_I\\
               \end{array}
\right].
\end{equation*}
Multiplying both sides by the adjugate of the matrix on the left and noting that $[C_1(\cF_x)](n)$ is a $\CC[x]$-submodule of $\cF_x(n)$ since $x$ commutes with $v_{-1}$ for any $v\in\cM(p)$, we get
\begin{equation*}
 d_n(x)\cdot w_i\in[C_1(\cF_x)](n)
\end{equation*}
for each $i$, where $d_n(x)$ is the determinant of the matrix. This determinant is not identically $0$ because $d_n(\lambda)=1$. Since $d_n(x)\cdot\cF_x(n)\subseteq[C_1(\cF_x)](n)$, this proves the lemma.
\end{proof}

We now fix non-zero polynomials $d_n(x)$ for all $n\in\NN$: for small $n$ we choose $d_n(x)=1$, and for all $n$ sufficiently large we choose $d_n(x)$ such that $d_n(x)\cdot\cF_x(n)\subseteq[C_1(\cF_x)](n)$. Then for all $N\in\NN$, set $p_N(x)=\prod_{n=0}^N d_n(x)$; by construction, these polynomials satisfy
\begin{equation*}
 p_N(x)\cdot\cF_x(n)\subseteq p_n(x)\cdot\cF_x(n)
\end{equation*}
whenever $n\leq N$. Finally, for $N\in\NN$, set $P_N(x)=\prod_{n=0}^N p_n(x)$; by construction, these polynomials have the property
\begin{equation}\label{eqn:p_n_polys}
 P_N(x)\cdot\cF_x(n)\subseteq P_{N-1}(x)\cdot[C_1(\cF_x)](n)
\end{equation}
whenever $n\leq N$ is sufficiently large.

We can now begin proving that series such as \eqref{eqn:intw_op_prod} satisfy suitable differential equations. Since $\cF_\lambda$ and $\cF_{\alpha_0-\lambda}$ are $C_1$-cofinite $\cM(p)$-modules, \cite[Theorems 1.4 and 2.3]{Hu-diff-eqn} already show that such series are solutions to differential equations with a regular singular point at $z=0$ and thus converge absolutely to multivalued analytic functions in $z$. However, we need these multivalued functions to be also analytic in $\lambda$, so here we adapt the methods of \cite{Hu-diff-eqn} using generic Fock modules to show that the coefficients of the differential equations may be taken to be analytic in $\lambda$. To shorten the discussion, we will follow a somewhat different exposition than \cite{Hu-diff-eqn} and derive the existence of the differential equations and the regularity of the singular point $z=0$ simultaneously.

Similar to \cite{Hu-diff-eqn}, let $R=\CC[z, (1-z)^{-1}]$ be the (Noetherian) ring of suitable rational functions in $z$. Then take three generic Fock modules $\cF_{x_1}$, $\cF_{x_2}$, $\cF_{x_3}$ and consider
\begin{equation*}
 T=R\otimes\cF_{x_1}\otimes\cF_{x_2}\otimes\cF_{x_3},
\end{equation*}
which is an $R[x_1,x_2,x_3]$-module in the obvious way. The $\NN$-gradings of the generic Fock modules induce an $\NN$-grading of $T$:
\begin{equation}\label{eqn:grading_of_T}
 T(n)=\bigoplus_{n_1+n_2+n_3=n} R\otimes\cF_{x_1}(n_1)\otimes\cF_{x_2}(n_2)\otimes\cF_{x_3}(n_3).
\end{equation}
As in \cite{Hu-diff-eqn}, we take the quotient of $T$ by certain intertwining-operator-inspired relations. Our relations are simpler than in \cite{Hu-diff-eqn} because in series such as \eqref{eqn:intw_op_prod}, the leftmost insertion is always the lowest-conformal-weight vector $v_{\alpha_0-\lambda}$. So for $v\in\cM(p)$ with $\mathrm{wt}\,v>0$ and homogeneous $w_1\in\cF_{x_1}$, $w_2\in\cF_{x_2}$, and $w_3\in\cF_{x_3}$, we define
\begin{align*}
 \cA(v,w_1,w_2,w_3)  = 1 & \otimes v_{-1} w_1\otimes w_2\otimes w_3\nonumber\\
 &-\sum_{k\geq 0} \left((1-z)^{-k-1}\otimes w_1\otimes v_k w_2\otimes w_3+1\otimes w_1\otimes w_2\otimes v_k w_3\right),
\end{align*}
\begin{align*}
 \mathcal{B}(v,&\,w_1, w_2,w_3)  = z^{\mathrm{wt}\,v+\mathrm{deg}\,w_3} \otimes w_1\otimes v_{-1} w_2\otimes w_3\nonumber\\
 &+\sum_{k\geq 0} z^{\mathrm{wt}\,v+\mathrm{deg}\,w_3}\left((-1)^k (1-z)^{-k-1}\otimes v_k w_1\otimes w_2\otimes w_3- z^{-k-1}\otimes w_1\otimes w_2\otimes v_k w_3\right),
\end{align*}
\begin{align*}
 \cC(v,w_1,w_2,w_3) & =  z^{\mathrm{wt}\,v+\mathrm{deg}\,w_2}  \otimes w_1\otimes  w_2\otimes v_{-1} w_3\nonumber\\
 &+\sum_{k\geq 0} (-1)^k z^{\mathrm{wt}\,v+\mathrm{deg}\,w_2}\left( 1\otimes v_k w_1\otimes w_2\otimes w_3
 + z^{-k-1}\otimes w_1\otimes v_k w_2\otimes  w_3\right).
\end{align*}
Since
\begin{equation*}
 \mathrm{deg}\,v_k w =\mathrm{wt}\,v+\mathrm{deg}\,w-k-1
\end{equation*}
for homogeneous $v\in\cM(p)$ and $w\in\cF_x$, the relations $\mathcal{B}(v,w_1,w_2,w_3)$ and $\cC(v,w_1,w_2,w_3)$ are indeed elements of $T$. 

Let $J$ be the $R[x_1,x_2,x_3]$-submodule of $T$ generated by $\cA(v,w_1,w_2,w_3)$, $\mathcal{B}(v,w_1,w_2,w_3)$, and $\cC(v,w_1,w_2,w_3)$ for all homogeneous $w_1\in\cF_{x_1}$, $w_2\in\cF_{x_2}$, $w_3\in\cF_{x_3}$, and $v\in\cM(p)$ such that $\mathrm{wt}\,v>0$. Somewhat differently from \cite{Hu-diff-eqn}, $T/J$ is not a finitely-generated $R[x_1,x_2,x_3]$-module; we need to take a submodule instead. Recall the polynomials $P_N(x)\in\CC[x]$ chosen above; we define $S$ to be the $R[x_1,x_2,x_3]$-submodule of $T/J$ generated by all
\begin{equation*}
 P_N(x_1)P_N(x_2)P_N(x_3)z^N\cdot w+J
\end{equation*}
for $N\in\NN$ and $w\in T(N)$. Similar to \cite[Corollary 1.2]{Hu-diff-eqn}, we have:
\begin{prop}
 The $R[x_1,x_2,x_3]$-module $S$ is finitely generated.
\end{prop}
\begin{proof}
 Since the homogeneous spaces of each generic Fock module are finitely-generated $\CC[x]$-modules, each homogeneous space $T(N)$ is a finitely-generated $R[x_1,x_2,x_3]$-module. Thus to prove the proposition, it is enough to show that when $N$ is sufficiently large, any generator $P_N(x_1)P_N(x_2)P_N(x_3)z^N\cdot w+J$ of $S$, where $w\in T(N)$, can be written as an $R[x_1,x_2,x_3]$-linear combination of generators $P_n(x_1)P_n(x_2)P_n(x_3)z^n\cdot\til{w}+J$ where $n<N$ and $\til{w}\in T(n)$. Indeed, \eqref{eqn:p_n_polys} and \eqref{eqn:grading_of_T} show that when $N$ is large enough and $w\in T(N)$,
 \begin{align*}
  P_N(x_1) P_N(x_2) & P_N(x_3)z^N\cdot w\in P_{N-1}(x_1)  P_{N-1}(x_2)P_{N-1}(x_3)z^N R\otimes\cdot\nonumber\\
  &\cdot\left(C_1(\cF_{x_1})\otimes\cF_{x_2}\otimes\cF_{x_3}+\cF_{x_1}\otimes C_1(\cF_{x_2})\otimes\cF_{x_3}+\cF_{x_1}\otimes\cF_{x_2}\otimes C_1(\cF_{x_3})\right).
 \end{align*}
Then because for any homogeneous $v\in\mathcal\cM(p)$, vertex operator degrees satisfy
\begin{equation*}
 \deg v_{-1} =\mathrm{wt}\,v>\mathrm{wt}\,v-k-1 =\deg v_k
\end{equation*}
for $k\geq 0$, the form of the generators $\cA(v,w_1,w_2,w_3)$, $\mathcal{B}(v,w_1,w_2,w_3)$, $\cC(v,w_1,w_2,w_3)$ of $J$ imply $P_N(x_1)P_N(x_2)P_N(x_3)z^N\cdot w+J$ is in the $R[x_1,x_2,x_3]$-submodule of $T/J$ generated by elements $P_n(x_1)P_n(x_2)P_n(x_3)z^n\cdot\til{w}+J$ for $n<N$ and $\til{w}\in T(n)$, as desired.
\end{proof}
Now similar to \cite[Corollary 1.3]{Hu-diff-eqn}, we get:
\begin{cor}
 For any homogeneous $w_1\in\cF_{x_1}$, $w_2\in\cF_{x_2}$, and $w_3\in\cF_{x_3}$, there exist $N\in\ZZ_+$ and elements $a_n(z;x_1,x_2,x_3)\in R[x_1,x_2,x_3]$ for $0\leq n\leq N-1$ such that
 \begin{align}\label{eqn:L(-1)_power_reln}
  P_{N+\sigma}(x_1)P_{N+\sigma}(x_2)&P_{N+\sigma}( x_3) \cdot(z^{N+\sigma} \otimes w_1\otimes L(-1)^N w_2\otimes w_3)+J\nonumber\\ &=\sum_{n=0}^{N-1}  a_n(z;x_1,x_2,x_3)\cdot(z^{n+\sigma}\otimes w_1\otimes L(-1)^n w_2\otimes w_3)+J,
 \end{align}
 where $\sigma=\mathrm{deg}\,w_1+\mathrm{deg}\,w_2+\mathrm{deg}\,w_3$.
\end{cor}
\begin{proof}
 Since $R[x_1,x_2,x_3]$ is a Noetherian ring by the Hilbert Basis Theorem, any submodule of the finitely-generated $R[x_1,x_2,x_3]$-module $S$ is finitely generated. In particular, the submodule generated by
 \begin{equation}\label{eqn:big_gen_set}
  \lbrace P_{n+\sigma}(x_1)P_{n+\sigma}(x_2)P_{n+\sigma}(x_3)\cdot(z^{n+\sigma}\otimes w_1\otimes L(-1)^n w_2\otimes w_3)+J\rbrace_{n\in\NN}
 \end{equation}
has a finite generating set. As each of these finitely many generators is a finite $R[x_1,x_2,x_3]$-linear combination of elements from the generating set \eqref{eqn:big_gen_set}, we may take the generating set to be finitely many of the elements in \eqref{eqn:big_gen_set}. Consequently, \eqref{eqn:L(-1)_power_reln} holds for $N$ sufficiently large (where we have absorbed the factors $P_{n+\sigma}(x_1)P_{n+\sigma}(x_2)P_{n+\sigma}(x_3)$ for $n<N$ into the elements $a_n(z;x_1,x_2,x_3)$). 
\end{proof}

We can use the preceding corollary to obtain differential equations for products of intertwining operators. Thus suppose we have families of intertwining operators $\cY_1^\lambda$ and $\cY_2^\lambda$ for $\lambda\in\CC$ of types $\binom{\cF_\lambda}{\cF_\lambda\,W}$ and $\binom{W}{\cF_{\alpha_0-\lambda}\,\cF_\lambda}$, respectively. Similar to the proof of \cite[Theorem 1.4]{Hu-diff-eqn}, there is for each $\lambda\in\CC$ a linear map
\begin{align*}
 \phi_{\cY_1^\lambda,\cY_2^\lambda}: T & \rightarrow \CC[\log z]\lbrace z\rbrace\nonumber\\
  \frac{f(z)}{(1-z)^m}\otimes w_1\otimes w_2\otimes w_3 &\mapsto \sum_{k\geq 0} (-1)^k\binom{-m}{k} z^k f(z)\cdot\nonumber\\
  &\hspace{5em}\cdot\langle v_{\alpha_0-\lambda},\cY_1^\lambda(p_\lambda(w_1),1)\cY_2^\lambda(p_{\alpha_0-\lambda}(w_2),z)p_\lambda(w_3)\rangle,
\end{align*}
where $f(z)\in\CC[z]$ and $m\in\NN$. As in \cite{Hu-diff-eqn}, the $R[x_1,x_2,x_3]$-submodule $J$ is contained in the kernel of $\phi_{\cY_1^\lambda,\cY_2^\lambda}$. This follows from the fact that $v_{\alpha_0-\lambda}$ is a lowest-conformal-weight vector of $\cF_{\alpha_0-\lambda}$ together with the Jacobi identity commutator formula \eqref{eqn:intw_op_comm} and the $n=-1$ case of the iterate formula \eqref{eqn:intw_op_iterate}. Thus $\phi_{\cY_1^\lambda,\cY_2^\lambda}$ descends to a well-defined linear map on $T/J$, which then restricts to a map $S\rightarrow\CC[\log z]\lbrace z\rbrace$. If we apply this map to the relation \eqref{eqn:L(-1)_power_reln}, multiply by $(1-z)^{r(\lambda)}$ for some function $r(\lambda)$, use the $L(-1)$-derivative property for intertwining operators and the product rule, and then divide by $z^\sigma$, we get:
\begin{thm}\label{thm:prod_diff_eqn}
 For any $w_1\in\cF_{x_1}$, $w_2\in\cF_{x_2}$, and $w_3\in\cF_{x_3}$, there exist $N\in\ZZ_+$ and elements $a_n(z;x_1,x_2,x_3)\in R[x_1,x_2,x_3]$ for $0\leq n\leq N-1$ such that for any families $\lbrace\cY_1^\lambda\rbrace_{\lambda\in\CC}$ and $\lbrace\cY_2^\lambda\rbrace_{\lambda\in\CC}$ of $\cM(p)$-module intertwining operators of types $\binom{\cF_\lambda}{\cF_\lambda\,W}$ and $\binom{W}{\cF_{\alpha_0-\lambda}\,\cF_\lambda}$, respectively, and for any analytic function $r(\lambda)$, the series
 \begin{equation*}
 (1-z)^{r(\lambda)}\langle v_{\alpha_0-\lambda},\cY_1^\lambda(p_\lambda(w_1),1)\cY_2^\lambda(p_{\alpha_0-\lambda}(w_2),z)p_\lambda(w_3)\rangle\in\CC[\log z]\lbrace z\rbrace
 \end{equation*}
 is a formal solution to the differential equation
 \begin{align*}
  P_N(\lambda)P_N(\alpha_0  -\lambda) P_N(\lambda) & z^N\left(\frac{d}{dz}+\frac{r(\lambda)}{1-z}\right)^N\varphi(\lambda,z)  \nonumber\\
  & =\sum_{n=0}^{N-1} a_n(z;\lambda,\alpha_0-\lambda,\lambda)z^n \left(\frac{d}{dz}+\frac{r(\lambda)}{1-z}\right)^n\varphi(\lambda,z).
 \end{align*}
\end{thm}

Since all the coefficient rational functions $a_n(z;\lambda,\alpha_0-\lambda,\lambda)$ in this theorem are analytic at $z=0$, the differential equation has a regular singular point at $z=0$. We will apply the theorem to $\cY_1^\lambda=\Omega_0(Y_{\cF_\lambda})$ and $\cY_2^\lambda=\cE_\lambda$ to prove Theorem \ref{thm:prod_analytic_modified}; we will also use the following result from the theory of ordinary differential equations (see for example \cite[Appendix A]{McR-rat} for a proof):
\begin{thm}\label{thm:diff_eqn_thm_1}
 Consider a regular-singular-point differential equation with parameter $\lambda$,
 \begin{equation*}
  z^N\dfrac{d^N\varphi}{dz^N} =\sum_{n=0}^{N-1} a_n(\lambda,z) z^n\dfrac{d^n\varphi}{dz^n},
 \end{equation*}
where the coefficient functions $a_n(\lambda,z)$ are analytic on $U\times B_1(0)$ for $U$ a non-empty open set of $\lambda\in\CC$ and $B_1(0)$ the open ball of radius $1$ centered at $z=0$. Suppose moreover that
\begin{equation*}
 \bigg\lbrace\varphi(\lambda,z)=\sum_{m\geq 0} q_m(\lambda)\,z^{h(\lambda)+m}\bigg\rbrace_{\lambda\in U}
\end{equation*}
is a family of formal series which solve the differential equation for each $\lambda\in U$, where $h(\lambda)$ and $q_m(\lambda)$ are analytic on $U$. Then for each $\lambda\in U$, the series $\varphi(\lambda,z)$ converges absolutely for each $z\in B_1(0)\setminus\lbrace 0\rbrace$, and
\begin{equation*}
 \varphi(\lambda,e^{\log z})= e^{(\log z)h(\lambda)}\sum_{m\geq 0} q_m(\lambda)\,z^m
\end{equation*}
defines a (single-valued) function which is analytic in both $\lambda$ and $z$ on $U\times(B_1(0)\setminus(-1,0])$.
\end{thm}

We can now complete the proof of Theorem \ref{thm:prod_analytic_modified}:
\begin{proof}
For any homogeneous $u,v\in\cM(p)$, $n\in\ZZ$, and analytic function $r(\lambda)$ on $\CC$, we take $w_1=u_n v_x$, $w_2=o(v) v_x$, and $w_3=v_x$ in Theorem \ref{thm:prod_diff_eqn} to conclude that
\begin{equation}\label{eqn:intw_op_prod_2}
(1-z)^{r(\lambda)}\big\langle v_{\alpha_0-\lambda}, \Omega_0(Y_{\cF_\lambda})(u_n v_\lambda,1)\cE_\lambda(o(v)v_{\alpha_0-\lambda},e^{\log z})v_\lambda\big\rangle
\end{equation}
is a formal solution to a differential equation of the form \eqref{eqn:diff_eqn_form} whose coefficient functions are analytic for $z\in\CC\setminus\lbrace 0,1\rbrace$ and $\lambda\in\CC\setminus S$ where $S$ is the finite set of roots of $P_N(\lambda)P_N(\alpha_0-\lambda)$.
Moreover, this differential equation has a regular singular point at $z=0$, and Proposition \ref{prop:poly_coeff} shows that \eqref{eqn:intw_op_prod_2} satisfies the conditions of Theorem \ref{thm:diff_eqn_thm_1} with $h(\lambda)=-2h_\lambda$ and $U=\CC\setminus S$. Thus \eqref{eqn:intw_op_prod_2} converges absolutely to an analytic function in both $\lambda$ and $z$ on $(\CC\setminus S)\times(B_1(0)\setminus(-1,0])$.
\end{proof}

\subsection{Proof of Theorem \ref{thm:diff_eqn_thm_2}}\label{subsec:analytic_thm_proof}

Throughout this subsection, we use the notation $B_r(z_0)$ for the open ball of radius $r$ centered at $z_0\in\CC$. Recall the setting of Theorem \ref{thm:diff_eqn_thm_2}: we have a series solution $\varphi(\lambda,\zeta)=\sum_{m=1}^M\sum_{k=0}^K f_{m,k}(\lambda,\zeta)e^{h_m\log\zeta}(\log\zeta)^k$ to a differential equation
\begin{equation}\label{eqn:diff_eqn_zeta_2}
\frac{d^N\varphi}{d\zeta^N} =\sum_{n=0}^{N-1} b_n(\lambda,\zeta)\dfrac{d^n\varphi}{d\zeta^n}
\end{equation}
whose coefficient functions $b_n(\lambda,\zeta)$ are analytic in $\lambda$ and $\zeta$ on $U\times(B_1(0)\setminus\lbrace 0\rbrace)$, where $U$ is a non-empty open subset of $\CC$. Moreover:
\begin{itemize}
 \item The $h_m\in\CC$ for $1\leq m\leq M$ are pairwise non-congruent mod $\ZZ$.
 \item For each $m$ and $k$, $f_{m,k}(\lambda,\zeta)=\sum_{n\in\ZZ} f_{m,k,n}(\lambda)\,\zeta^n$ is a Laurent series in $\zeta$ whose coefficients are functions of $\lambda$ defined on $U$.
 \item The series $\varphi(\lambda,\zeta)$ converges absolutely on $U\times(B_1(0)\setminus(-1,0])$, and thus for any $\lambda\in U$, $\varphi(\lambda,\zeta)$ is analytic in $\zeta$ on $B_1(0)\setminus(-1,0]$.
 \item For some non-empty open subset $V_0\subseteq B_1(0)\setminus(-1,0]$, $\varphi(\lambda,\zeta)$ is analytic in both $\lambda$ and $\zeta$ on $U\times V_0$.
\end{itemize}
Our goal is to show that for $1\leq m\leq M$, $0\leq k\leq K$, and $n\in\ZZ$, the coefficient functions $f_{m,k,n}(\lambda)$  are analytic in $\lambda$ on $U$. 
 
 First, since the series $\varphi(\lambda,\zeta)$ is absolutely convergent for $\lambda\in U$ and $\zeta\in B_1(0)\setminus\lbrace 0\rbrace$, so are the Laurent series $f_{m,k}(\lambda,\zeta)=\sum_{n\in\ZZ} f_{m,k,n}(\lambda)\,\zeta^n$. This means that for any simply-connected open subset $V\subseteq B_1(0)\setminus\lbrace 0\rbrace$ and any single-valued branch of logarithm $\ell(\zeta)$ defined on $V$, the series
\begin{equation*}
 \sum_{m=1}^M\sum_{k=0}^K f_{m,k}(\lambda,\zeta)\,e^{h_m\ell(\zeta)}\ell(\zeta)^k
\end{equation*}
also converges absolutely for all $\lambda\in U$ to a function that is analytic in $\zeta$ on $V$. We now show that this new series is also analytic in $\lambda$:
  
\begin{lem}\label{lem:analytic_continuation}
 For any simply-connected open subset $V\subseteq B_1(0)\setminus\lbrace 0\rbrace$ and any single-valued branch of logarithm $\ell(\zeta)$ defined on $V$, the function $\sum_{m=1}^M\sum_{k=0}^K f_{m,k}(\lambda,\zeta)\,e^{h_m\ell(\zeta)}\ell(\zeta)^k$ is analytic in both $\lambda$ and $\zeta$ on $U\times V$.
\end{lem}
\begin{proof}
Fix any $\zeta_1\in V$; we need to show that $\varphi_V(\lambda,\zeta)=\sum_{m=1}^M\sum_{k=0}^K f_{m,k}(\lambda,\zeta)\,e^{h_m\ell(\zeta)}\ell(\zeta)^k$ is analytic in both $\lambda$ and $\zeta$ for $\lambda\in U$ and $\zeta$ contained in an open neighborhood of $\zeta_1$. Recall we are assuming that $\varphi(\lambda,\zeta)=\sum_{m=1}^M\sum_{k=0}^K f_{m,k}(\lambda,\zeta)\,e^{h_m\log\zeta}(\log\zeta)^k$ is analytic in both $\lambda$ and $\zeta$ on $U\times V_0$ for some non-empty open set $V_0\subseteq B_1(0)\setminus(-1,0]$. Then if we fix $\zeta_0\in V_0$, we can obtain $\varphi_V(\lambda,\zeta)$ on $V$ (for any $\lambda\in U$) by analytic continuation of $\varphi(\lambda,\zeta)$ along some continuous path $\gamma: [0,1]\rightarrow B_1(0)\setminus\lbrace 0\rbrace$ such that $\gamma(0)=\zeta_0$ and $\gamma(1)=\zeta_1$.

We can cover the image of the path $\gamma$ with finitely many overlapping open disks $B_r(\gamma(t_i))$, $0\leq i\leq I$, as follows: First take $r>0$ to be no larger than the minimum distance from the image of $\gamma$ to the compact set $(\overline{B_1(0)}\setminus B_1(0))\cup\lbrace 0\rbrace$, so that $B_r(\gamma(t))\subseteq B_1(0)\setminus\lbrace 0\rbrace$ for all $t\in[0,1]$. We then take $t_0=0$ so that our first disk is $B_r(\zeta_0)$; for convenience, we may assume $r$ is small enough so that $B_r(\zeta_0)\subseteq V_0$. Now assuming we have chosen $t_i$ for some $i\geq 0$, we choose $t_{i+1}$ (if it exists) to be the minimum element of $[t_i,1]$  such that $\vert\gamma(t_{i+1})-\gamma(t_i)\vert\geq \frac{r}{2}$. This process terminates after finitely many steps: Since $\gamma$ is (uniformly) continuous on the compact set $[0,1]$, there is a $\delta>0$ such that $\vert\gamma(t)-\gamma(\til{t})\vert<\frac{r}{2}$ if $\vert t-\til{t}\vert<\delta$; consequently, $t_{i+1}-t_i>\delta$ for all $i$ such that $t_{i+1}$ exists. At the end of the process, we have $\zeta_1=\gamma(1)\in B_{r/2}(\gamma(t_I))$ by construction.

We now prove by induction on $i$ that the analytic continuation of $\varphi(\lambda,\zeta)$ along the path $\gamma$ is analytic in both $\lambda$ and $\zeta$ on $U\times B_r(\gamma(t_i))$. Since $\zeta_1\in B_r(\gamma(t_I))$, this will prove the lemma. For $i=0$, the analyticity claim holds by assumption since we are assuming $B_r(\gamma(t_0))=B_r(\zeta_0)\subseteq V_0$. If we now assume the claim holds for some $i\geq 0$, then the analytic continuation to $B_r(\gamma(t_{i+1}))$, which we denote as $\varphi_{i+1}(\lambda,\zeta)$, is analytic in both $\lambda$ and $\zeta$ on $U\times(B_r(\gamma(t_i))\cap B_r(\gamma(t_{i+1})))$. Since $\gamma(t_{i+1})\in B_r(\gamma(t_i))\cap B_r(\gamma(t_{i+1}))$ by construction, we can thus expand
\begin{equation}\label{eqn:phi_power_series}
 \varphi_{i+1}(\lambda,\zeta)=\sum_{n\geq 0} \varphi_{i+1,n}(\lambda)(\zeta-\gamma(t_{i+1}))^n
\end{equation}
for all $\lambda\in U$ and $\zeta\in B_\varepsilon(\gamma(t_{i+1}))$ for $\varepsilon>0$ sufficiently small, where the functions $\varphi_{i+1,n}(\lambda)$ can be expanded as a power series about any $\lambda\in U$ and thus are analytic functions on $U$.

We also know by induction that $\varphi_{i+1}(\lambda,\zeta)$ is a solution to the differential equation \eqref{eqn:diff_eqn_zeta_2} on $U\times B_r(\gamma(t_{i+1}))$. Since $\gamma(t_{i+1})$ is a regular point of the differential equation, since the coefficient functions $b_n(\lambda,\zeta)$ in \eqref{eqn:diff_eqn_zeta_2} are analytic in both $\lambda$ and $\zeta$ on $U\times B_r(\gamma(t_{i+1}))$, and since the coefficient functions $\varphi_{i+1,n}(\lambda)$ are analytic on $U$, it follows that the power series \eqref{eqn:phi_power_series} converges absolutely to a solution of \eqref{eqn:diff_eqn_form} on the entire domain $U\times B_r(\gamma(t_{i+1}))$, and that moreover this solution is analytic in both $\lambda$ and $\zeta$ on $U\times B_r(\gamma(t_{i+1}))$ (recall the regular singular point generalization of this result in Theorem \ref{thm:diff_eqn_thm_1}). Thus $\varphi_{i+1}(\lambda,\zeta)$ is analytic in both $\lambda$ and $\zeta$, proving the inductive hypothesis and thus also the lemma.
\end{proof}

We will use the preceding lemma to show that the functions $f_{m,k}(\lambda,\zeta)$ are analytic in both $\lambda$ and $\zeta$. To do so, we fix for any simply-connected open subset of $V\subseteq B_1(0)\setminus\lbrace 0\rbrace$ an analytic single-valued branch of logarithm $\ell(\zeta)$ defined on $V$. Let $\cS(U,V)$ denote the set of all functions $\psi(\lambda,\zeta)$ such that:
\begin{itemize}
 \item $\psi(\lambda,\zeta)$ is analytic in both $\lambda$ and $\zeta$ on $U\times V$.
 \item $\psi(\lambda,\zeta)$ has the form
\begin{equation*}
 \psi(\lambda,\zeta)=\sum_{m=1}^M\sum_{k=0}^K g_{m,k}(\lambda,\zeta)e^{h_m\ell(\zeta)}\ell(\zeta)^k
\end{equation*}
for some $M\in\ZZ_+$ and $K\in\NN$, where the $h_m\in\CC$ are non-congruent mod $\ZZ$ and each $g_{m,k}(\lambda,\zeta)=\sum_{n\in\ZZ} g_{m,k,n}(\lambda)\zeta^n$ is an absolutely-convergent Laurent series for each $\lambda\in U$.

\item Every single-valued branch
\begin{equation*}
 \psi^{(n)}(\lambda,\zeta)=\sum_{m=1}^M\sum_{k=0}^K e^{2\pi i n h_m} g_{m,k}(\lambda,\zeta)e^{h_m\ell(\zeta)}(\ell(\zeta)+2\pi i n)^k
\end{equation*}
for $n\in\ZZ$ is also analytic in both $\lambda$ and $\zeta$ on $U\times V$.
\end{itemize}
It is clear that every $e^{h\ell(\zeta)}\CC[\ell(\zeta)]$-linear combination of functions in $\cS(U,V)$ is an element of $\cS(U,V)$, for any $h\in\CC$, and that if $\psi\in\cS(U,V)$, then $\psi^{(n)}\in\cS(U,V)$ as well for any $n\in\ZZ$. By Lemma \ref{lem:analytic_continuation}, our original series $\varphi(\lambda,\zeta)$ is a function in $\cS(U,V)$ for any simply-connected open set $V\subseteq B_1(0)\setminus\lbrace 0\rbrace$.

\begin{lem}\label{lem:coeff_analytic}
 For any function $\psi(\lambda,\zeta)=\sum_{m=1}^M\sum_{k=0}^K g_{m,k}(\lambda,\zeta)e^{h_m\ell(\zeta)}\ell(\zeta)^k\in\cS(U,V)$, the Laurent series $g_{m,k}(\lambda,\zeta)$ are analytic in both $\lambda$ and $\zeta$ on $U\times V$, for all $1\leq m\leq M$ and $0\leq k\leq K$.
\end{lem}
\begin{proof}
 We prove the lemma by induction on $K$. When $K=0$, and setting $g_m=g_{m,0}$ for $1\leq m\leq M$, we have
 \begin{equation*}
  \left[\begin{array}{cccc}
         1 & 1 & \cdots & 1\\
         e^{2\pi i h_1} & e^{2\pi i h_2} & \cdots & e^{2\pi i h_M}\\
         \vdots & \vdots & \ddots & \vdots \\
         e^{2\pi i (M-1) h_1} & e^{2\pi i (M-1) h_2} & \cdots & e^{2\pi i (M-1) h_M}\\
        \end{array}\right]
\left[\begin{array}{c}
       g_{1}(\lambda,\zeta)e^{h_1\ell(\zeta)}\\
       g_{2}(\lambda,\zeta)e^{h_2\ell(\zeta)}\\
       \vdots\\
       g_{M}(\lambda,\zeta)e^{h_M\ell(\zeta)}\\
      \end{array}\right]
=\left[\begin{array}{c}
       \psi(\lambda,\zeta)\\
       \psi^{(1)}(\lambda,\zeta)\\
       \vdots\\
       \psi^{(M-1)}(\lambda,\zeta)\\
      \end{array}\right]
 \end{equation*}
Because the $h_m$ are non-congruent mod $\ZZ$, the Vandermonde matrix is invertible and hence each $g_{m}(\lambda,\zeta)e^{h_m\ell(\zeta)}$ is a $\CC$-linear combination of functions which are analytic in both $\lambda$ and $\zeta$ on $U\times V$. Thus each $g_{m}(\lambda,\zeta)$ is also analytic on $U\times V$.

Now we prove the case $K\geq 1$ by induction on the maximum $\til{m}\leq M$ such that $g_{\til{m},K}(\lambda,\zeta)\neq 0$. We have
\begin{align*}
 \psi^{(1)}(\lambda,\zeta)-e^{2\pi i h_{\til{m}}}\psi(\lambda,\zeta) & =\sum_{m\neq \til{m}}\sum_{k=0}^K g_{m,k}(\lambda,\zeta)e^{h_m\ell(\zeta)}\left(e^{2\pi i h_m}(\ell(\zeta)+2\pi i)^k-e^{2\pi i h_{\til{m}}}\ell(\zeta)^k\right)\nonumber\\
 &\hspace{4em}+e^{2\pi i h_{\til{m}}}\sum_{k=0}^K g_{\til{m},k}(\lambda,\zeta)e^{h_{\til{m}}\ell(\zeta)}\left((\ell(\zeta)+2\pi i)^k-\ell(\zeta)^k\right)\nonumber\\
 & =\sum_{m=1}^M\sum_{k=0}^K g'_{m,k}(\lambda,\zeta) e^{h_m\ell(\zeta)}\ell(\zeta)^k,
\end{align*}
where $g'_{m,K}(\lambda,\zeta)=0$ for $m\geq \til{m}$ and
\begin{equation*}
 g'_{\til{m},K-1} =2\pi i K e^{2\pi i h_{\til{m}}} g_{\til{m},K}(\lambda,\zeta).
\end{equation*}
Since $\psi^{(1)}(\lambda,\zeta)-e^{2\pi i h_{\til{m}}}\psi(\lambda,\zeta)\in\cS(U,V)$, induction on $\til{m}$ implies that $g_{\til{m},K}(\lambda,\zeta)$ is analytic in both $\lambda$ and $\zeta$ on $U\times V$. Then also $g_{\til{m},K}(\lambda,\zeta)e^{h_{\til{m}}\ell(\zeta)}\ell(\zeta)^K\in\cS(U,V)$, so that
\begin{equation*}
 \psi(\lambda,\zeta)-g_{\til{m},K}(\lambda,\zeta)e^{h_{\til{m}}\ell(\zeta)}\ell(\zeta)^K\in\cS(U,V)
\end{equation*}
as well. Then by induction on $\til{m}$ again, this implies that $g_{m,k}(\lambda,\zeta)$ is analytic in both $\lambda$ and $\zeta$ on $U\times V$ for all $m$ and $k$, completing the induction. 

Note that this argument also works for the base case $\til{m}=1$ of the induction on $\til{m}$, since in this case  $K-1$ is the maximum power of $\ell(\zeta)$ in $\psi^{(1)}(\lambda,\zeta)-e^{2\pi i h_1}\psi(\lambda,\zeta)$, and then the inductive hypothesis for the induction on $K$ yields the desired analyticity of $g_{1,K}(\lambda,\zeta)$.
\end{proof}

The preceding lemma implies that the coefficient functions $f_{m,k}(\lambda,\zeta)$ of our original series $\varphi(\lambda,\zeta)$ are analytic in both $\lambda$ and $\zeta$ on $U\times(B_1(0)\setminus\lbrace 0\rbrace)$. Thus each $f_{m,k}$ has a Laurent series expansion
\begin{equation*}
 f_{m,k}(\lambda,\zeta)=\sum_{n,n'\in\ZZ} c_{m,k,n,n'}\,(\lambda-\lambda_0)^{n'}\zeta^n
\end{equation*}
about $(\lambda_0,0)$ for any $\lambda_0\in U$, where 
\begin{equation*}
 c_{m,k,n,n'}=\frac{1}{(2\pi i)^2}\oint_{\vert\zeta\vert=r}\oint_{\vert\lambda-\lambda_0\vert=r'} \zeta^{-n-1}(\lambda-\lambda_0)^{-n'-1} f_{m,k}(\lambda,\zeta)\,d\lambda\,d\zeta
\end{equation*}
for suitable $r,r'>0$. Since $f_{m,k}(\lambda,\zeta)$ is analytic in $\lambda$ at $\lambda_0$ for any $\zeta$ such that $\vert\zeta\vert=r$, we get $c_{m,k,n,n'}=0$ for $n'<0$. Thus by the uniqueness of Laurent series expansions, all coefficient functions $f_{m,k,n}(\lambda)$ for $n\in\ZZ$ have power series expansions about any $\lambda_0\in U$ and thus are analytic on $U$. This completes the proof of Theorem \ref{thm:diff_eqn_thm_2}.

\section{Application to cyclic orbifolds of the triplet algebras}

The full automorphism group of the triplet vertex operator algebra $\cW(p)$, $p>1$, is $PSL(2,\CC)$ \cite{ALM}, and thus the (conjugacy classes of) finite automorphism groups of $\cW(p)$ follow an ADE classification. For $m\in\ZZ_+$, the finite subgroup of $PSL(2,\CC)$ of type $A_m$ is $\ZZ/m\ZZ$, and the vertex operator algebra $\cW(p)^{A_m}$ is the corresponding cyclic orbifold subalgebra of $\cW(p)$. It is a simple current extension of $\cM(p)$:
\begin{equation}\label{eqn:Wp_Am_decomp}
\cW(p)^{A_m} = \bigoplus_{n \in \ZZ} \cM_{2mn+1, 1}.
\end{equation}
In \cite{ALM}, Adamovi\'{c}, Lin, and Milas proved that $\cW(p)^{A_m}$ is $C_2$-cofinite and they constructed $2pm^2$ distinct irreducible $\cW(p)^{A_m}$-modules, which they conjectured to be the full list of irreducible $\cW(p)^{A_m}$-modules \cite[Conjecture~4.10]{ALM}. They verified this conjecture for small values of $m$ and $p$ in \cite{ALM2}, and then in \cite{AM-log-mods-app}, Adamovi\'{c} and Milas reduced the conjecture to \cite[Conjecture~2.3]{AM-log-mods-app}, which amounts to the simple current property of the $\cM(p)$-modules $\cM_{r,1}$, $r\in\ZZ$. Thus the fusion rules in Lemma \ref{lem:simple_curr_tens} and \cite[Theorem 5.2.1(1)]{CMY-singlet} combined with \cite[Theorem 2.5]{AM-log-mods-app} already complete the classification of irreducible $\cW(p)^{A_m}$-modules. In this section, we will use the tensor category structure on $\cO^T_{\cM(p)}$ and the vertex operator algebra extension theory of \cite{CKM-exts, CMY-completions} to quickly rederive this classification of irreducible $\cW(p)^{A_m}$-modules. We will also describe the projective covers of all irreducible $\cW(p)^{A_m}$-modules, compute all fusion rules involving irreducible $\cW(p)^{A_m}$-modules, and establish rigidity and non-degeneracy of the braided tensor category of $\cW(p)^{A_m}$-modules.

Before studying the representation theory of $\cW(p)^{A_m}$ in more detail, we  recall the direct limit completions of vertex tensor categories studied in \cite{CMY-completions}. For any vertex operator algebra $V$ and category $\cC$ of grading-restricted generalized $V$-modules, the direct limit completion, or $\mathrm{Ind}$-category, of $\cC$ is defined to be the category $\ind(\cC)$ of generalized $V$-modules (typically with infinite-dimensional conformal weight spaces) whose objects are the unions of their $\cC$-submodules. Equivalently, a generalized $V$-module $X$ is an object of $\ind(\cC)$ if and only if every vector $b\in X$ generates a $V$-submodule which is an object of $\cC$. The main Theorem 1.1 of \cite{CMY-completions} states that $\ind(\cC)$ is a vertex algebraic braided tensor category (with structure as given in \cite{HLZ1}-\cite{HLZ8}) under the following conditions:
\begin{itemize}
 \item The category $\cC$ is closed under submodules, quotients, and finite direct sums, and every module in $\cC$ is finitely generated.
 
 \item The vertex operator algebra $V$ is an object of $\cC$, and $\cC$ admits the vertex algebraic braided tensor category structure of \cite{HLZ1}-\cite{HLZ8}.
 
 \item For any intertwining operator $\cY$ of type $\binom{X}{W_1\,W_2}$ where $W_1$, $W_2$ are objects of $\cC$ and $X$ is an object of $\ind(\cC)$, the image $\im\cY\subseteq X$ is an object of $\cC$.
\end{itemize}
For the third condition above, recall that the image of an intertwining operator $\cY$ of type $\binom{X}{W_1\,W_2}$ is the submodule of $X$ spanned by coefficients of powers of $x$ and $\log x$ in $\cY(w_1,x)w_2$ for $w_1\in W_1$, $w_2\in W_2$. In the case $V=\cM(p)$, we have:
\begin{prop}\label{prop:ind_OT_tens_cat}
 The direct limit completions $\ind(\cO_{\cM(p)})$ and $\ind(\cO_{\cM(p)}^T)$ both admit the vertex algebraic braided tensor category structure of \cite{HLZ8}.
\end{prop}
\begin{proof}
For $\ind(\cO_{\cM(p)})$, the result follows from \cite[Theorem 7.1]{CMY-completions}, which states that for any vertex operator algebra $V$, the category of $C_1$-cofinite grading-restricted generalized $V$-modules satisfies the conditions of \cite[Theorem 1.1]{CMY-completions} if it is closed under contragredient modules. For $\ind(\cO_{\cM(p)}^T)$, the intertwining operator condition is the only one left to check. If $\cY$ is an intertwining operator of type $\binom{X}{W_1\,W_2}$ where $W_1$, $W_2$ are objects of $\cO^T_{\cM(p)}$ and $X$ is an object of $\ind(\cO^T_{\cM(p)})$, then $\im\cY$ is an object of $\cO_{\cM(p)}$ since the category of $C_1$-cofinite $\cM(p)$-modules and its direct limit completion satisfy the conditions of \cite[Theorem 1.1]{CMY-completions}. Thus by the universal property of tensor products in $\cO_{\cM(p)}$, $\im\cY$ is a quotient of $W_1\tens W_2$. But since $W_1\tens W_2$ is also the tensor product of $W_1$ and $W_2$ in $\cO_{\cM(p)}^T$, and since $\cO_{\cM(p)}^T$ is closed under quotients, $\im\cY$ is also an object of $\cO_{\cM(p)}^T$.
\end{proof}

We also have:
\begin{prop}\label{prop:gen_Mp_mod_in_DLC}
 Any grading-restricted generalized $\cM(p)$-module is an object of the direct limit completion $\ind(\cO_{\cM(p)})$.
\end{prop}
\begin{proof}
 Let $M$ be a grading-restricted generalized $\cM(p)$-module. The grading-restriction conditions imply that
 \begin{equation*}
  M=\bigoplus_{\mu\in\CC/\ZZ}\bigoplus_{n\in\NN} M_{[h_\mu+n]},
 \end{equation*}
where for any coset $\mu\in\CC/\ZZ$, $h_\mu\in\mu$ is chosen so that $M_{[h_\mu-n]}=0$ for $n\in\ZZ_+$. Moreover, each $M_{[h_\mu+n]}$ is finite dimensional. To show that $M$ is an object of $\ind(\cO_{\cM(p)})$, it is enough to show that each (grading-restricted) submodule $M_\mu=\bigoplus_{n\in\NN} M_{h_\mu+n}$ is an object of $\ind(\cO_{\cM(p)})$ (since $\ind(\cO_{\cM(p)})$ is closed under arbitrary direct sums).

Pick an irreducible $A(\cM(p))$-submodule of the (finite-dimensional) lowest conformal weight space of $M_\mu$; it generates an $\cM(p)$-submodule $M_1\subseteq M_\mu$. Then pick an irreducible $A(\cM(p))$-submodule of the lowest conformal weight space of $M_\mu/M_1$; it generates an $\cM(p)$-submodule $M_2/M_1\subseteq M_\mu/M_1$. Continuing in this manner, we obtain a filtration
\begin{equation*}
0 \subseteq M_1 \subseteq M_2 \subseteq \cdots \subseteq M_\mu,
\end{equation*}
where each $M_i/M_{i-1}$ is a homomorphic image of a generalized Verma $\cM(p)$-module. 

Since Theorems \ref{thm:typical_GVM} and \ref{thm:atypical_GVM} show that all generalized Verma $\cM(p)$-modules have finite length, each $M_i$ has finite length and thus is an object of $\cO_{\cM(p)}$. Further, finite-dimensionality of the weight spaces of $M_\mu$ imply that $M_\mu=\cup_{i=1}^\infty M_i$, so each $M_\mu$, and thus also $M=\bigoplus_{\mu\in\CC/\ZZ} M_\mu$, is the union of its $\cO_{\cM(p)}$-submodules. Thus $M$ is an object of $\ind(\cO_{\cM(p)})$.
\end{proof}

Now by \eqref{eqn:Wp_Am_decomp}, the vertex operator algebra $\cW(p)^{A_m}$ restricts to an $\cM(p)$-module in $\ind(\cO_{\cM(p)})$ (and also in $\ind(\cO_{\cM(p)}^T)$). Thus by \cite[Theorem 3.2]{HKL} (or more precisely \cite[Theorem 7.5]{CMY-completions}), $\cW(p)^{A_m}$ is a commutative algebra in the braided tensor category $\ind(\cO_{\cM(p)})$ (or $\ind(\cO_{\cM(p)}^T)$). We use $\rep \cW(p)^{A_m}$ to denote the tensor category of (possibly non-local) $\cW(p)^{A_m}$-modules (as in \cite{KO, HKL, CKM-exts, CMY-completions}) which restrict to generalized $\cM(p)$-modules in $\ind(\cO_{\cM(p)})$. Then $\rep^0 \cW(p)^{A_m}$ is the braided tensor category of (local) generalized $\cW(p)^{A_m}$-modules in $\ind(\cO_{\cM(p)})$. Let $\cC_{\cW(p)^{A_m}}$ be the category of all grading-restricted generalized $\cW(p)^{A_m}$-modules; it is a braided tensor category by \cite{Hu-C2}. Since all objects of $\cC_{\cW(p)^{A_m}}$ are also grading-restricted generalized $\cM(p)$-modules, Proposition \ref{prop:gen_Mp_mod_in_DLC} shows that $\cC_{\cW(p)^{A_m}}$ is a subcategory of $\rep^0\cW(p)^{A_m}$; indeed it is a braided tensor subcategory by \cite[Theorem 3.65]{CKM-exts} (or \cite[Theorem 7.7]{CMY-completions}).

 Let $\cF_{\cW(p)^{A_m}}: \cO_{\cM(p)} \rightarrow \rep \cW(p)^{A_m}$ be the tensor functor of induction, defined on objects by $\cF_{\cW(p)^{A_m}}(M)=\cW(p)^{A_m}\tens M$ (where $\tens$ denotes the tensor product on $\ind(\cO_{\cM(p)})$) and on morphisms by $\cF_{\cW(p)^{A_m}}(f)=\Id_{\cW(p)^{A_m}}\tens f$. Induction is exact since $\cO_{\cM(p)}$ is rigid (see for example the proof of \cite[Theorem 3.2.4]{CMY-singlet}). Moreover, $\cF_{\cW(p)^{A_m}}$ maps simple objects in $\cO_{\cM(p)}$ to simple objects in $\rep \cW(p)^{A_m}$ by \cite[Proposition 4.4]{CKM-exts} (which applies because tensoring with $\cM_{2mn+1,1}$, $n\neq 0$, does not fix any simple object in $\cO_{\cM(p)}$). Moreover, the argument of \cite[Proposition 5.0.4]{CMY3} shows that every simple object of $\rep\cW(p)^{A_m}$ is isomorphic to the induction of a simple $\cM(p)$-module in $\cO_{\cM(p)}$, and that $\cF_{\cW(p)^{A_m}}(M_1)\cong\cF_{\cW(p)^{A_m}}(M_2)$ for simple modules $M_1$ and $M_2$ if and only if $M_2\cong\cM_{2mn+1,1}\tens M_1$ for some $n\in\ZZ$. This discussion shows that we can use induction to classify all irreducible $\cW(p)^{A_m}$-modules (see also \cite[Theorem~2.5]{AM-log-mods-app}):
\begin{thm}\label{thm:exhaust-simple}
The category $\cC_{\cW(p)^{A_m}}$ of grading-restricted generalized $\cW(p)^{A_m}$-modules has precisely $2pm^2$ distinct simple objects, given by 
\begin{equation*}
\cW_{\overline{r},s}:=\cF_{\cW(p)^{A_m}}(\cM_{r,s}),\qquad\overline{r}=r+2m\ZZ\in\ZZ/2m\ZZ,\,\, 1\leq s\leq p
\end{equation*}
and 
\begin{equation*}
\cV_{\lambda+mL}:=\cF_{\cW(p)^{A_m}}(\cF_{\lambda}),\qquad \lambda+mL\in\left(\frac{1}{m}L^\circ\setminus L^\circ\right)\bigg/ mL.
\end{equation*}
\end{thm}
\begin{proof}
Any simple generalized $\cW(p)^{A_m}$-module in $\rep^0\cW(p)^{A_m}$ is necessarily grading-restricted since $\cW(p)^{A_m}$ is $C_2$-cofinite (see \cite[Corollary 5.7]{ABD}).  Thus it is enough to determine all simple objects of $\rep^0\cW(p)^{A_m}$, and for this it is enough to determine which irreducible $\cM(p)$-modules induce to local $\cW(p)^{A_m}$-modules. As in the discussion preceding Lemma \ref{lem:untwisted_induction}, $\cF_{\cW(p)^{A_m}}(M)$ is local if and only if $\cR^2_{\cM_{2m+1,1},M}=\Id_{\cM_{2m+1,1}\tens M}$. For simplicity of notation, we use $\cM_\lambda$ for any $\lambda\in\CC$ to denote the irreducible $\cM(p)$-socle of $\cF_\lambda$. Then similar calculations as in the proof of Proposition \ref{prop:irred_mod_grading} show that
\begin{align*}
\cR_{\cM_{2m+1,1},\cM_\lambda}^2 & =\theta_{\cM_{2m+1,1}\tens\cM_\lambda}\circ(\theta_{\cM_{2m+1,1}}^{-1}\tens\theta_{\cM_\lambda}^{-1})\nonumber\\
& =e^{-2\pi i\alpha_{2m+1,1}\lambda}\Id_{\cM_{2m+1,1}\tens\cM_\lambda} =e^{2\pi i m\alpha_+\lambda}\Id_{\cM_{\alpha_{2m+1,1}}\tens\cM_\lambda},
\end{align*}
so $\cF_{\cW(p)^{A_m}}(\cM_\lambda)$ is local if and only if $\lambda\in\frac{1}{m} L^\circ$.

For $\lambda_1,\lambda_2\in\frac{1}{m} L^\circ$, we also have $\cF_{\cW(p)^{A_m}}(\cM_{\lambda_1})\cong\cF_{\cW(p)^{A_m}}(\cM_{\lambda_2})$ if and only if
\begin{equation*}
\cM_{\lambda_2}\cong\cM_{2mn+1,1}\tens\cM_{\lambda_1}\cong\cM_{\lambda_1+\alpha_{2mn+1,1}}
\end{equation*}
for some $n\in\ZZ$. Since
\begin{equation*}
\lbrace\alpha_{2mn+1,1}\,\vert\,n\in\ZZ\rbrace = -m\alpha_+\ZZ=mL,
\end{equation*}
we see that $\cW(p)^{A_m}$ has precisely $2pm^2$ distinct irreducible modules parametrized by $\frac{1}{m}L^\circ/mL\cong\frac{1}{m}\ZZ/2pm\ZZ$. Moreover, the modules in the statement of the proposition give a complete list of isomorphism class representatives.
\end{proof}

By \cite{Hu-C2}, every simple module in $\cC_{\cW(p)^{A_m}}$ has a projective cover. To determine these projective modules, we treat $\cW(p)^{A_m}$ as a commutative algebra in $\ind(\cO_{\cM(p)}^T)$. First:
\begin{prop}
Every generalized $\cW(p)^{A_m}$-module in $\rep^0\cW(p)^{A_m}$ restricts to an $\cM(p)$-module in $\ind(\cO_{\cM(p)}^T)$.
\end{prop}
\begin{proof}
Let $X$ be a generalized $\cW(p)^{A_m}$-module in $\rep^0\cW(p)^{A_m}$. By definition, $X$ is the union of its $\cO_{\cM(p)}$-submodules, so we just need to show that any finite-length $\cM(p)$-submodule $M\subseteq X$ is an object of $\cO_{\cM(p)}^T$.  We may assume that $M$ is indecomposable, in which case we need to show that $\cR_{\cM_{2,1},M}^2$ is a scalar multiple of the identity.

We first claim that $\cR_{\cM_{2m+1,1},M}^2=\Id_{\cM_{2m+1,1}\tens M}$. To prove this, let $i: M\hookrightarrow X$ and $j: \cM_{2m+1,1}\hookrightarrow \cW(p)^{A_m}$ denote the inclusions, and let $\mu_X: \cW(p)^{A_m}\tens X\rightarrow X$ denote the morphism induced by the vertex operator $Y_X: A\otimes X\rightarrow X((x))$ and the universal property of the tensor product in $\ind(\cO_{\cM(p)})$. By naturality of the monodromy isomorphisms in $\ind(\cO_{\cM(p)})$, the diagram
 \begin{equation*}
 \xymatrixcolsep{5pc}
 \xymatrixrowsep{2.5pc}
  \xymatrix{
  \cM_{2m+1,1}\tens M \ar[d]^{\cR_{\cM_{2m+1,1},M}^2} \ar[r]^{\Id_{\cM_{2m+1,1}}\tens i} & \cM_{2m+1,1}\tens X \ar[d]^{\cR_{\cM_{2m+1,1},X}^2} \ar[r]^{j\tens\Id_X} & \cW(p)^{A_m}\tens X \ar[d]^{\cR^2_{\cW(p)^{A_m},X}}\\
  \cM_{2m+1,1}\tens M \ar[r]^{\Id_{\cM_{2m+1,1}}\tens i} \ar[rd] & \cM_{2m+1,1}\tens X \ar[d]^(.4){\mu_X\vert_{\cM_{2m+1,1}\tens X}} \ar[r]^{j\tens\Id_X} & \cW(p)^{A_m}\tens X \ar[ld]^{\mu_X}\\
  & X & \\
  }
 \end{equation*}
commutes, with $\Id_{\cM_{2m+1,1}}\tens i$ injective by the exactness of $\cM_{2m+1,1}\tens\bullet$ and $j\tens\Id_X$ injective because $\cM_{2m+1,1}$ is a direct summand of $\cW(p)^{A_m}$. Now, $\mu_X\circ\cR_{\cW(p)^{A_m},X}^2=\mu_X$ by the definition of $\rep^0 \cW(p)^{A_m}$ (as given in \cite{KO,CKM-exts}, for example), so
\begin{equation*}
 \mu_X\vert_{\cM_{2m+1,1}\tens X}\circ(\Id_{\cM_{2m+1,1}}\tens i)\circ\cR_{\cM_{2m+1,1},M}^2=\mu_X\circ(j\tens i)=\mu_X\vert_{\cM_{2m+1,1}\tens X}\circ(\Id_{\cM_{2m+1,1}}\tens i).
\end{equation*}
Because $\Id_{\cM_{2m+1,1}}\tens i$ is injective, it is enough to show $\mu_X\vert_{\cM_{2m+1,1}\tens X}$ is injective as well. In fact $\mu_X\vert_{\cM_{2m+1,1}\tens X}$ is an isomorphism with inverse
\begin{align*}
 X\xrightarrow{l_X^{-1}} & \cM_{1,1}\tens X\xrightarrow{(\mu_X\vert_{\cM_{2m+1,1}\tens \cM_{-2m+1,1}})^{-1}\tens\Id_X} (\cM_{2m+1,1}\tens \cM_{-2m+1,1})\tens X\nonumber\\
 &\xrightarrow{\cA_{\cM_{2m+1,1},\cM_{-2m+1,1},X}^{-1}} \cM_{2m+1,1}\tens(\cM_{-2m+1,1}\tens X)\nonumber\\
 &\xrightarrow{\Id_{\cM_{2m+1,1}}\tens\mu_X\vert_{\cM_{-2m+1,1}\tens X}} \cM_{2m+1,1}\tens X,
\end{align*}
since $\cM_{2m+1,1}$ is a simple current and the multiplication $\mu_X$ is associative. This proves the claim.

Now we consider $\cR^2_{\cM_{2,1},M}$. Recall the open Hopf link map $h_M$ defined in the proof of Theorem \ref{thm:OT_properties}, as well as the standard open Hopf link $\Phi_{\bullet,M}$ discussed in Remark \ref{rem:open_Hopf_link}. Since $\Phi_{\bullet,M}$ defines a ring homomorphism from the Grothendieck ring of $\cO_{\cM(p)}$ to $\mathrm{End}_{\cM(p)}\,M$ (see for example the graphical proof in \cite[Section 3.1.3]{CG}), the relation between $h_M$ and $\Phi_{\cM_{2,1},M}$ from Remark \ref{rem:open_Hopf_link} combined with $\cR_{\cM_{2m+1,1},M}^2=\Id_{\cM_{2m+1,1}\tens M}$ implies
\begin{equation*}
h_M^{2m}=\frac{\Phi_{\cM_{2,1},M}^{2m}}{(\dim_{\cM(p)} \cM_{2,1})^{2m}} =\frac{\Phi_{\cM_{2m+1,1},M}}{\dim_{\cM(p)}\cM_{2m+1,1}} =\Id_M.
\end{equation*}
Thus $h_M$ has finite order on all finite-dimensional conformal weight spaces of $M$, which means that $h_M$ is diagonalizable on $M$ with $2m$th roots of unity as eigenvalues. Assuming as we may that $M$ is indecomposable, $h_M=e^{\pi i n/m}\Id_M$ for some $n\in\lbrace 0,1,\ldots, 2m-1\rbrace$. It is now immediate from the definition of $h_M$ that
\begin{equation*}
\Id_{\cM_{0,1}}\tens\cR_{\cM_{2,1},M}^2 = e^{\pi i n/m}\Id_{\cM_{0,1}\tens(\cM_{2,1}\tens M)}.
\end{equation*}
Then by naturality of the unit and associativity isomorphisms, 
\begin{equation*}
\cR_{\cM_{2,1},M}^2 =F\circ(\Id_{\cM_{2,1}}\tens(\Id_{\cM_{0,1}}\tens\cR_{\cM_{2,1},M}^2))\circ F^{-1} = e^{\pi i n/m}\Id_{\cM_{2,1}\tens M}
\end{equation*} 
where $F$ is the composition
\begin{align*}
\cM_{2,1}\tens(\cM_{0,1}\tens(\cM_{2,1}\tens M))  &\xrightarrow{\cA_{\cM_{2,1},\cM_{0,1},\cM_{2,1}\tens M}} (\cM_{2,1}\tens\cM_{0,1})\tens(\cM_{2,1},\tens M) \nonumber\\
&\xrightarrow{\til{e}\tens\Id_{\cM_{2,1}\tens M}} \cM_{1,1}\tens(\cM_{2,1}\tens M)\xrightarrow{l_{\cM_{2,1}\tens M}} \cM_{2,1}\tens M
\end{align*}
and $\til{e}: \cM_{2,1}\tens\cM_{0,1}\rightarrow\cM_{1,1}$ is any isomorphism. This proves the proposition.
\end{proof}

By the preceding proposition, $\rep^0\cW(p)^{A_m}$ is precisely the braided tensor category of generalized $\cW(p)^{A_m}$-modules which restrict to $\cM(p)$-modules in $\cO_{\cM(p)}^T$. Since $\cC_{\cW(p)^{A_m}}$ is a braided tensor subcategory, we can now identify the projective objects in $\cC_{\cW(p)^{A_m}}$ as the inductions of projective objects in $\cO_{\cM(p)}^T$:
\begin{thm}
For $\lambda+mL\in\left(\frac{1}{m}L^\circ\setminus L^\circ\right)\big/ mL$ and $\overline{r}\in\ZZ/2m\ZZ$, the irreducible $\cW(p)^{A_m}$-modules $\cV_{\lambda+mL}$ and $\cW_{\overline{r},p}$ are projective in $\cC_{\cW(p)^{A_m}}$. For $\overline{r}=r+2m\ZZ\in\ZZ/2m\ZZ$ and $1\leq s\leq p-1$, the irreducible $\cW(p)^{A_m}$-module $\cW_{\overline{r},s}$ has a projective cover $\mathcal{R}_{\overline{r},s}:=\cF_{\cW(p)^{A_m}}(\cP_{r,s})$ with Loewy diagram
\begin{equation*}
\begin{matrix}
  \begin{tikzpicture}[->,>=latex,scale=1.5]
\node (b1) at (1,0) {$\cW_{\overline{r},s}$};
\node (c1) at (-1, 1){$\mathcal{R}_{\overline{r},s}$:};
   \node (a1) at (0,1) {$\cW_{\overline{r-1},p-s}$};
   \node (b2) at (2,1) {$\cW_{\overline{r+1},p-s}$};
    \node (a2) at (1,2) {$\cW_{\overline{r},s}$};
\draw[] (b1) -- node[left] {} (a1);
   \draw[] (b1) -- node[left] {} (b2);
    \draw[] (a1) -- node[left] {} (a2);
    \draw[] (b2) -- node[left] {} (a2);
\end{tikzpicture}
\end{matrix} .
\end{equation*}
\end{thm}
\begin{proof}
For notational simplicity, let $\cP_\lambda$ for $\lambda\in\CC$ denote the projective cover in $\cO_{\cM(p)}^T$ of the irreducible $\cM(p)$-module $\cM_\lambda\subseteq\cF_\lambda$. Since $\cO_{\cM(p)}^T$ is generated as a tensor category by its simple objects, \cite[Theorem~1.4(1)]{CKL} implies that the induced module $\cF_{\cW(p)^{A_m}}(\cP_\lambda)$ is local if and only if 
$\cF_{\cW(p)^{A_m}}(\cM_\lambda)$ is local. That is, $\cF_{\cW(p)^{A_m}}(\cP_\lambda)$ is an object of $\rep^0\cW(p)^{A_m}$ for $\lambda\in\frac{1}{m}L^\circ$, and then because $\cP_\lambda$ is projective in $\cO_{\cM(p)}^T$, the same argument as in \cite[Lemma 5.0.6]{CMY3} and \cite[Lemma 17]{ACKM} shows that $\cF_{\cW(p)^{A_m}}(\cP_\lambda)$ is projective in $\ind(\cO_{\cM(p)}^T)$ and then also in $\rep^0\cW(p)^{A_m}$.

The proof that $\cF_{\cW(p)^{A_m}}(\cP_\lambda)$ for $\lambda\in\frac{1}{m}L^\circ$ is a projective cover of $\cF_{\cW(p)^{A_m}}(\cM_\lambda)$ in $\rep^0\cW(p)^{A_m}$ is the same as the proof of \cite[Proposition 5.0.7]{CMY3}, so we omit it here. For $\lambda\in\frac{1}{m}L^\circ\setminus L^\circ$ or $\lambda=\alpha_{r,p}$, $r\in\ZZ$, we have $\cF_{\cW(p)^{A_m}}(\cP_\lambda)=\cF_{\cW(p)^{A_m}}(\cM_\lambda)$, so $\cF_{\cW(p)^{A_m}}(\cM_\lambda)$ is a simple projective object in the subcategory $\cC_{\cW(p)^{A_m}}$. For $\lambda=\alpha_{r,s}$ with $r\in\ZZ$, $1\leq s\leq p-1$, the Loewy diagram of $\mathcal{R}_{\overline{r},s}=\cF_{\cW(p)^{A_m}}(\cP_{r,s})$ can be derived using a similar argument as that in \cite[Theorem 7.9]{MY}, using the Loewy diagram of $\cP_{r,s}$ from \eqref{eqn:Prs_Loewy_diag}, exactness of $\cF_{\cW(p)^{A_m}}$, and Frobenius reciprocity. In particular, $\mathcal{R}_{\overline{r},s}$ has finite length and thus is a projective object in $\cC_{\cW(p)^{A_m}}$ as well as a projective cover of $\cW_{\overline{r},s}$.
\end{proof}

Since the induction functor $\cF_{\cW(p)^{A_m}}$ is monoidal, all tensor products of simple objects in $\cC_{\cW(p)^{A_m}}$ follow immediately from the $\cM(p)$-module fusion rules in \cite[Theorem 5.2.1(1)]{CMY-singlet}, Theorem \ref{thm:Mrs_Flambda}, Theorem \ref{thm:typ_typ_atyp_fusion}, and Theorem \ref{thm:typ_typ_typ_fusion}:
\begin{thm}
Tensor products of simple modules in $\cC_{\cW(p)^{A_m}}$ are as follows:
\begin{enumerate}
\item For $\overline{r},\overline{r'}\in\ZZ/2m\ZZ$ and $1\leq s,s'\leq p$,
\begin{equation*}
  \cW_{\overline{r},s}\tens \cW_{\overline{r'},s'}\cong\bigoplus_{\substack{\ell=\vert s-s'\vert+1\\ \ell+s+s'\equiv 1\,(\mathrm{mod}\,2)}}^{\mathrm{min}(s+s'-1, 2p-1-s-s')} \cW_{\overline{r+r'-1},\ell}\oplus\bigoplus_{\substack{\ell=2p+1-s-s'\\ \ell+s+s'\equiv 1\,(\mathrm{mod}\,2)}}^p \mathcal{R}_{\overline{r+r'-1},\ell},
 \end{equation*}
where sums are taken to be empty if the lower bound exceeds the upper bound, and we use the notation $\cR_{\overline{r},p}:=\cW_{\overline{r},p}$ for $\overline{r}\in\ZZ/2m\ZZ$.

\item For $\overline{r}\in\ZZ/2m\ZZ$, $1\leq s\leq p$, and $\lambda+mL\in\left(\frac{1}{m}L^\circ\setminus L^\circ\right)\big/ mL$,
\begin{equation*}
\cW_{\overline{r},s}\tens\cV_{\lambda+mL}\cong\bigoplus_{\ell=0}^{s-1} \cV_{\lambda+\alpha_{r,s}+\ell\alpha_-+mL}.
\end{equation*}

\item For $\lambda+mL,\mu+mL\in\left(\frac{1}{m}L^\circ\setminus L^\circ\right)\big/mL$ such that $\lambda+\mu\in\alpha_0+\alpha_{r,s}+mL$ for some $r\in\ZZ$ and $1\leq s\leq p$,
\begin{equation*}
  \cV_{\lambda+mL}\tens\cV_{\mu+mL}\cong\bigoplus_{\substack{\ell= s\\ \ell\equiv s\,\,(\mathrm{mod}\,2)\\}}^p \cR_{\overline{r},\ell}\oplus\bigoplus_{\substack{\ell=p+2-s\\\ell\equiv p-s\,\,(\mathrm{mod}\,2)\\}}^p \cR_{\overline{r-1},\ell}.
 \end{equation*}

\item For $\lambda+mL,\mu+mL\in\left(\frac{1}{m}L^\circ\setminus L^\circ\right)\big/mL$ such that $\lambda+\mu\notin L^\circ$,
\begin{equation*}
\cV_{\lambda+mL}\tens\cV_{\mu+mL} \cong \bigoplus_{\ell=0}^{p-1} \cV_{\lambda + \mu + \ell \alpha_-+mL}.
\end{equation*}
\end{enumerate}
\end{thm}

Finally, we establish the non-semisimple modularity of $\cC_{\cW(p)^{A_m}}$:
\begin{thm}
 The tensor category $\cC_{\cW(p)^{A_m}}$ of grading-restricted generalized $\cW(p)^{A_m}$-modules is rigid and thus also ribbon, and its braiding is non-degenerate. That is, $\cC_{\cW(p)^{A_m}}$ is a non-semisimple modular tensor category.
\end{thm}
\begin{proof}
By Theorem \ref{thm:exhaust-simple} and because induction maps rigid objects to rigid objects, all simple $\cW(p)^{A_m}$-modules are rigid. Rigidity of $\cC_{\cW(p)^{A_m}}$ then follows from \cite[Theorem~4.4.1]{CMY-singlet} since every object in $\cC_{\cW(p)^{A_m}}$ has finite length. Non-degeneracy of the braiding follows from \cite[Main Theorem 1]{McR-rat} (alternatively, we could prove this through direct calculation of monodromies using the classification of simple $\cW(p)^{A_m}$-modules and the balancing equation, as in the proof of the $m=1$ case in \cite[Theorem 4.7]{GN}).
\end{proof}


\begin{thebibliography}{CCFGH}

\bibitem[ABD]{ABD}
T. Abe, G. Buhl and C. Dong, Rationality, regularity, and $C_2$-cofiniteness, \textit{Trans. Amer. Math. Soc.} \textbf{356} (2004), no. 8, 3391--3402.



\bibitem[Ad1]{Ad}
D. Adamovi\'{c}, Classification of irreducible modules of certain subalgebras of free boson vertex algebra, \textit{J. Algebra} \textbf{270} (2003), no. 1, 115--132.


\bibitem[Ad2]{Ad3}
D. Adamovi\'{c},  A construction of admissible $A^{(1)}_1$-modules of level $-\frac{4}{3}$, \textit{J. Pure Appl. Algebra} \textbf{196} (2005), no. 2-3, 119--134. 

\bibitem[Ad3]{Ad-sl}
D.~Adamovi\'c,
Realizations of simple affine vertex algebras and their modules: The cases ${\widehat{sl(2)}}$ and ${\widehat{osp(1,2)}}$,
\textit{Comm. Math. Phys.} \textbf{366} (2019) no.3, 1025--1067.

\bibitem[ACG]{ACG}
D.~Adamovi\'{c}, T.~Creutzig and N.~Genra,
Relaxed and logarithmic modules of $\widehat{\mathfrak{sl}_3}$, arXiv:2110.15203.

\bibitem[ACGY]{ACGY}
D. Adamovi\'{c}, T.~Creutzig, N.~Genra and J.~Yang,
The vertex algebras $\mathcal R^{(p)}$ and $\mathcal V^{(p)}$, \textit{Comm. Math. Phys.} \textbf{383} (2021), no. 2, 1207--1241.

\bibitem[ACPV]{ACPV}
D.~Adamovi\'{c}, T.~Creutzig, O.~Per\v{s}e and I.~Vukorepa,
Tensor category $KL_k(\mathfrak{sl}_{2n})$ via minimal affine $W$-algebras at the non-admissible level $k =-\frac{2n+1}{2}$,
arXiv:2212.00704.

\bibitem[ALM1]{ALM}
D. Adamovi\'{c}, X. Lin and A. Milas, $ADE$ subalgebras of the triplet vertex algebra $\cW(p)$: $A$-series, \textit{Commun. Contemp. Math.} \textbf{15} (2013), no. 6, 1350028, 30 pp.

\bibitem[ALM2]{ALM-Dm}
D. Adamovi\'{c}, X. Lin and A. Milas, $ADE$ subalgebras of the triplet vertex algebra $\cW(p)$: $D$-series, \textit{Internat. J. Math.} \textbf{25} (2014), no. 1, 1450001, 34 pp. 

\bibitem[ALM3]{ALM2}
D. Adamovi\'{c}, X. Lin and A. Milas, Vertex algebras $\cW(p)^{A_m}$ and $\cW(p)^{D_m}$ and constant term identities, \textit{SIGMA Symmetry Integrability Geom. Methods Appl.} \textbf{11} (2015), Paper 019, 16 pp. 

\bibitem[AM1]{AM-trip}
D. Adamovi\'{c} and A. Milas, On the triplet vertex algebra $\cW(p)$, {\em Adv. Math.} \textbf{217} (2008), no. 6, 2664--2699.




\bibitem[AM2]{AM-log-mods}
D. Adamovi\'{c} and A. Milas, Lattice construction of logarithmic modules for certain vertex
algebras, \textit{Selecta Math. (N.S.)} \textbf{15} (2009), no. 4, 535--561.

\bibitem[AM3]{AM-doub}
D. Adamovi\'{c} and A. Milas, The doublet vertex operator superalgebras $\cA(p)$ and $\cA_{2,p}$, \textit{Recent Developments in Algebraic and Combinatorial Aspects of Representation Theory}, 23--38, Contemp. Math., \textbf{602}, Amer. Math. Soc., Providence, RI, 2013.

\bibitem[AM4]{AM-log-mods-app}
D. Adamovi\'{c} and A. Milas, Some applications and constructions of intertwining operators in logarithmic conformal field theory. \textit{Lie Algebras, Vertex Operator Algebras, and Related Topics}, 15--27, Contemp. Math., {\bf 695}, Amer. Math. Soc., Providence, RI, 2017.



\bibitem[AW]{AW}
R.~Allen and S.~Wood,
Bosonic ghostbusting -- The bosonic ghost vertex algebra admits a logarithmic module category with rigid fusion, \textit{Comm. Math. Phys.} \textbf{390} (2022), no. 2, 959--1015. 

\bibitem[ACKR]{ACKM}
J. Auger, T. Creutzig, S. Kanade and M. Rupert, Braided tensor categories related to $\mathcal{B}_p$ vertex algebras, \textit{Comm. Math. Phys.} \textbf{378} (2020), no. 1, 219--260.


\bibitem[CF]{CF}
N.~Carqueville and M.~Flohr, Nonmeromorphic operator product expansion and $C_2$-cofiniteness for a family of
$\cW$-algebras, \textit{J. Phys. A} \textbf{39} (2006), no. 4, 951--966.


\bibitem[CCFGH]{CCFGH}
M.~Cheng, S.~Chun, F.~Ferrari, S.~Gukov and S.~Harrison,
3d Modularity,  \textit{J. High Energy Phys.} 2019, no. 10, 010, 93 pp.

\bibitem[CGP1]{CGP}
 F.~Costantino,  N.~Geer,  B.~Patureau-Mirand, Quantum invariants of 3-manifolds via link surgery presentations and non-semi-simple categories, \textit{J. Topol.} \textbf{7} (2014), no. 4, 1005--1053. 

\bibitem[CGP2]{CGP2}
 F.~Costantino,  N.~Geer,  B.~Patureau-Mirand, Some remarks on the unrolled quantum group of $\mathfrak{sl}(2)$, \textit{J. Pure Appl. Algebra} \textbf{219} (2015), no. 8, 3238--3262.

\bibitem[Cr]{C}
T. Creutzig W-algebras for Argyres-Douglas theories. \textit{Eur. J. Math.} \textbf{3} (2017), no. 3, 659--690.

\bibitem[CDGG]{CDGG}
T.~Creutzig, T.~Dimofte, N.~Garner and N.~Geer,
A QFT for non-semisimple TQFT, arXiv:2112.01559.



\bibitem[CGR]{CGR}
T. Creutzig, A. Gainutdinov and I. Runkel, A quasi-Hopf algebra for the triplet vertex operator algebra, {\em Commun. Contemp. Math.} \textbf{22} (2020), no. 3, 1950024, 71 pp.

\bibitem[CG]{CG}
T. Creutzig and T. Gannon, Logarithmic conformal field theory, log-modular tensor categories and modular forms, \textit{J. Phys. A} \textbf{50} (2017), no. 40, 404004, 37 pp.

\bibitem[CGN]{CGN}
T.~Creutzig, N.~Genra and S.~Nakatsuka,
Duality of subregular $W$-algebras and principal $W$-superalgebras,
\textit{Adv. Math.} \textbf{383} (2021), Paper No. 107685, 52 pp.

\bibitem[CGNS]{CGNS}
T.~Creutzig, N.~Genra, S.~Nakatsuka and R.~Sato,
Correspondences of categories for subregular $W$-algebras and principal $W$-superalgebras, \textit{Comm. Math. Phys.} \textbf{393} (2022), no. 1, 1--60. 



 \bibitem[CJORY]{CJORY}
T. Creutzig, C. Jiang, F. Orosz Hunziker, D. Ridout and J. Yang, Tensor categories arising from the Virasoro algebra, \textit{Adv. Math.} \textbf{380} (2021), 107601, 35 pp.

\bibitem[CKL]{CKL}
 T. Creutzig, S. Kanade and A. Linshaw, Simple current extensions beyond semi-simplicity, \textit{Commun. Contemp. Math.} \textbf{22} (2020), no. 1, 1950001, 49 pp.

\bibitem[CKLR]{CKLR}
T. Creutzig, S. Kanade, A. Linshaw and D. Ridout, Schur-Weyl duality for Heisenberg cosets, {\em Transform. Groups} \textbf{24} (2019), no. 2, 301--354.

\bibitem[CKM]{CKM-exts}
T. Creutzig, S. Kanade and R. McRae, Tensor categories for vertex operator superalgebra extensions, to appear in \textit{Mem. Amer. Math. Soc.}, arXiv:1705.05017.


\bibitem[CLR]{CLM}
T.~Creutzig, S.~Lentner and M.~Rupert,
Characterizing braided tensor categories associated to logarithmic vertex operator algebras, arXiv:2104.13262.

\bibitem[CL]{CL}
T.~Creutzig and A.~Linshaw,
Trialities of $\mathcal{W}$-algebras, \textit{Camb. J. Math.} \textbf{10} (2022), no. 1, 69--194.

\bibitem[CMY1]{CMY-completions} 
T. Creutzig, R. McRae and J. Yang, Direct limit completions of vertex tensor categories, \textit{Commun. Contemp. Math.} \textbf{24} (2022), no. 2, Paper No. 2150033, 60 pp. 

\bibitem[CMY2]{CMY-singlet}
T. Creutzig, R. McRae and J. Yang, On ribbon categories for singlet vertex algebras, \textit{Comm. Math. Phys.} \textbf{387} (2021),  no. 2, 865--925.

\bibitem[CMY3]{CMY3}
T. Creutzig, R. McRae and J. Yang, Tensor structure on the Kazhdan-Lusztig category for affine $\mathfrak{gl}(1|1)$, \textit{Int. Math. Res. Not. IMRN} 2022, no. 16, 12462--12515.

\bibitem[CMY4]{CMY-Bp}
T. Creutzig, R. McRae and J. Yang, Rigid tensor structure on big module categories for some $W$-(super)algebras in type $A$, arXiv:2210.04678.

\bibitem[CM]{CM}
T.~Creutzig and A.~Milas,
False theta functions and the Verlinde formula,
\textit{Adv. Math.} \textbf{262} (2014), 520--545.



\bibitem[CMR]{CMR}
T.~Creutzig, A.~Milas and M.~Rupert, Logarithmic link invariants of $\overline{U}_q^H(\mathfrak{sl}_2)$ and asymptotic dimensions of singlet vertex algebras,
\textit{J. Pure Appl. Algebra} \textbf{222} (2018), no. 10, 3224--3247.

\bibitem[CRW]{CRW}
T.~Creutzig, D.~Ridout and S.~Wood, Coset constructions of logarithmic $(1,p)$ models, \textit{Lett. Math. Phys.} \textbf{104} (2014), no. 5, 553--583.

\bibitem[CY]{CY}
T. Creutzig and J. Yang, Tensor categories of affine Lie algebras beyond admissible levels, \textit{Math. Ann.} \textbf{380} (2021), no. 3-4, 1991--2040. 

\bibitem[DL]{DL}
C. Dong and J. Lepowsky, \textit{Generalized Vertex Algebras and Relative Vertex Operators},  Progress in Mathematics, \textbf{112}, Birkh\"{a}user Boston, Inc., Boston, MA, 1993. x+202 pp.

\bibitem[DLM]{DLM}
C. Dong, H. Li and G. Mason,  Vertex operator algebras and associative algebras, \textit{J. Algebra} \textbf{206} (1998), no. 1, 67--96.

\bibitem[FF]{FF}
B. Feigin and D. Fuchs, Representations of the Virasoro algebra, \textit{Representation of Lie Groups and Related Topics}, 465--554, Adv. Stud. Contemp. Math., \textbf{7}, Gordon and Breach, New York, 1990.

\bibitem[FGST1]{FGST1}
B.~Feigin, A.~Gainutdinov, A.~Semikhatov and I.~Tipunin,
Modular group representations and fusion in logarithmic conformal field theories and in the quantum group center, \textit{Comm. Math. Phys.} \textbf{265} (2006), no. 1, 47--93.

\bibitem[FGST2]{FGST2}
B.~Feigin, A.~Gainutdinov, A.~Semikhatov and I.~Tipunin, Logarithmic extensions of minimal models: characters and modular transformations, \textit{Nuclear Phys. B} \textbf{757} (2006), no. 3, 303--343.

\bibitem[FS]{FS}
B.~Feigin and A.~Semikhatov,
$\mathcal{W}^{(2)}_n$ algebras,
\textit{Nuclear Phys. B} \textbf{698} (2004), no. 3, 409--449.

\bibitem[FHL]{FHL}
I. Frenkel, Y.-Z. Huang and J. Lepowsky, On axiomatic approaches to vertex operator algebras
and modules, \textit{Mem. Amer. Math. Soc.} \textbf{104} (1993), no. 494, viii+64 pp.

\bibitem[FLM]{FLM}
I. Frenkel, J. Lepowsky and A. Meurman, \textit{Vertex Operator Algebras and the Monster}, Pure and Applied Mathematics, \textbf{134}, Academic Press, Inc., Boston, MA, 1988, liv+508 pp.



\bibitem[FZ1]{FZ1}
I.  Frenkel and Y. Zhu, Vertex operator algebras associated to representations of affine and Virasoro
algebras, {\em Duke Math. J.} {\bf 66} (1992), no. 1, 123--168.


\bibitem[FZ2]{FZ2}
I.~Frenkel and M.~Zhu, Vertex algebras associated to modified regular representations of the Virasoro algebra, {\em Adv. Math.} {\bf 229} (2012), no. 6, 3468--3507.

\bibitem[FHST]{FHST}
J.~Fuchs, S.~Hwang, A.~Semikhatov and I.~Tipunin, Nonsemisimple fusion algebras and the Verlinde formula,
\textit{Comm. Math. Phys.} \textbf{247} (2004), no. 3, 713--742.

\bibitem[GR]{GR}
M. Gaberdiel and I. Runkel, From boundary to bulk in logarithmic CFT, \textit{J. Phys. A} \textbf{41} (2008), no. 7, 075402, 29 pp.

\bibitem[GN]{GN}
T. Gannon and C. Negron, Quantum $SL(2)$ and logarithmic vertex operator algebras at $(p, 1)$-central charge, to appear in \textit{J. Eur. Math. Soc. (JEMS)}, arXiv:2104.12821.

\bibitem[Hu1]{Hu-diff-eqn}
Y.-Z. Huang, Differential equations and intertwining operators, \textit{Commun. Contemp. Math.} \textbf{7} (2005), no. 3, 375--400.

\bibitem[Hu2]{Hu-mod}
Y.-Z. Huang,  Rigidity and modularity of vertex tensor categories, \textit{Commun. Contemp. Math.} \textbf{10} (2008), suppl. 1, 871--911.

\bibitem[Hu3]{Hu-C2}
Y.-Z. Huang, Cofiniteness conditions, projective covers and the logarithmic tensor product theory, \textit{J. Pure Appl. Algebra}  \textbf{213}  (2009), no. 4, 458--475.



\bibitem[HKL]{HKL}
Y.-Z. Huang, A. Kirillov, Jr. and J. Lepowsky, Braided tensor categories and extensions of vertex operator algebras, \textit{Comm. Math. Phys.} \textbf{337} (2015), no. 3, 1143--1159.

\bibitem[HLZ1]{HLZ1}
Y.-Z. Huang, J. Lepowsky and L. Zhang, Logarithmic tensor category theory for generalized modules for a
conformal vertex algebra, I: Introduction and strongly graded algebras and their generalized modules, \textit{Conformal Field Theories and Tensor Categories}, 169--248, Math. Lect. Peking Univ., Springer, Heidelberg, 2014.
	
\bibitem[HLZ2]{HLZ2}
Y.-Z. Huang, J. Lepowsky and L. Zhang, Logarithmic tensor category theory for
generalized modules for a conformal vertex algebra, II: Logarithmic formal
calculus and properties of logarithmic intertwining operators, arXiv:1012.4196.
	
\bibitem[HLZ3]{HLZ3}
Y.-Z. Huang, J. Lepowsky and L. Zhang, Logarithmic tensor category theory for
generalized modules for a conformal vertex algebra, III: Intertwining maps and
tensor product bifunctors, arXiv:1012.4197.
	
\bibitem[HLZ4]{HLZ4}
Y.-Z. Huang, J. Lepowsky and L. Zhang, Logarithmic tensor category theory for
generalized modules for a conformal vertex algebra, IV: Constructions of tensor
product bifunctors and the compatibility conditions, arXiv:1012.4198.
	
\bibitem[HLZ5]{HLZ5}
Y.-Z. Huang, J. Lepowsky and L. Zhang, Logarithmic tensor category theory for
generalized modules for a conformal vertex algebra, V: Convergence condition
for intertwining maps and the corresponding compatibility condition,
arXiv:1012.4199.
	
\bibitem[HLZ6]{HLZ6}
Y.-Z. Huang, J. Lepowsky and L. Zhang, Logarithmic tensor category theory for
generalized modules for a conformal vertex algebra, VI: Expansion condition,
associativity of logarithmic intertwining operators, and the associativity
isomorphisms, arXiv:1012.4202.
	
\bibitem[HLZ7]{HLZ7}
Y.-Z. Huang, J. Lepowsky and L. Zhang, Logarithmic tensor category theory for
generalized modules for a conformal vertex algebra, VII: Convergence and
extension properties and applications to expansion for intertwining maps,
arXiv:1110.1929.
	
\bibitem[HLZ8]{HLZ8}
Y.-Z. Huang, J. Lepowsky and L. Zhang, Logarithmic tensor category theory for
generalized modules for a conformal vertex algebra, VIII: Braided tensor
category structure on categories of generalized modules for a conformal vertex
algebra, arXiv:1110.1931.


\bibitem[IK]{IK}
K. Iohara and Y. Koga, \textit{Representation Theory of the Virasoro Algebra}, Springer Monographs in Mathematics, Springer-Verlag London, Ltd., London, 2011, xviii+474 pp.

\bibitem[Ka]{Kausch:1990vg}
H.~Kausch,
Extended conformal algebras generated by a multiplet of primary fields,
\textit{Phys. Lett. B} \textbf{259} (1991), no. 4, 448--455. 

\bibitem[KR]{KR}
K.~Kawasetsu and D.~Ridout,
Relaxed highest-weight modules II: Classifications for affine vertex algebras, \textit{Commun. Contemp. Math.} \textbf{24} (2022), no. 5, Paper No. 2150037, 43 pp.

\bibitem[KO]{KO} 
A. Kirillov, Jr. and V. Ostrik, On a $q$-analogue of the McKay correspondence and the $ADE$ classification of $\mathfrak{sl}_2$ conformal field theories, \textit{Adv. Math.} \textbf{171} (2002), no. 2, 183--227.



\bibitem[LL]{LL}
J. Lepowsky and H. Li, \textit{Introduction to Vertex
Operator Algebras and Their Representations}, Progress in Mathematics, \textbf{227}, Birkh\"{a}user Boston, Inc., Boston, MA, 2004. xiv+318 pp.

\bibitem[Li]{Li-intw-ops}
H. Li, Determining fusion rules by $A(V)$-modules and bimodules, \textit{J.
Algebra}  \textbf{212}  (1999), no. 2,  515--556.


\bibitem[McR]{McR-rat}
R. McRae, On rationality for $C_2$-cofinite vertex operator algebras, arXiv:2108.01898.

\bibitem[MY]{MY}
R. McRae and J. Yang, Structure of Virasoro tensor categories at central charge $13-6p-6p^{-1}$ for integers $p>1$, arXiv:2011.02170.


\bibitem[Miy]{Mi}
M. Miyamoto, $C_1$-cofiniteness and fusion products of vertex operator algebras, \textit{Conformal Field Theories and Tensor Categories}, 271--279, Math. Lect. Peking Univ., Springer, Heidelberg, 2014.

\bibitem[NT]{NT}
K. Nagatomo and A. Tsuchiya, The triplet vertex operator algebra $W(p)$ and the restricted quantum group $\overline{U}_q(sl_2)$ at $q=e^{\frac{\pi i}{p}}$, {\em Exploring New Structures and Natural Constructions in Mathematical Physics}, 1--49, Adv. Stud. Pure Math., \textbf{61}, Math. Soc. Japan, Tokyo, 2011.

\bibitem[RS]{Rozansky:1992rx}
L.~Rozansky and H.~Saleur,
Quantum field theory for the multivariable Alexander-Conway polynomial,
\textit{Nuclear Phys. B} \textbf{376} (1992), no. 3, 461--509.

\bibitem[TW]{TW}
A. Tsuchiya and S. Wood, The tensor structure on the representation category of the $\cW_{p}$ triplet algebra, \textit{J. Phys. A} \textbf{46} (2013), no. 44, 445203, 40 pp.

\bibitem[Zh]{Zh}
Y. Zhu, Modular invariance of characters of vertex operator algebras, \textit{J. Amer. Math. Soc.} \textbf{9} (1996), no. 1, 237--302. 


\end{thebibliography}
\end{document}